%% file: hypermaps.tex
\documentclass[a4paper]{amsart}
\usepackage[margin=1.4 in]{geometry}
\usepackage{hyperref,color}

\input{defs.tex}
\graphicspath{{Figures/}}

\newcommand{\W}{W}

\author[O. Bernardi and \'E. Fusy]{Olivier Bernardi$^{*}$ \and \'{E}ric Fusy$^{\dagger}$}
\thanks{$^{*}$Department of Mathematics, Brandeis University, Waltham MA, USA,
bernardi@brandeis.edu. Supported by NSF grant DMS-1308441 and DMS-1800681.\\
$^{\dagger}$LIX, \'Ecole Polytechnique, Palaiseau, France, fusy@lix.polytechnique.fr.
Supported by the European project
ExploreMaps (ERC StG 208471), the ANR grant 
``Cartaplus'' 12-JS02-001-01, and the ANR grant ``EGOS'' 12-JS02-002-01.}

\title[Unified bijections for planar hypermaps]{Unified bijections for planar hypermaps with general cycle-length constraints}
\date{\today}

\begin{document}
\begin{abstract}
We present a general bijective approach to planar hypermaps with two main results. 
First we obtain unified bijections for classes of maps or hypermaps defined by face-degree constraints and girth constraints. 
To any such class we associate bijectively a class of plane trees characterized by local constraints. This unifies and greatly generalizes several bijections for maps and hypermaps. 
Second, we present yet another level of generalization of the bijective approach by considering classes of maps with non-uniform girth constraints. More precisely, we consider \emph{well-charged maps}, which are maps with an assignment of \emph{charges} (real numbers) to vertices and faces, with the constraints that the length of any cycle of the map is at least equal to the sum of the charges of the vertices and faces enclosed by the cycle. 
We obtain a bijection between charged hypermaps and a class of plane trees characterized by local constraints.
\end{abstract}

\maketitle

\section{Introduction}
A \emph{planar map} is an embedding of a connected planar graph in the sphere, considered up to orientation-preserving homeomorphism. A rich literature has been devoted to the enumerative combinatorics of planar maps by various approaches, such as Tutte's method~\cite{Tu63} based on generating function equations, the matrix integral method initiated by Br\'ezin et al. in~\cite{Bre}, and 
the bijective approach initiated by Cori and Vauquelin~\cite{CoriVa} and popularized by Schaeffer~\cite{Schaeffer:these}. 

Planar hypermaps are a natural generalization of planar maps. Precisely, a \emph{planar hypermap} is a planar map in which faces are colored in two colors, say that there are dark faces and light faces, in such a way that every edge separates a light face from a dark face. 
The dark faces of the hypermap play the role of embedded hyperedges, and as such, hypermaps 
 can be seen as embedded hypergraphs~\cite{CoMa92}, and classical
 maps (embedded graphs) identify to hypermaps in which every edge has been replaced
by a dark face of degree 2; see Figure~\ref{fig:maps-are-hypermaps}(a). 

Hypermaps have played a prominent role to tackle various problems: for instance an exact solution 
of the Ising model on random planar lattices has been obtained by a reduction to the enumeration 
of planar hypermaps with control on the face-degrees~\cite{BMSc02,BoKa}; and in a 
similar spirit different models of hard particles on random planar lattices
have been exactly solved~\cite{BMSc02,BDG07}. Hypermaps also encompass the notion
of constellations, which are a convenient visual encoding of factorizations
in the symmetric group~\cite{BMSc00,ZvLa97}. In particular, the famous Hurwitz numbers~\cite{ZvLa97,GoJa97,Ok03,DuPoSc} (which count factorizations into transpositions, or equivalently certain branched coverings of the sphere) are naturally encoded by certain constellations. 
Bijective methods have played a crucial role in all these enumerative problems related to hypermaps. 

\fig{width=12cm}{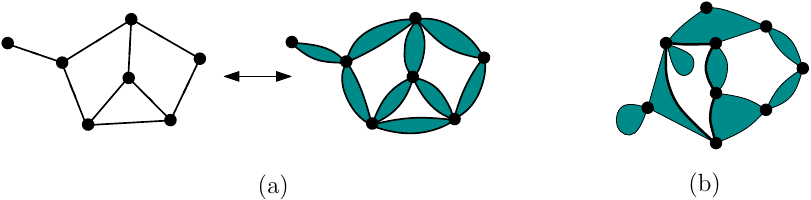}{(a) A map and the corresponding hypermap (obtained by replacing every edge by a dark face of degree 2). (b) A general hypermap (with dark faces of arbitrary degrees), of ingirth $4$ (due to the cycle indicated by bold lines).}

In this article, we present a unified bijective approach for planar hypermaps. Our results generalize the bijective approach for maps presented in \cite{BFbij,BFgir} in two ways: first we deal with the more general case of hypermaps, and second we consider more general cycle-length conditions via the new concept of \emph{charged maps}. This approach also unifies and greatly generalizes several known bijections for hypermaps together with several known bijections for maps. 
We will discuss in details the relation between our approach and previously known bijections below (see Figure \ref{fig:diagram}) and in Section \ref{section:recovering-bijections}. However, let us point out already that the bijections in \cite{BMSc02,BDG04,BDG07} are recovered as special cases of our framework. These have applications to solving several statistical mechanics models on maps: Ising model, hard particle model, forest model, and blocked edge model. It is our hope that the toolbox we establish in the present article will find many more applications in the realm of statistical mechanics.


Our strategy is similar to the one developed in \cite{BFbij}. Namely, we first establish a ``master bijection'' between a class of oriented hypermaps and a class of plane trees, that we call \emph{hypermobiles} (see Figure~\ref{fig:intro}(a) for an example). Then we specialize this master bijection to obtain our bijective results about classes of hypermaps defined in terms of face-degree conditions and girth conditions. This requires to exhibit canonical orientations characterizing these classes of maps, and then identifying the hypermobiles associated through the master bijection. To be precise, our canonical orientations and hypermobiles are actually \emph{weighted}, that is, each edge carries a weight in $\mathbb{R}$; see Figure~\ref{fig:intro}(a).
In \cite{BFbij} we relied on the concept of \emph{minimal $\al$-orientations}, that is, orientations such that the indegree at each vertex is fixed by a function $\alpha$, and containing no counterclockwise oriented cycle. 
We mention that Section \ref{sec:hyperflow} contains a generalizations of this framework to hypermaps which could be of independent interest.


In the first part of this article (Sections~\ref{sec:bij_plane}-\ref{section:recovering-bijections}) we establish the master bijection and we use it to obtain bijections for classes of hypermaps defined by \emph{ingirth} constraints. 
The ingirth for hypermaps is a generalization of the notion of girth for maps: it is defined as the smallest length of a cycle $C$ such that all faces adjacent to $C$ in the interior of $C$ are light (with the ``interior'' being defined with respect to a distinguished ``outer face'').  
Similarly as in~\cite{BFgir} (which deals with maps), we exhibit canonical orientations for hypermaps characterizing the ingirth constraints. Then, by applying the master bijection to canonically oriented (and weighted) hypermaps we obtain bijections between any class of hypermaps defined by face-degree constraints and ingirth constraints (with the sole restriction that the ingirth equals the degree  of the outer face, which is dark), and a class of weighted hypermobiles (characterized by local degree and weight conditions). 
We show that the bijections for hypermaps in \cite{BMSc00,BMSc02,BDG04,BDG07} are special cases of our construction.
In terms of counting, we obtain for any $d\geq 1$ an expression for the generating function of 
rooted hypermaps of ingirth $d$ and dark outer face of degree $d$, with control on the dark and light
face degrees.


\begin{figure}
\begin{center}
\includegraphics[width=\linewidth]{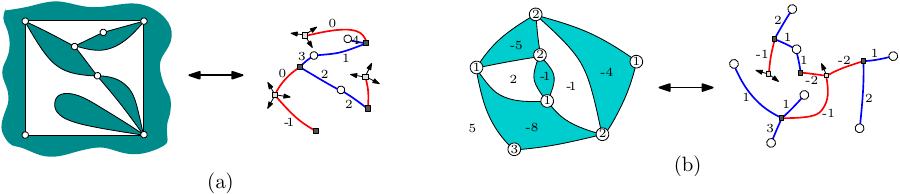}
\end{center}
\caption{Left: example of correspondence between a planar hypermap of ingirth $4$ and a (weighted) hypermobile. Right: example of correspondence between a charged hypermap and a (weighted) hypermobile.} 
\label{fig:intro}
\end{figure}

In a second part of the article (Sections~\ref{sec:bij-charged-maps}-\ref{sec:counting-annular}), we consider \emph{charged hypermaps}, which are a generalization of hypermaps well suited to study non-uniform cycle-length constraints. Roughly speaking, a \emph{fittingly charged hypermap} is a hypermap together with an assignment of a real number, called \emph{charge}, to each vertex and face such that
\begin{compactitem}
\item for any cycle $C$ enclosing a set $R$ of faces (possibly $R$ contains the outer face) such that $C$ is only incident to light faces of $R$, the sum of the charges of the vertices and faces enclosed by $R$ is at most the length of $C$,
\item the charges of vertices are positive, and the sum of all charges is 0.
\end{compactitem} 
See Figure~\ref{fig:intro}(b) for an example of a fittingly charged hypermap, and Section~\ref{sec:bij-charged-maps} for more precise definitions.
We show (again using the master bijection together with canonical orientations) that there is a bijection between the class of fittingly charged hypermaps and a class of weighted hypermobiles. This bijection keeps track of the face-degrees and of all the charges. An example is shown in the right-part of Figure~\ref{fig:intro}. The bijections in the first part of the article are special cases of the bijection for charged hypermaps. We also show in Section~\ref{sec:application-charged-maps} that the machinery of charged hypermaps can be used to get bijections for classes of \emph{annular hypermaps} defined by face-degree conditions and two types of ingirth conditions (and we count these hypermaps in Section~\ref{sec:counting-annular}). 

Let us mention that our master bijection comes in three ``flavors'' $\Phi_+,\Phi_-$ and $\Phi_0$ (see Theorem~\ref{theo:master_bij1}). The flavor depends on the type of rooting of the hypermap: the hypermap has either a marked light face, a marked dark face or a marked vertex. Accordingly, our results for charged hypermaps come in three flavors (see Theorems~\ref{theo:bij-shifted-dark}, \ref{theo:bij-shifted-light} and~\ref{theo:bij-shifted-0}).\\

\noindent\textbf{Charged hypermaps: a preview.}
The machinery of charged hypermaps proves well suited to  establish unified bijections for hypermaps. 
We hope that this machinery will be useful to tackle new problems in the future, and in particular  to prove isoperimetric inequalities for random maps in the spirit of \cite{Miermont:sphericity,LeGall:isoperimetric-maps}.
In order to give a preview of the notion of charged hypermaps, and illustrate its potential use, we now state a special case of our results about charged hypermaps. For simplicity, we will also restrict to the case of charged \emph{maps}.  
Given a map $M$ with a distinguished \emph{root-vertex} $v_0$, we call \emph{partial charge function} a function $\si$ from the vertex set $V$ to $\mathbb{R}$. We say that $\si$ \emph{fits} $M$ if the following conditions hold:
\begin{itemize}
\item[(a)] for any subset $R$ of faces of $M$ defining a simply connected region of the sphere (after adding the edges and vertices incident only to faces in~$R$), the set of edges $\R$ separating a face in $R$ and a face not in $R$ satisfies    
$|\R|\geq 2+ \sum_{v\textrm{ inside }R}(\si(v)-2),$
 with strict inequality if $v_0$ is inside $R$ (a vertex is said to be inside $R$ if all the incident faces are in $R$),
\item[(b)] $\si(v_0)=0$, $\si(v)>0$ for all $v\neq v_0$, and $\sum_{v\in V}\si(v)=2|V|-4$.
\end{itemize}
We call \emph{mobile} a plane tree with two types of vertices -- \emph{round} and \emph{square} -- and with dangling half-edges -- called \emph{buds} -- incident to square vertices. The \emph{excess} of a mobile is the number of half-edges incident to round vertices minus the number of buds.
We call \emph{suitably weighted} a mobile with no edge joining two round vertices, where each edge joining a square vertex to a round vertex carries a positive weight, 
such that the sum of weights of edges incident to a square vertex 
$v$ is $\deg(v)-2$ (the weight is 0 for edges joining two square vertices). 
\begin{thm}[Special case of Theorem~\ref{theo:bij-shifted-0}]\label{thm:special-charged-maps}
There is a bijection between the set of pairs $(M,\si)$ where $M$ is a map with a distinguished root-vertex, and $\si$ is a partial charge function fitting $M$, and the set of suitably weighted mobiles of excess 0. Moreover,
faces of degree $k$ of the map correspond bijectively to square vertices of degree $k$ in the mobile, and
vertices of charge $w$ correspond bijectively to round vertices of weight $w$ (i.e., the incident edge weights sum to~$w$). 
\end{thm}
We hope that this type of bijections can be used to study cycle lengths in large random maps, and their scaling limit, the so-called \emph{Brownian map} \cite{LeGall:limitmaps,LeGall:Brownian-map-is-limit,Miermont:Brownian-map-is-limit}. In particular, since typical distances in random maps with $n$ edges scale like $n^{1/4}$, it would be interesting to look at a partial charge function $\si$ such that $\si(v)=2\pm\frac{\al}{n^{1/4}}$ for all $v$ (for some constant $\al$, and with the signs being independent and uniformly random). In this case, Theorem~\ref{thm:special-charged-maps} gives a way of counting maps such that the boundary of any simply connected set of faces $R$ satisfies $|\R|\geq 2+ \sum_{v\textrm{ inside }R}\si(v)-2$, which is asymptotically Gaussian of amplitude $\al\sqrt{\be}n^{1/4}$ if $R$ contains $\be n$ vertices. This may give a bijective method for proving isoperimetric inequalities in the spirit of \cite{LeGall:isoperimetric-maps}.\\

\noindent{\bf Relation with other bijections for maps and hypermaps.} 
As already said, the present article generalizes our previous work on maps (again 
this relies on the fact that maps are merely hypermaps with all dark faces of degree $2$). The diagram in Figure~\ref{fig:diagram} summarizes the relations between the bijections in the present article and previous ones. 
Precisely, the master bijection for hypermaps given in Section~\ref{sec:master} generalizes the master 
bijection for maps given in~\cite{BFbij}, and the bijection for hypermaps of \igirth\ $d$, dark outer face of degree $d\geq 1$ and control on the face-degrees, generalizes the bijection for plane maps
of outer degree $d$ and girth $d$ obtained
in~\cite{BFgir}. The case $d=1$ for hypermaps identifies to the bijection of Bousquet-M\'elou and Schaeffer~\cite{BMSc02} 
(stated in terms of bipartite maps in~\cite{BMSc02}) with applications to the Ising model and the hard particle model. 
The case $d\geq 2$ admits a natural specialization to $d$-constellations, which coincides with the bijection of Bousquet-M\'elou and 
Schaeffer~\cite{BMSc00}. And we also provide a special formulation for the case $d=0$, 
from which we recover the bijection by Bouttier, Di Francesco and Guitter for vertex-rooted hypermaps~\cite{BDG04} and for vertex-rooted hypermaps with blocked edges~\cite{BDG07}
(with applications to hard particle models, the Ising model, and forested maps enumeration). 

Moreover, since we generalize the results for maps in~\cite{BFgir}, 
we also recover the various known bijections for maps obtained as specializations in~\cite{BFgir}: in particular the case $d=1$ in~\cite{BFgir} identifies to the bijection of Bouttier et al. in~\cite{Boutt},
 the case $d=2$ includes the bijections of~\cite{Sc97} for bipartite maps and of~\cite{PS03a} 
for loopless triangulations, the case $d=3$ includes the bijection of~\cite{FuPoScL} for 
simple triangulations, and the case $d=4$ includes the bijection of~\cite[Sect.~2.3.3]{Schaeffer:these} for simple quadrangulations. 
Similarly the bijection for annular hypermaps (two marked faces) in Section~\ref{sec:application-charged-maps} generalizes the bijection for annular maps obtained in~\cite[Sect.~5]{BFgir}. 

In contrast, the results in the second part of the article (bijection between hypermobiles
and charged hypermaps, allowing to formulate non-uniform girth constraints) are totally new
(the subcase of charged maps is not covered in~\cite{BFgir}, and in fact dealing directly with the 
more general case of charged hypermaps somehow simplifies the proofs). 

\begin{figure}[h]
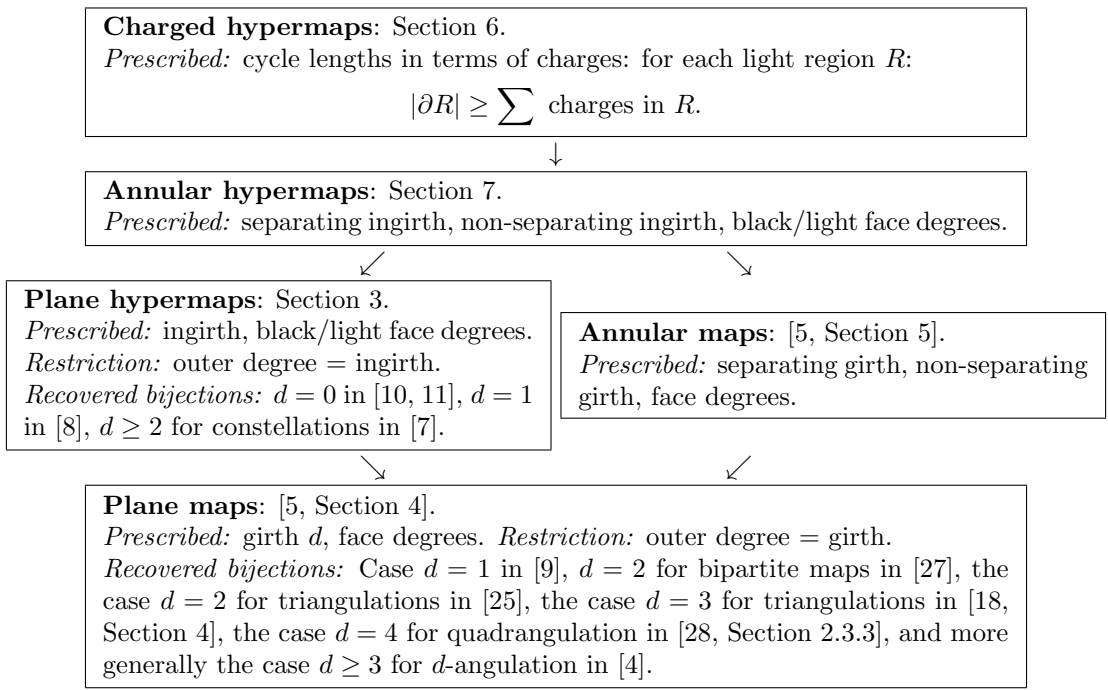

\begin{center}
\framebox{
\begin{minipage}{.8\linewidth}
\textbf{Charged hypermaps}: Section~\ref{sec:bij-charged-maps}.\\
\emph{Prescribed:} cycle lengths in terms of charges: for each light region $R$:
$$|\partial R|\geq \sum \textrm{ charges in }R.$$
\end{minipage}
}
\\[1pt]$\downarrow$\\[1pt]
\framebox{
\begin{minipage}{.8\linewidth}
\textbf{Annular hypermaps}: Section~\ref{sec:application-charged-maps}.\\
\emph{Prescribed:} separating ingirth, non-separating ingirth, dark/light face degrees.
\end{minipage}
}
\\[1pt]$\swarrow$\hspace{.3\linewidth}$\searrow$ \\[1pt]
\framebox{
\begin{minipage}{.45\linewidth}
\textbf{Plane hypermaps}: Section~\ref{sec:bij_plane}.\\
\emph{Prescribed:} ingirth, dark/light face degrees. \\
\emph{Restriction:} outer degree = ingirth.\\
\emph{Recovered bijections:} $d=0$ in \cite{BDG04,BDG07}, $d=1$ in \cite{BMSc02}, $d\geq 2$ for constellations in \cite{BMSc00}.
\end{minipage}
}
\framebox{
\begin{minipage}{.45\linewidth}
\textbf{Annular maps}: \cite[Section 5]{BFgir}.\\
\emph{Prescribed:} separating girth, non-separating girth, face degrees.
\end{minipage}
}\hspace{.01\linewidth}
\\[1pt]$\searrow$\hspace{.3\linewidth}$\swarrow$ \\[1pt]
\framebox{
\begin{minipage}{.8\linewidth}
\textbf{Plane maps}: \cite[Section 4]{BFgir}.\\
\emph{Prescribed:} girth $d$, face degrees. \emph{Restriction:} outer degree = girth.\\
\emph{Recovered bijections:} Case $d=1$ in \cite{Boutt}, $d=2$ for bipartite maps in \cite{Sc97}, the case $d=2$ for triangulations in \cite{PS03a}, the case $d=3$ for triangulations in \cite[Section~4]{FuPoScL}, the case $d=4$ for quadrangulation in \cite[Section~2.3.3]{Schaeffer:these}, and more generally the case $d\geq 3$ for $d$-angulation in \cite{BFbij}.
\end{minipage}
}
\end{center}
\caption{Relation between the bijections in this article and previous ones; arrows indicate specializations.} \label{fig:diagram}
\end{figure}

\vspace{.2cm}

We would like to mention two other general combinatorial methods for counting maps. 
Recall that our master bijection for hypermaps generalizes the master bijection for maps given in~\cite{BFbij}.
In the recent article \cite{AlPo13}, Albenque and Poulalhon have presented another general bijective approach to maps.
The two approaches are closely related and use essentially the same canonical orientations (exhibited in \cite{BFbij}). 
The main difference between the approach in \cite{BFbij} and in \cite{AlPo13} is that the master bijections between oriented maps and trees are different (one tree is a spanning tree of the map, while the other is a spanning tree of the quadrangulation of the map).
Both master bijections are actually based on the two types of trees shown to be associated with ``minimal accessible orientations'' in the article \cite{OB:boisees} (which has been reformulated and extended to higher genus in~\cite{OB:covered-maps}). The existence of these two ``master bijections'' explains why \emph{two} different bijections have been found for several classes of maps, one being generalized in  \cite{BFbij} and the other in \cite{AlPo13}.
For instance,  \cite{BFbij} and \cite{AlPo13} respectively generalize the bijections originally found in~\cite{FuPoScL} and~\cite{PS03b} for simple triangulations (i.e. triangulations of girth 3). It seems however that the master bijection in~\cite{BFbij} is better suited to deal with classes of maps where several face degrees are allowed. 
%

Another unified combinatorial approach to maps was developed by Bouttier and Guitter in~\cite{BG13irr} (building on \cite{BG:countinued-fractions}). 
They show that one of the desirable feature of trees, namely that they are easy to enumerate thanks to their natural recursive structure, could be directly achieved at the level of the maps themselves via so-called \emph{slice decomposition} of maps. With this method, they obtain the generating function of maps of pseudo girth $d$ (maps in which cycles have length at least $d$, except for the contours of faces, which can be of length $d-1$) with control on the face-degrees, thereby generalizing the counting results of~\cite{BFbij} (in which faces of degree $d-1$ were forbidden). 

It is unclear if the methods used in~\cite{AlPo13,BG13irr} can be generalized  to hypermaps, and/or to charged maps.

\vspace{.2cm}


\noindent{\bf Outline.} The outline of the paper is as follows. 
In Section~\ref{sec:master}, we define hypermaps and hypermobiles, and we present the master bijection between a class of oriented hypermaps and a class of hypermobiles.
In Section~\ref{sec:bij_plane}, we consider for each $d\geq 1$ the class $\cC_d$ of hypermaps of ingirth $d$ with a dark outer face of degree $d$. 
By applying the master bijection to canonically oriented maps in $\cC_d$ we obtain a bijection between $\cC_d$ and a class of hypermobiles.
In Section~\ref{sec:count_plane_hypermaps}, we obtain the generating function of the class $\cC_d$ of hypermaps counted according to the degree distribution of their faces (by recursively decomposing the associated hypermobiles).
In Section~\ref{section:recovering-bijections}, we show that the bijections described in \cite{BMSc00,BMSc02,BDG04,BDG07} are special cases of the bijections obtained in Section~\ref{sec:bij_plane}.
In Section~\ref{sec:bij-charged-maps}, we obtain a general bijection for fittingly charged hypermaps. As before, this bijection is obtained by first characterizing fittingly charged hypermaps by suitable canonical orientations and then applying the master bijection.
In Section~\ref{sec:application-charged-maps}, we use the framework of charged hypermaps to obtain bijections for classes of annular hypermaps characterized by separating and non separating girth constraint. 
In Section~\ref{sec:counting-annular} we obtain the generating function of those classes.
In Section~\ref{sec:proof_master_bij}, we gather some proofs about the master bijection.
In Section~\ref{sec:proofs}, we gather our proofs about canonical orientations.\\

\section{Master bijection for hypermaps}\label{sec:master}
\subsection{Hypermaps and hyperorientations}
A \emph{map} is a connected graph embedded on the sphere, considered up to continuous deformation. 
An \emph{Eulerian map} is a map such that all vertices have even degree. Such maps are also those
whose faces can be bicolored -- say there are \emph{dark faces} and \emph{light faces} -- 
in such a way that every edge separates a dark face from a light face. Note that this bicoloration is unique up to the choice of the color of a given face. 
An \emph{hypermap} is a face-bicolored Eulerian map; dark faces are also called \emph{hyperedges}. The \emph{underlying map} is the (Eulerian) map obtained
from the hypermap by forgetting the face types. A \emph{corner} of a map is the an angular section between two consecutive half-edges around a vertex. 
The \emph{degree} of a vertex or face $a$, denoted by $\deg(a)$, is the number of incident corners.

 A \emph{face-rooted hypermap} is a hypermap with a marked face (either dark or light) called the \emph{outer face}. The other faces are called \emph{inner faces}. The vertices and edges are called \emph{outer} if they are incident to the outer face and inner otherwise. The \emph{outer degree} of a face-rooted hypermap is the degree of the outer face.
Observe that face-rooted hypermaps, can also be thought of as hypermaps embedded in the plane (with the outer face being infinite), and for this reason they are sometimes called \emph{plane hypermaps}.
A \emph{dark-rooted} (resp. \emph{light-rooted}) hypermap is a face-rooted hypermap such that the outer face is dark (resp. light). 
A \emph{vertex-rooted hypermap} is a hypermap with a marked vertex called the \emph{root-vertex}. A \emph{corner-rooted} hypermap 
is a hypermap with a marked corner called the \emph{root-corner}.
\begin{figure}[h!]
\begin{center}
\includegraphics[width=8cm]{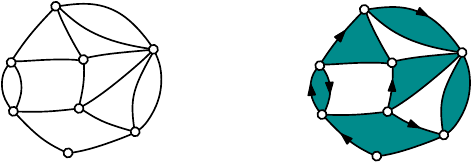}
\end{center}
\caption{Left: an Eulerian map $M$ (all vertices of $M$ have even degree).
Right: a hypermap (having $M$ as underlying Eulerian map) endowed with a hyperorientation.} 
\label{fig:hypermap}
\end{figure} 

A \emph{hyperorientation} $O$ of a hypermap $H$ is a partial orientation (edges are either oriented or unoriented) of the edges of $H$
such that each oriented edge has a dark face on its right. Oriented edges are called \emph{1-way},
unoriented edges are called \emph{0-way}. 
Directed outer edges are called \emph{cw-outer} or \emph{ccw-outer} respectively, depending on whether they have the 
outer face on their left or on their right. 
A \emph{directed path} from $u$ to $v$ is a sequence of 1-way edges $e_1,\ldots,e_k$ such that the origin of $e_1$ is $u$, the end of $e_k$ is $v$, and for all $i\in\{1,\ldots,k-1\}$ the end of $e_i$ is the origin of $e_{i+1}$. This directed path is a \emph{circuit} if $u=v$. A circuit is called \emph{simple} if the origins of $e_1,\ldots,e_k$ are all distinct. 
If $H$ is an hyperoriented face-rooted hypermap, a simple circuit $C$ is called \emph{clockwise} if the outer face is in the region delimited by $C$ on the left of $C$, and \emph{counterclockwise} otherwise. Similarly, if $H$ is a vertex-rooted hypermap, a simple circuit $C$ is said to be \emph{clockwise} if 
the root-vertex is either on $C$ or in the region delimited by $C$ on the left of $C$; 
and $C$ is said to be \emph{counterclockwise} if the root-vertex is either 
on $C$ or in the region delimited by $C$ on the right of $C$ (note that a circuit passing
by the root-vertex is clockwise and counterclockwise at the same time). 
The hyperorientation is called \emph{minimal} if it has no counterclockwise circuit, 
and is called \emph{accessible} from a vertex $v$ if every vertex $u$ can be reached from $v$ by a directed path. 
By a slight abuse of terminology, we will often refer to a hyperoriented hypermap as a hyperorientation.  

We now define three families of hyperorientions that will play a central role in the master bijections (see Figure~\ref{fig:familieshypermap}). 
We call \emph{face-rooted hyperorientation} a face-rooted hypermap endowed with a hyperorientation. Light-rooted, dark-rooted and vertex-rooted hyperorientations are defined similarly.
\begin{itemize}
\item We define $\cHp$ as the family of light-rooted hyperorientations that are accessible from every outer vertex, minimal, and such that every outer edge is 1-way (the outer face contour is a clockwise circuit, not necessarily a 
simple circuit).  
\item We define $\cHm$ as the family of dark-rooted hyperorientations that are accessible from every outer vertex, such that the outer face contour is a simple counterclockwise circuit, and it is the unique counterclockwise circuit in the hyperorientation. 
\item We define $\cHz$ as the family of vertex-rooted hyperorientations that are accessible from the root vertex $v_0$, and minimal.
\end{itemize}

\begin{figure}[h!]
\begin{center}
\includegraphics[width=12.8cm]{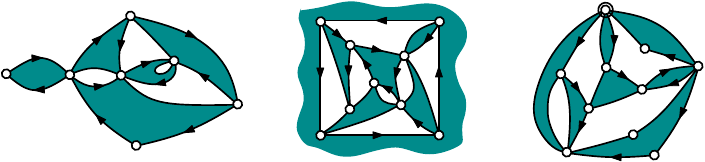}
\end{center}
\caption{Left: a (light-rooted) hyperorientation in $\cHp$. Middle: a (dark-rooted) hyperorientation in $\cHm$. Right: a (vertex-rooted) hyperorientation in $\cHz$.} 
\label{fig:familieshypermap}
\end{figure}

\begin{remark}\label{rk:edges-inward} 
We point out that if a hyperorientation $H$ is in $\cHm$, then there is no inner edge of $H$ incident to an outer-vertex and oriented 1-way toward that outer vertex. Indeed, if we suppose by contradiction that such an inner edge $e$ exists, then because $H$ is accessible, there is a path $P$ of inner edges starting at an outer vertex and ending with the edge $e$. However, this path $P$ together with the contour of the outer face creates a counterclockwise cycle; see Figure~\ref{fig:edges-inward}. This gives a contradiction. 
Similarly, if a hyperorientation is in $\cHz$, then every incidence of an edge $e$ with the root-vertex
$v_0$ is such that $e$ is 0-way or 1-way out of $v_0$. 
\end{remark}

\begin{figure}[h!]
\begin{center}
\includegraphics[width=3cm]{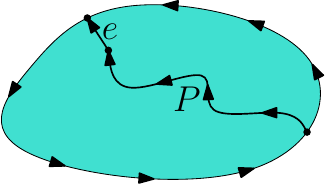}
\end{center}
\caption{The directed path $P$ of inner edges starting at an outer vertex and ending with the edge $e$.} 
\label{fig:edges-inward}
\end{figure}

\subsection{Master bijection $\Phi_*$}
We now define the classes of planes trees which are in bijection with the classes of hyperorientations in $\cHp$, $\cHm$ and $\cHz$.
We consider plane trees with dangling half-edges called \emph{buds}. 
An \emph{hypermobile} is a plane tree with buds having 3 types of vertices -- \emph{round}, \emph{dark square}, and \emph{light square} -- and such that 
\begin{compactitem}
\item buds are incident to light square vertices,
\item every edge is incident to exactly one dark square vertex (hence the edge joins a dark square vertex to either a light square vertex or a round vertex).
\end{compactitem}
The \emph{degree} of a vertex in the hypermobile is the number of incident half-edges (including buds, for light square vertices). 
The \emph{excess} of the hypermobile is the number of edges with a round extremity, minus the number of buds. 
We denote respectively by $\cTp$, $\cTm$, and $\cTz$ the families of hypermobiles of positive excess, negative excess, and zero excess. 

We now describe the master bijection for hypermaps. Actually, there are 3 bijections denoted by $\Phi_+$, $\Phi_-$ and $\Phi_0$, and mapping the classes of hyperorientations $\cHp$, $\cHm$, $\cHz$ respectively to the classes of hypermobiles having positive, negative, and zero excess.

Let $X$ be an hyperorientation in $\cH_*$ with $*\in\{+,-,0\}$. The hypermobile $\Phi_*(X)$ is obtained as follows:
\begin{itemize}
\item Place a dark (resp. light) square vertex of $\Phi_*(X)$ in each dark (resp. light) face of $X$; the vertices of $X$ will become the \emph{round vertices} of $\Phi_*(X)$.
\item Create the edges of $\Phi_*(X)$ by applying to each edge of $X$ the local rule indicated in Figure~\ref{fig:local-rule-hyperori} (ignore the weights $w$ for the time being). Then erase all the edges of $X$.
\item To complete the construction in the case $*=+$ delete the light square vertex in the outer face of $X$ (together with the incident buds). To complete the construction in the case $*=-$, delete the dark square vertex in the outer face of $X$, all the outer vertices of $X$ and the edges linking them. To complete the construction in the case $*=0$, simply delete the root-vertex of $X$.
\end{itemize}
The mappings $\Phi_*$, are illustrated in Figure~\ref{fig:master-bij-hyperori}.

\begin{figure}[h!]
\begin{center}
\includegraphics[width=.6\linewidth]{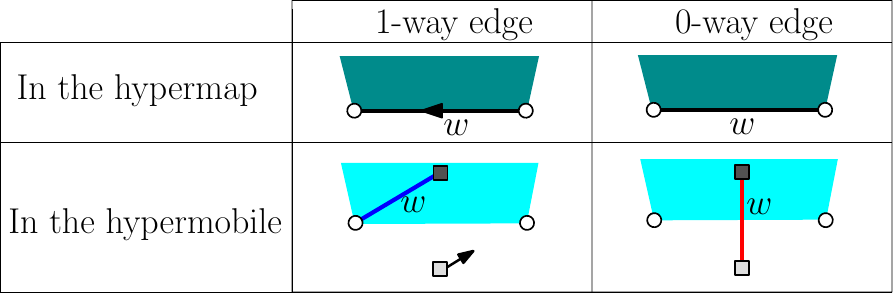}
\end{center}
\caption{Local rules applied in the bijections $\Phi_+$, $\Phi_-$, $\Phi_0$ to every edge of a hyperorientation. 
The rule for the transfer of a weight $w$ is also indicated (for the edge-weighted version of the bijections).}
\label{fig:local-rule-hyperori}
\end{figure} 

\begin{figure}[h!]
\begin{center}
\includegraphics[width=\linewidth]{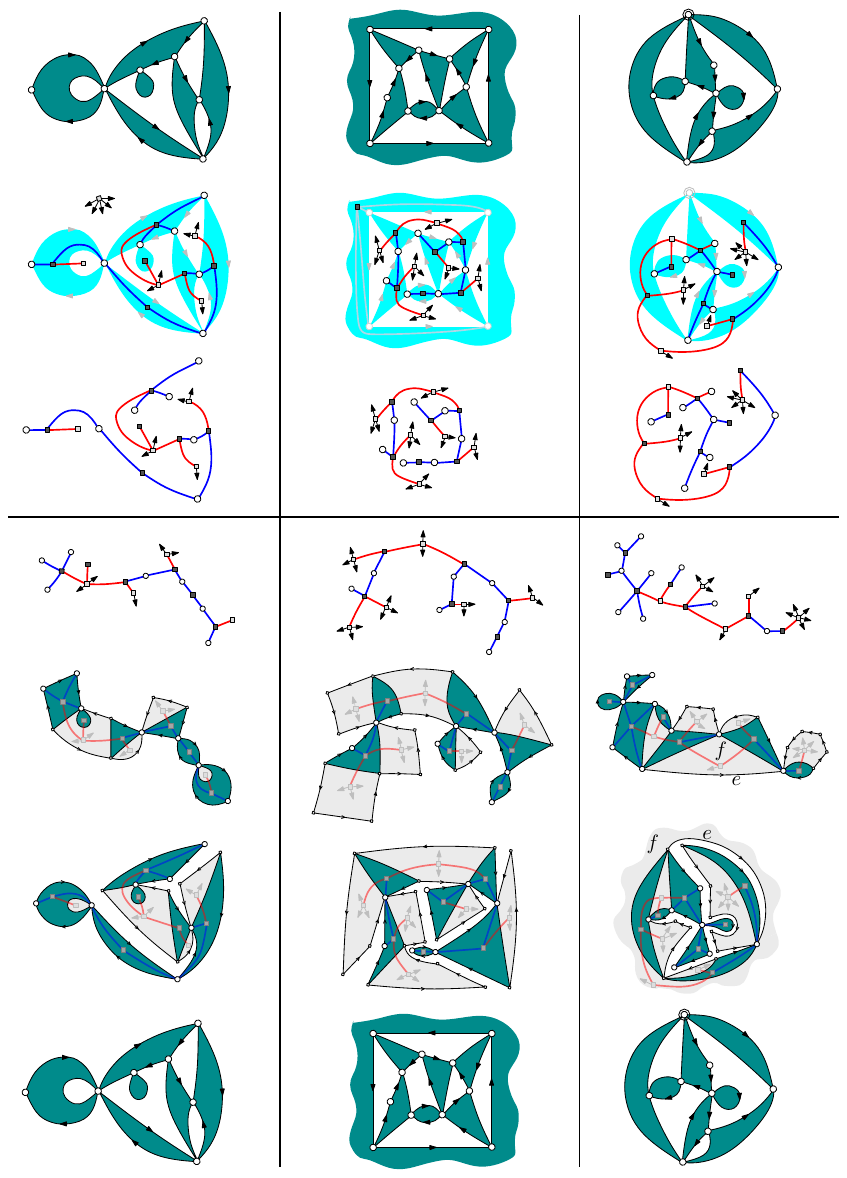}
\end{center}
\caption{The master bijection $\Phi_*$ from hyperorientations to hypermobiles (upper part: 
left $\Phi_+$, middle $\Phi_-$, right $\Phi_0$) and its inverse $\Psi_*$ (lower part: left $\Psi_+$, middle $\Psi_-$, right $\Psi_0$). Hypermobile edges are blue or red whether they have a round extremity or not.}
\label{fig:master-bij-hyperori}
\end{figure} 

\begin{remark} For $X\in\cHp$, all the outer edges are oriented 1-way with the root-face on their left, hence the local rules of Figure~\ref{fig:local-rule-hyperori} do not create any edge incident to the light square vertex in the outer face of $X$ (only buds). Thus, the last step to complete $\Phi_+(X)$ only deletes an isolated vertex. 
Similarly, for $X\in\cH_0$, the last step to complete $\Phi_0(X)$ only deletes an isolated vertex. Lastly, for $X\in\cHm$ the local rules of Figure~\ref{fig:local-rule-hyperori} do not create any edge incident to the outer vertices of $X$, except for the edges joining them to the dark square vertex in the outer face of $X$ (because by Remark~\ref{rk:edges-inward} no inner edge is 1-way toward an outer vertex). Hence the last step to complete $\Phi_-(X)$ only deletes an isolated ``star graph'' made of these vertices and edges. 
\end{remark}

\begin{theo}\label{theo:master_bij1}
For $*\in\{+,-,0\}$ the mapping $\Phi_*$ is a bijection between $\cH_*$ and $\cT_*$. 
For $\Phi_+$ the outer degree of $\gamma\in\cHp$ is equal to the excess of $\tau=\Phi_+(\gamma)$, 
for $\Phi_-$ the outer degree of $\gamma\in\cHm$ is equal to minus the excess of $\tau=\Phi_-(\gamma)$. 
\end{theo}

The proof of Theorem~\ref{theo:master_bij1} is postponed to Section~\ref{sec:proof_master_bij}. 
We will now formulate a version of the bijections $\Phi_*$ for \emph{edge-weighted} hyperorientations, and explain the parameter correspondences.

A hyperorientation is \emph{weighted} by assigning a weight in $\br$ to each
edge. In that case, the \emph{weight} of a vertex is the total weight of its
incident ingoing edges, the \emph{weight} of a light face is the total weight of
its incident 0-way edges, and the \emph{weight} of a dark face is the total weight
of its incident edges. For hyperorientations is in $\cHm$, we take the convention that all outer edges
(which are 1-way) have weight $1$. 
Similarly a hypermobile is \emph{weighted} by assigning a weight in $\br$ 
to each of its (non-bud) edges. The \emph{weight} of a vertex of a hypermobile $M$
is the total weight of its incident (non-bud) edges, and the degree
of a vertex of $M$ is the number of incident half-edges (including buds, for light square vertices). 
The local rule of Figure~\ref{fig:local-rule-hyperori} 
can directly be adapted so as to transfer
the weight of an edge of the hypermap to the corresponding
edge in the associated hypermobile, see Figure~\ref{fig:local-rule-hyperori}.
Hence, Theorem \ref{theo:master_bij1} has the following corollary.
\begin{cor}
The mapping $\Phi_+$ (resp. $\Phi_-$, $\Phi_0$) is a bijection between weighted hyperorientations 
from $\cHp$ (resp. $\cHm$, $\cHz$) and weighted hypermobiles of positive excess (resp. negative excess, zero excess).
\end{cor}


We now formulate the parameter correspondences between hypermaps and hypermobiles. In order to make a formulation valid simultaneously for $\Phi_+$, $\Phi_-$ and $\Phi_0$, we first define the \emph{frozen} vertices, edges and faces of a hyperorientation $H$ in $\cH_+$, $\cH_-$ and $\cH_0$.
For $H\in\cHp$, only the outer face is frozen. For $H\in\cH_0$, only the root-vertex is frozen. For $H\in\cHm$, the outer face, all the outer edges and all the outer vertices are frozen. With this terminology, for $*\in\{+,-,0\}$, for $H\in\cH_*$ and $T=\Phi_*(X)$, we have
\begin{compactitem}
\item each non-frozen light (resp. dark) face of $H$ corresponds to a light (resp. dark) square vertex of the same degree and same weight in $T$;
\item each non-frozen edge of $H$ corresponds to a (non-bud) edge of the same weight in $T$;
\item each non-frozen vertex of $H$ of weight $w$ and indegree $\delta$ corresponds to a round vertex of $T$ of weight $w$ and degree $\delta$.
\end{compactitem}

\subsection{Inverse bijections $\Psi_*$}
We will now describe the inverses $\Psi_+$, $\Psi_-$, and $\Psi_0$ of the bijections $\Phi_+$, $\Phi_-$, and $\Phi_0$. 
Let $T$ be a hypermobile. We associate with $T$ an \emph{outerplanar map} $\hT$ (a plane map such that every vertex is incident to the outer face) 
as follows:
\begin{compactitem}
\item for each dark (resp. light) vertex of degree $d$ in $T$ we create a dark (resp. light) polygon of degree $d$ following the rules illustrated in Figure~\ref{fig:grow};
\item for each edge $e$ of $T$ between a dark square and a light square vertex, we glue together the two face sides of the corresponding polygons at $e$;
\item for each round vertex $v$ of $T$ of degree $d$ we merge the $d$ neighboring polygon corners with the vertex $v$.
\end{compactitem}
\begin{figure}[h!]
\begin{center}
\includegraphics[width=\linewidth]{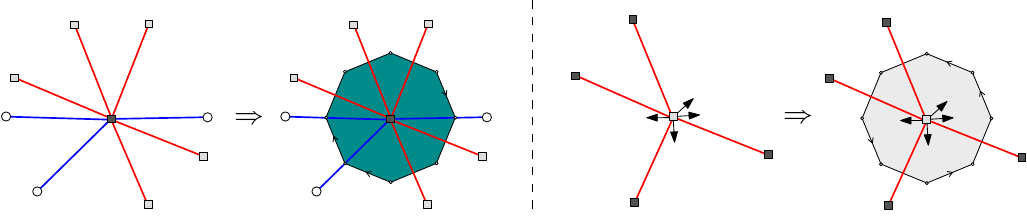}
\end{center}
\caption{Left: growing a dark polygon at a dark square vertex of a hypermobile. Right: growing a light polygon at a light square vertex of a hypermobile.}
\label{fig:grow}
\end{figure}
See Figure~\ref{fig:outerplanarmap} for an example. 
Note that the inner edges of the outerplanar map $\hT$ are 0-way and the outer edges are 1-way: cw-outer edges of $\hT$ correspond to edges between a round and a dark square vertex in $T$, ccw-outer edges of $\hT$ correspond to buds of $T$.

\begin{figure}[h!]
\begin{center}
\includegraphics[width=0.7\linewidth]{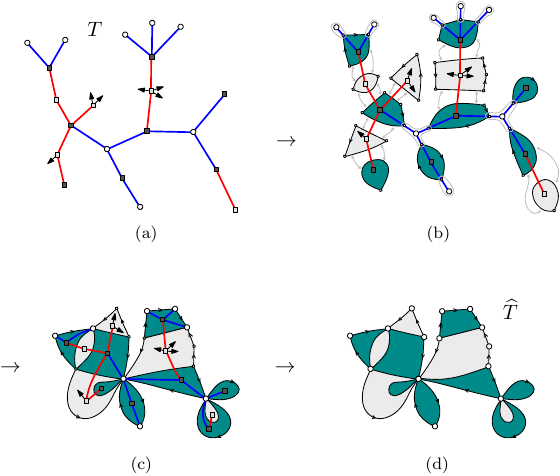}
\end{center}
\caption{From a hypermobile $T$ to the associated outerplanar map $\hT$. 
(a) The hypermobile $T$. 
(b) Creating the polygons around square vertices. 
(c) Gluing polygon-sides corresponding to each edge of $T$ between a dark square and a light square vertex, and merging polygon-corners neighboring each round vertex of $T$. 
(d) The outerplanar map $\hT$.}
\label{fig:outerplanarmap}
\end{figure} 


The mappings $\Psi_+$, $\Psi_-$, $\Psi_0$ (which will be proved to be the inverse bijections of $\Phi_+$, $\Phi_-$, $\Phi_0$ respectively) are defined as follows. Let $T$ be a hypermobile, and let $\hT$ be the associated outerplanar map. We will now define a canonical way of gluing together the cw-outer and ccw-outer edges of $\hT$; see Figure~\ref{fig:outerplanarmap_closed}.

A word $w$ (i.e. sequence of letters) on the alphabet $\{a,\ba\}$ is a \emph{parenthesis word} if $w$ has as many letters $a$ as letters $\ba$, and for any prefix of $w$ the number of letters $a$ is at least equal to the number of letters $\ba$. 
A \emph{cyclic word} is a word considered up to cyclic shift of the letters. Given a cyclic word $w$ on the alphabet $\{a,\ba\}$, we say that a letter $a$ and a letter $\ba$ are \emph{cw-matching}, if the subword of $w$ starting after the letter $a$ and ending before the letter $\ba$ is a (possibly empty) parenthesis word. An example is given in Figure~\ref{fig:outerplanarmap_closed}(a). 
It is easy to see that for any letter $a$ there is at most one cw-matching letter $\ba$ and vice-versa. Moreover if a cyclic word $w$ has $n_a$ letters and $n_\ba$ letters $\ba$ with $n_a\geq n_\ba$ (resp. $n_a\leq n_\ba$), then all the letters are cw-matching except for $n_a-n_\ba$ letters $a$ (resp. $n_\ba-n_a$ letters $\ba$).

We are now ready to define a canonical way of gluing the cw-outer and ccw-outer edges of $\hT$. We associate a cyclic word $w_T$ with the sequence of outer edges appearing in clockwise order around the outer face of $\hT$ by encoding the cw-outer and ccw-outer edges by the letters $a$ and $\ba$ respectively.
We say that a cw-outer edge and a ccw-outer edge of $\hT$ are \emph{cw-matching} if the corresponding letters $a$ and $\ba$ are cw-matching in $w_T$. It is easy to see that all the pairs of cw-matching edges can be glued together into 1-way edges (that is, there is no breach of planarity in doing so for every pair of cw-matching edges). An example is given in Figure~\ref{fig:outerplanarmap_closed}(b). 
If the excess $\eps$ of $T$ is positive, then $\hT$ has $\epsilon$ more cw-outer edges than ccw-outer edges. Thus the map obtained after gluing the cw-matching edges of $\hT$ has an outer face of degree $\epsilon$ which is incident only to cw-outer edges. Hence coloring the outer face as light gives an oriented light-rooted hypermap, that we denote by $\Psi_+(T)$. Similarly, if the excess $\epsilon$ of $T$ is negative, then the map obtained after gluing the cw-matching edges of $\hT$ has an outer face of degree $-\epsilon$ which is incident only to ccw-outer edges. Hence coloring the outer face as dark gives an oriented dark-rooted hypermap, that we denote $\Psi_-(T)$. Lastly, if the excess of $T$ is 0, then all the outer edges of $\hT$ are glued. Moreover, there is a unique vertex $v_0$ of the glued map without incident ingoing edges. In this case, taking $v_0$ as the root-vertex gives an oriented vertex-rooted hypermap that we denote by $\Psi_0(T)$. Examples for $\Phi_+$, $\Phi_-$, $\Phi_0$ are given in Figure~\ref{fig:master-bij-hyperori}.

\begin{figure}[h!]
\begin{center}
\includegraphics[width=\linewidth]{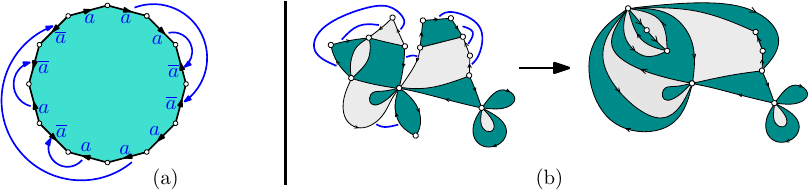}
\end{center}
\caption{
(a) The cyclic word $aaa\ba\ba aaa\ba a\ba\ba$ (represented in clockwise order around a polygon), and the cw-matching pairs of $a$'s and $\ba$'s (indicated by the arrows).
(b) Gluing of the cw-matching pairs of edges of the outerplanar maps $\hT$.}
\label{fig:outerplanarmap_closed}
\end{figure}

\begin{theo}\label{theo:master_bij2}
The mappings $\Psi_+$, $\Psi_-$, and $\Psi_0$ are the inverses of the bijections $\Phi_+$, $\Phi_-$, and $\Phi_0$ respectively.
\end{theo}

We shall prove Theorem~\ref{theo:master_bij2} in Section~\ref{sec:proof_master_bij}.

\subsection{Alternative formulation of the inverse bijections $\Psi_*$}\label{sec:alter}
For the sake of completeness, we now give an alternative description of the mappings $\Psi_*$, which is closer to the description of many known bijections (in particular, the bijections in~\cite{BMSc00,BMSc02,BDG07} discussed in Section~\ref{section:recovering-bijections}).
Let $T$ be a hypermobile, where buds are interpreted as \emph{outgoing sprouts}. 
Add ingoing sprouts as follows: for each edge $e=\{u,v\}\in T$ connecting a round vertex $u$ to a dark square vertex $v$, 
 insert an ingoing sprout in the corner following $e$ in counterclockwise order around $v$. See Figure~\ref{fig:closure_other_way} for an example. 
 We associate a cyclic word $w_T$ with the sequence of sprouts appearing in clockwise order around the outer face of $T$ (with the outer face on the left of the walker) by encoding the ingoing and outgoing sprouts by the letters $a$ and $\ba$ respectively. We join the cw-matching ingoing and outgoing sprouts to form oriented edges, 
and then remove from $T$ the round vertices and their incident edges.
The embedded partially oriented graph $G$ thus obtained is called the \emph{partial closure} of $T$. 
If $T$ has nonzero excess $\eps$, then there remain $|\eps|$ unmatched sprouts in $G$ (which are ingoing if $\eps>0$, outgoing if $\eps<0$). 
The \emph{complete closure} of $T$, denoted by $\bG$ is defined as follows: if $\eps=0$ then $\bG=G$, while if $\eps>0$ (resp. $\eps<0$) $\bG$ is obtained from $G$ by adding a new light (resp. dark) square vertex $v_0$ in the face containing all the sprouts and connecting these sprouts to $v_0$ 
by new edges directed away from $v_0$ (resp. toward $v_0$). Observe that $\bG$ is a bipartite map since every edge is incident to one dark square vertex.
Finally, we call $\Psi_*(T)$ (for $*\in\{+,-,0\}$ depending on whether the excess $\eps$ is positive, negative, or zero) the dual of $\bG$ which is a hypermap (the dual of dark squares are taken to be dark faces). The edges of $\Psi_*(T)$ are oriented as follows: an edge $e'$ of $\Psi_*(T)$ which is dual to an oriented edge $e$ of $\bG$ (made by connecting two sprouts) is oriented 1-way from the right-side of $e$ to the left-side of $e$, while an edge of $\Psi_*(T)$ which is dual to an original edge of $T$ is left unoriented. Lastly, if the excess $\eps$ is non-zero we take the dual of the vertex $v_0$ of $\bG$ to be the root-face of $\Psi_*(T)$, while if $\eps=0$ we take the dual of the root-face of $\bG$ to be the root-vertex of $\Psi_0(T)$. See Figure~\ref{fig:closure_other_way} for an example. 

It is easy to see that the formulation with sprouts given here is equivalent to 
the formulation with outerplanar maps given above. Indeed, if we 
superimpose the hypermobile with the associated outerplanar map,
then each cw-matching operation in one formulation is equivalent to a cw-matching operation in the other formulation. 
\begin{figure}
\begin{center}
\includegraphics[width=0.7\linewidth]{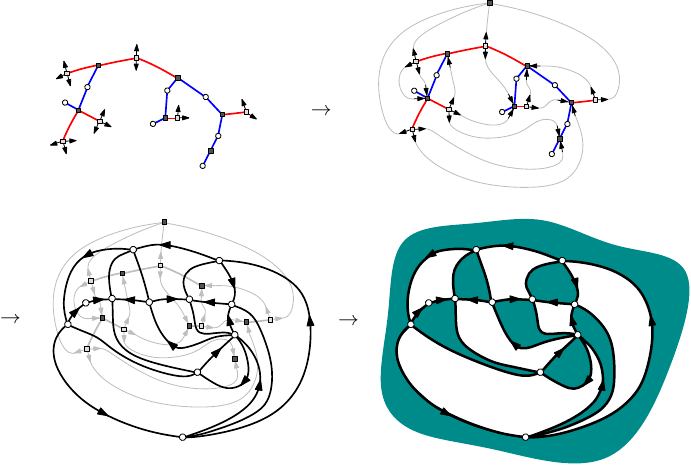}
\end{center}
\caption{The closure mapping $\Psi_-$ performed by joining pairs of cw-matching outgoing and ingoing sprouts, and then taking the dual.}
\label{fig:closure_other_way}
\end{figure}

\subsection{Relation with the master bijection for maps defined in \cite{BFgir}}
In a preceding article~\cite{BFbij} we gave master bijections for planar maps.
More precisely, we considered \emph{bi-oriented} planar maps. A \emph{bi-orientation} of a map is a choice of a direction for each half-edge of the map (thus there are 4 ways of bi-orienting any edge). Seeing maps as a special case of hypermaps, we can describe a bi-orientation as a hyperorientation of the associated hypermap in the way indicated in Figure~\ref{fig:equivalence_local_rule}.

 \begin{figure}[h!]
\begin{center}
\includegraphics[width=.7\linewidth]{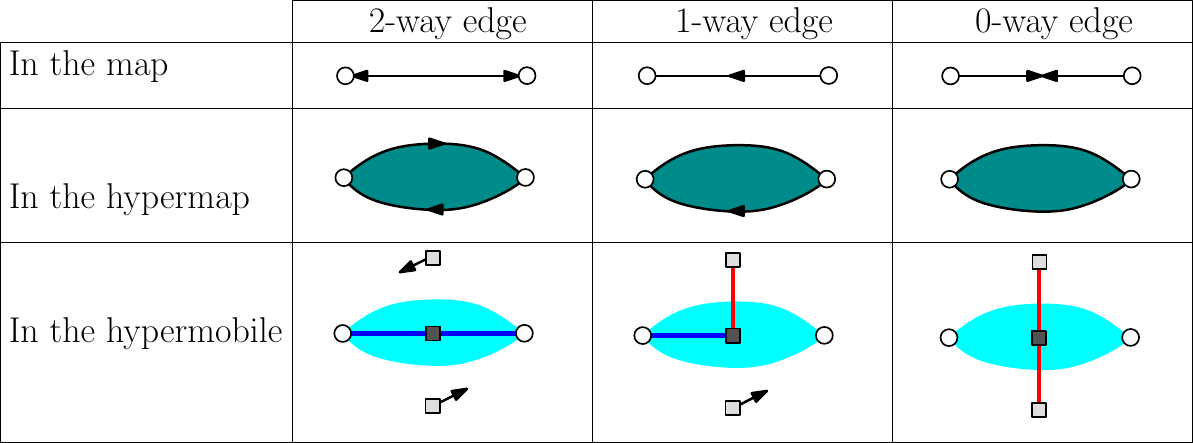}
\end{center}
\caption{
Maps identify to hypermaps by blowing each edge $e$ into a dark face $f$ of degree $2$,
the middle-line also shows how to naturally transfer the orientation information so that indegrees are preserved. 
The bottom line shows that applying the local rules of Figure~\ref{fig:local-rule-hyperori} to (the two edges $\epsilon_1,\epsilon_2$ 
of) $f$ is equivalent to applying the local rules given in~\cite{BFbij} to the underlying edge $e$.}
\label{fig:equivalence_local_rule}
\end{figure}

The master bijection in~\cite{BFbij} consists of 3 constructions denoted by $\Phi_+$, $\Phi_-$, $\Phi_0$
operating on 3 families $\cOp$, $\cOm$, $\cO_0$ of bi-orientations. It is easy to check that,
 under the classical identification of blowing each edge of a map into a dark face of degree $2$,
the families $\cOp$, $\cOm$, $\cO_0$ of bi-orientations considered in~\cite{BFbij} identify respectively to the subfamilies of $\cHp$, $\cHm$, $\cHz$
where all inner dark faces are of degree $2$. Moreover the local rules to carry out the bijections are equivalent
under this identification, see Figure~\ref{fig:equivalence_local_rule}. 
Hence Theorems~\ref{theo:master_bij1} and~\ref{theo:master_bij2} extend the results given in~\cite{BFbij} about the master bijections for maps.

The proof of Theorems~\ref{theo:master_bij1} and~\ref{theo:master_bij2} could actually be obtained using a reduction to the results about the master bijection for maps established in \cite{BFbij}\footnote{In this reduction we would apply
the master bijection of \cite{BFbij} to partially oriented maps, and observe that one can characterize the mobiles which are the image of (hyperoriented) hypermaps (because bicolorability of the faces can be detected on the associated mobiles).}. These in turn, were obtained using results established in \cite{OB:boisees}. Instead we chose to give a simplified -- self-contained -- proof in Section~\ref{sec:proof_master_bij}.


\section{Bijections for plane hypermaps according to the \igirth}\label{sec:bij_plane}
We first define the \emph{ingirth} of a plane hypermap $H$. A simple cycle $C$ of $H$ is called \emph{inward} if all the faces incident to $C$ and inside $C$ (on the side of $C$ not containing the outer face) are light. The \emph{ingirth} of $H$ is then defined as the minimal length of inward cycles. 
Note that the \igirth\ of a plane hypermap whose dark faces have degree $2$ is equal to the girth of the corresponding map. 
In this section, we present bijections for dark-rooted hypermaps with control on the face-degrees and on the \igirth. 

\begin{figure}
\begin{center}
\includegraphics[width=\linewidth]{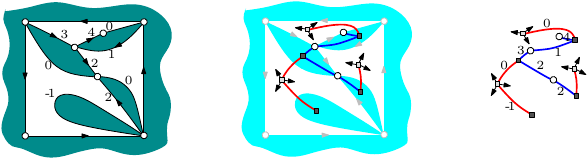}
\end{center}
\caption{The bijection of Theorem~\ref{theo:bij_d_plane} on an example (case $d=4$). 
Left: a dark-rooted hypermap endowed with its unique $4$-weighted 
hyperorientation in $\cHm$. Right: the associated $4$-weighted hypermobile.} 
\label{fig:bij_ingirth_d}
\end{figure}

For $d\geq 1$ and $H$ a dark-rooted hypermap of outer degree $d$, 
a \emph{$d$-weighted hyperorientation} of $H$ is a weighted hyperorientation 
of $H$ such that:
\begin{itemize}
\item The $1$-way edges have positive weight, the $0$-way edges have non-positive weight.
\item Inner vertices have weight $d$.
\item Each light face $f$ has weight $d-\deg(f)$.
\item Each dark inner face $f$ has weight $d\cdot\deg(f)-d-\deg(f)$.
\item Outer vertices and outer edges have weight $1$.
\end{itemize}

\begin{theo}\label{theo:plane_dori}
Let $d$ be a positive integer. A dark-rooted hypermap of outer degree $d$ can be endowed 
with a $d$-weighted hyperorientation if and only if it has \igirth\ $d$. In this case, it has a unique $d$-weighted hyperorientation in $\cHm$. 
\end{theo}
The proof of this theorem is postponed to Section~\ref{sec:proofs}. 
We now define the corresponding hypermobiles. For $d\geq 1$, a \emph{$d$-weighted hypermobile} is a weighted hypermobile such that:
\begin{itemize}
\item Edges incident to a round vertex have positive weight, while edges incident to a light square vertex have non-positive weight.
\item Round vertices have weight $d$.
\item Each light square vertex $v$ has weight $d-\deg(v)$.
\item Each dark square vertex $v$ has weight $d\cdot\deg(v)-d-\deg(v)$.
\end{itemize}

\begin{claim}\label{claim:excess-d-weighted-hypermobile}
Every $d$-weighted hypermobile has excess $-d$. 
\end{claim}
\begin{proof}
Let $n_R, n_L, n_D$ be respectively the numbers of round vertices, light square vertices, and dark square vertices,
let $e_R$ (resp. $e_L$) be the number of edges with a round (resp. light square) extremity, and denote by 
$e=e_R+e_L$ the total number of edges (excluding buds), and by $b$ the number of buds (the excess is $e_R-b$).  
The total weight at round vertices is $dn_R$, the total weight at light square vertices is $dn_L-e_L-b$,
and the total weight at dark square vertices is $de-dn_D-e$. Hence we have 
$de-dn_D-e=dn_R+(dn_L-e_L-b)$. Together with $n_R+n_L+n_D=e+1$,
this gives $e-e_L-b=-d$, hence $e_R-b=-d$.
\end{proof}

\begin{remark} \label{rk:weight-multiple}
The weights of edges in a $d$-weighted hypermobile are always integers. Indeed, since every vertex has integer weight, no vertex can be incident to exactly 1 edge with a non-integer weight. Hence there cannot exist a non-empty subset of edges with non-integer weights (because any such subset has a vertex of degree 1). 
Note also that the same argument shows that if the weights of the vertices of a hypermobile are all multiples of a number $k$, then the edge weights are also multiples of $k$.
\end{remark}

Given Theorem~\ref{theo:plane_dori} we can apply the master bijection $\Phi_-$ for hypermaps. Given the parameter correspondence for $\Phi_-$ we obtain the following result; see Figure~\ref{fig:bij_ingirth_d} for an example.

\begin{theo}\label{theo:bij_d_plane}
Let $d$ be a positive integer. Dark-rooted hypermaps of outer degree $d$ and \igirth\ $d$ are in bijection with $d$-weighted mobiles. 
Each light (resp. dark) inner face in the hypermap corresponds to a light (resp. dark) square vertex of the same degree in the associated hypermobile. 
\end{theo}


\section{Counting plane hypermaps of \igirth\ $d$}\label{sec:count_plane_hypermaps}
In this section we determine the generating function $F_d$ of corner-rooted hypermaps of ingirth~$d$ with a dark outer face of degree $d$. Via the master bijection established in Section~\ref{sec:bij_plane} and Lemma~\ref{lem:count} below, this is reduced to counting \emph{rooted} $d$-weighted hypermobiles (whereas counting dark-rooted hypermaps of \igirth\ $d$ amounts to counting \emph{unrooted} $d$-weighted hypermobile which is harder). Then using the classical recursive decomposition of trees at their root we determine~$F_d$.

Recall that a \emph{corner-rooted hypermap} is a hypermap with a marked corner. 
For a corner-rooted hypermap, we define the \emph{root-face} as the face containing the marked corner, and the \igirth\ is defined with respect to this face. 
We now want to use the bijection of Theorem~\ref{theo:bij_d_plane} about dark-rooted hypermaps of \igirth\ $d$ in order to count corner-rooted hypermaps of \igirth\ $d$.
Note that a given face-rooted hypermap with outer degree $d$ can correspond to less than $d$ corner-rooted hypermaps if the face-rooted hypermap has some symmetries. However the master bijection $\Phi_-$ behaves nicely with respect to symmetries and we get the following lemma.
\begin{lem}\label{lem:count}
Let $H$ be a dark-rooted hyperorientation in $\cHm$ and let $T=\Phi_-(H)$ be the corresponding hypermobile. Let $a$ and $b$ be respectively the number of distinct marked hypermobiles obtained by marking a bud of $T$ and by marking an edge of $T$ having a round extremity. Then the number of distinct corner-rooted maps obtained from $H$ by choosing a root-corner in the root face is $c=a-b$.
\end{lem}
\begin{proof}
Let $\de$ be the outer-degree of $H$. By Theorem~\ref{theo:bij_d_plane}, $T$ has excess $-\de$, that is, its numbers $\al$ and $\be$ of buds and edges with a round extremity are related by $\al-\be=\de$. Moreover it is clear from the definition of $\Phi_-$ that $H$ has a symmetry of order $k$ (which has to be a rotational symmetry preserving the root face) if and only if $T$ has a symmetry of order $k$. In other words, $c=\de/k$ if and only if $a=\al/k$ and $b=\be/k$. Thus, $c=\de/k=\al/k-\be/k=a-b$.
\end{proof}


Let $d\geq 1$, and let $\cF_d$ be the family of corner-rooted hypermaps of \igirth\ $d$ with a dark outer face of degree $d$. 
Let $F_d\equiv F_d(x_1,y_1;x_2,y_2;\ldots)$ be the generating function of $\cF_d$
where $x_k$ marks the number of light faces of degree $k$, and $y_k$ marks
the number of dark inner faces of degree $k$. 
Let $A_d\equiv A_d(x_1,y_1;x_2,y_2;\ldots)$ (resp. $B_d\equiv B_d(x_1,y_1;x_2,y_2;\ldots)$)
be the generating function of $d$-weighted hypermobiles with a marked bud (resp. with 
a marked edge having a round extremity), where $x_k$ marks the number of light square vertices
of degree $k$, and $y_k$ marks the number of dark square vertices of degree $k$. 
The bijection of Theorem~\ref{theo:bij_d_plane} and Lemma~\ref{lem:count} ensure that
$$
F_d=A_d-B_d.
$$
We now calculate $A_d$ and $B_d$, with the help of auxiliary generating functions. 
A \emph{planted $d$-hypermobile} is a tree $T$ 
that can be obtained as one of the two 
components after cutting a $d$-weighted hypermobile at the middle of 
an edge $e$. The extremity of $e$ in the chosen component
is called the \emph{root-vertex} of $T$, the half-edge of $e$ in the chosen
component is called the \emph{root-leg} of $T$, and 
 the weight of $e$ is called
the \emph{root-weight} of $T$. 
For $i\in\bz$, we denote by $\cW_i$ (resp. $\cL_i$) the family of planted $d$-hypermobiles with root-weight
$i$ with a root-vertex which is dark-square (resp. not dark-square). 
Define $W_i\equiv W_i(x_1,y_1;x_2,y_2;\ldots)$ (resp. $L_i\equiv L_i(x_1,y_1;x_2,y_2;\ldots)$) as the generating function of $\cW_i$ (resp. $\cL_i$) where $x_k$
marks the number of light square vertices of degree $k$, and $y_k$ marks
the number of dark square vertices of degree $k$. Note that $W_i=L_i=0$ if $i>d$. We also define
\begin{equation}\label{eq:LM}
\W(u)=u+\sum_{k\leq 0}W_k u^{k+1},\ \ \ L(u)=u^{-d+1}\sum_{k\in\bz}L_k u^k.
\end{equation}

We now write equations specifying the series $W_i$ and $L_i$ using the classical recursive decomposition of trees at the root. 
As in~\cite{BFbij,BFgir}, we will need the following notation: for $k\geq 0$ and $s\geq 0$, 
define the multivariate polynomial $h_k(w_1,\ldots,w_s)$ as
$$
h_k(w_1,\ldots,w_s)=[t^k]\frac{1}{1-\sum_{m=1}^st^m w_m}.
$$
In other words, $h_k$ is the generating function of the compositions of $k$ with weight $w_i$ for each part of size $i$.

For $i\geq 1$ any mobile in $\cL_i$ has a root-vertex $v$ which is round. Hence the children of $v$ are dark square, and the edges incident to $v$ have positive weight. Moreover, the total weight at $v$ is $d$, with a contribution $i$ from the root leg.
Hence the total weight of the edges from $v$ to its children is $d-i$. This gives
\begin{equation}\label{eq:1}
L_i=h_{d-i}(W_1,\ldots,W_{d-1})\ \ \mathrm{for}\ i\geq 1.
\end{equation}

For $i\leq 0$, a mobile in $\cL_i$ has a root-vertex $v$ which is light square. 
Hence the children of $v$ are dark square, and the edges incident to $v$ have non-positive weight.
Moreover, the total weight at $v$ is $d-\deg(v)$, with a contribution $i$ from the root-leg. 
If $v$ has degree $\delta$, the $\delta-1$ other half-edges incident to $v$ are either buds or are on an edge (with non-positive weight) leading to a dark square vertex.
This gives
$$
L_i=\sum_{\delta\geq d-i}x_{\delta}[u^{d-\delta-i}]
\Big( 1+\sum_{k\leq 0}W_k u^k \Big)^{\delta-1}\ \ \mathrm{for}\ i\leq 0.
$$ 
In other words, 
\begin{equation}\label{eq:2}
L_i=[u^{d-i-1}]\sum_{\delta\geq d-i}x_{\delta}\W(u)^{\delta-1}\ \ \mathrm{for}\ i\leq 0.
\end{equation}
For $i\in\bz$, a mobile in $\cW_i$ has a root-vertex $v$ that is dark square.
If $v$ has degree $\delta$ then the weight of $v$ is $d\delta-d-\delta$,
with a contribution $i$ from the root-leg. Hence
$$
W_i=\sum_{\delta\geq 1}y_{\delta}[u^{d\delta-d-\delta-i}]\Big( \sum_{k\in\bz}L_k u^k \Big)^{\delta-1}.
$$
In other words, 
\begin{equation}\label{eq:3}
W_i=[u^{-i-1}]\sum_{\delta\geq 1}y_{\delta}\ \!L(u)^{\delta-1}\ \ \mathrm{for}\ i\in\bz.
\end{equation}

\begin{theo}\label{theo:count_plane_hypermaps}
For $d\geq 1$ the generating function $F_d(\xx,\yy)$ of the class $\cF_d$ of corner-rooted hypermaps of \igirth\ $d$ having a dark root-face of degree $d$ is given by
$$
F_d=L_0-\sum_{i=1}^d L_i W_i,
$$
where the series $L_i$ and $W_i$ are specified by \eqref{eq:2} and~\eqref{eq:3} (with $L(u)$ and $\W(u)$ defined in~\eqref{eq:LM}).
 
Moreover
$$
\frac{\partial F_d}{\partial x_k}=\frac{d}{k}[u^d]\W(u)^k,\ \ \ \ \ \frac{\partial F_d}{\partial y_k}=\frac{d}{k}[u^{-d}]L(u)^k.
$$ 
\end{theo}
\begin{proof}
About the expression of $F_d$, we have seen that $F_d=A_d-B_d$. Note that $A_d=L_0$ (because the marked bud can be turned into a leg of
weight~$0$) 
and $B_d=\sum_{i=1}^d L_i W_i$ (because the marked edge $e$ can have any weight
$i\in\{1..d\}$, and cutting $e$ in its middle yields 
two planted $d$-hypermobiles that are respectively
in $\cL_i$ and in $\cW_i$). 
For expressing the partial derivatives of $F_d$, we note that $x_k\frac{k}{d}\frac{\partial F_d}{\partial x_k}$ is the generating function of dark-rooted hypermaps with a dark root face of degree $d$, and an additional marked corner in an inner light face of degree $k$. By the bijection $\Phi_-$ this is also the generating function of $d$-weighted hypermobiles with a marked corner at a light square vertex of degree $k$, which is easily seen to be $x_k[u^d]\W(u)^k$.
A similar argument gives the expression for the partial derivative according to $y_k$. 
\end{proof}

\begin{remark}
The generating function $F_d$ of hypermaps can be specialized into a generating function of maps. More precisely, the class $\cG_d$ of corner-rooted maps of girth $d$ with a root-face of degree $d$, identifies with the set of hypermaps in $\cF_d$ such that every inner dark face has degree 2. Thus the generating $G_d$ of $\cG_d$ is obtained from $F_d$ by setting $y_2=1$ and $y_{\delta}=0$ for $\delta\neq 2$. Theorem~\ref{theo:count_plane_hypermaps} then gives the expressions of $G_d$ given in~\cite{BFgir} (upon observing that \eqref{eq:3} yields $W_i=L_{d-2-i}$, that is, $L_i=W_{d-2-i}$).
\end{remark}

For any sets $\Delta,\Delta'$, the generating function $F_{d,\Delta,\Delta'}$ of corner-rooted hypermaps in $\cF_d$ with inner light face having degree in $\Delta$ and inner dark face having degree in $\Delta'$ is obtained by setting $x_k=0\ \mathrm{for}\ k\notin \Delta,\ y_k=0\ \mathrm{for}\ k\notin \Delta'$. We point out that the generating function $F_{d,\Delta,\Delta'}$ is algebraic as soon as $\Delta,\Delta'$ are both finite (because only a finite number of auxiliary series $W_i,L_i$ are involved).
For instance, for $d=4$, $\Delta=\{4\}$ and $\Delta'=\{3\}$, we have
$$
F_{4,\{4\},\{3\}}=L_0-L_1W_1-L_2W_2-L_3W_3-L_4W_4,
$$ 
where the series $\{L_0,L_1,L_2,L_3,L_4,W_0,W_1,W_2,W_3,W_4\}$ are specified by
\begin{eqnarray*}
&&L_0=x_4(1+W_0)^3,\ L_1=W_1^3+2W_1W_2+W_3,\ L_2=W_1^2+W_2,\ L_3=W_1,\ L_4=1,\\
&&W_0=2y_3L_2L_3,\ W_1=y_3(2L_1L_3+L_2^2),\ W_2=2y_3L_1L_2,\ W_3=y_3L_1^2,\ W_4=2y_3L_1.
\end{eqnarray*}


\section{Recovering known bijections as specializations}\label{section:recovering-bijections}
In this section we show that the bijections described in \cite{BMSc00,BMSc02} can be recovered by specializing the bijections of Theorem~\ref{theo:bij_d_plane}, and the bijections described in \cite{BDG04,BDG07} can be recovered by specializing the master bijection $\Phi_0$ (in a way which can be thought of as the case $d=0$ of Theorem~\ref{theo:bij_d_plane}).

\subsection{The Bousquet-M\'elou Schaeffer bijection for bipartite maps}
Recall that a \emph{bipartite map} is a map whose vertices can be colored in black and white such that each edge connects a black vertex to a white vertex. A \emph{1-leg bipartite map} is a bipartite map with a marked vertex of degree $1$; this vertex is considered as black, hence it fixes the coloring of all vertices. Note that 1-leg bipartite maps are dual to dark-rooted hypermaps of outer degree $1$. In~\cite{BMSc02}, Bousquet-M\'elou and Schaeffer have given a bijection between 1-leg bipartite maps and so-called well-charged blossom trees. We show here that the bijection in~\cite{BMSc02} is equivalent (up to duality) to the case $d=1$ of the bijection of Theorem~\ref{theo:bij_d_plane}. 

A \emph{blossom tree} is a bipartite plane tree (with black and white vertices) with dangling half-edges. The dangling half edges at black and white vertices are called \emph{outbuds} and \emph{inbuds} respectively (the terminology in \cite{BMSc02} is actually \emph{buds} and \emph{leaves} but this is confusing in the present context).
A \emph{planted subtree} of a blossom tree $T$ is a subtree that can be obtained as one of the two components after cutting at the middle of an edge $e$ of $T$ (not at a bud). 
The extremity of $e$ in the chosen component is called the root-vertex of the planted subtree.
The \emph{charge} of a blossom tree or subtree is its number of inbuds minus its number of outbuds\footnote{This notion of charge is taken from~\cite{BMSc02} and is not related to the notion of charge (which constraints the cycle-lengths) to be introduced in Section~\ref{sec:bij-charged-maps}.}. A blossom tree is \emph{well-charged} if it has charge $1$ and every planted subtree has charge at most $1$ when its root-vertex is black, and at least $0$ when its root-vertex is white. A well-charged blossom tree is represented in Figure~\ref{fig:1leg}(b). 

We first show that well-charged blossom trees identify to $1$-weighted hypermobiles, see Figure~\ref{fig:1leg}(a)-(b). 
By definition the round vertices of $1$-weighted hypermobiles have weight 1 hence are leaves (i.e., vertices of degree 1).
Thus, forgetting the weights, a 1-weighted hypermobile identifies to a blossom tree by interpreting dark and light square vertices as black and white vertices, round vertices as outbuds, and buds as inbuds. Hence we define the \emph{charge} of a $1$-weighted hypermobile or of a planted $1$-hypermobile as its number of buds minus its number of round vertices.
An easy induction (using the same recursive decomposition as in Section~\ref{sec:count_plane_hypermaps}) ensures that a planted 1-hypermobile of root-weight $w$ such that the root-vertex is dark square (resp. light square) has charge $-w$ (resp. $w+1$). Thus the fact that the edge weights of 1-hypermobiles are positive for edges having a round endpoint and non-positive otherwise corresponds to the fact that the associated blossom tree is well-charged.
Thus well-charged blossom trees identify to $1$-weighted hypermobiles (if one starts from a well-charged blossom tree, the weights on the corresponding $1$-hypermobile are determined: each edge $e$ gets a weight $c-1$, where $c$ is the charge of the planted subtree rooted on the dark square endpoint of $e$).

\begin{figure}
\begin{center}
\includegraphics[width=\linewidth]{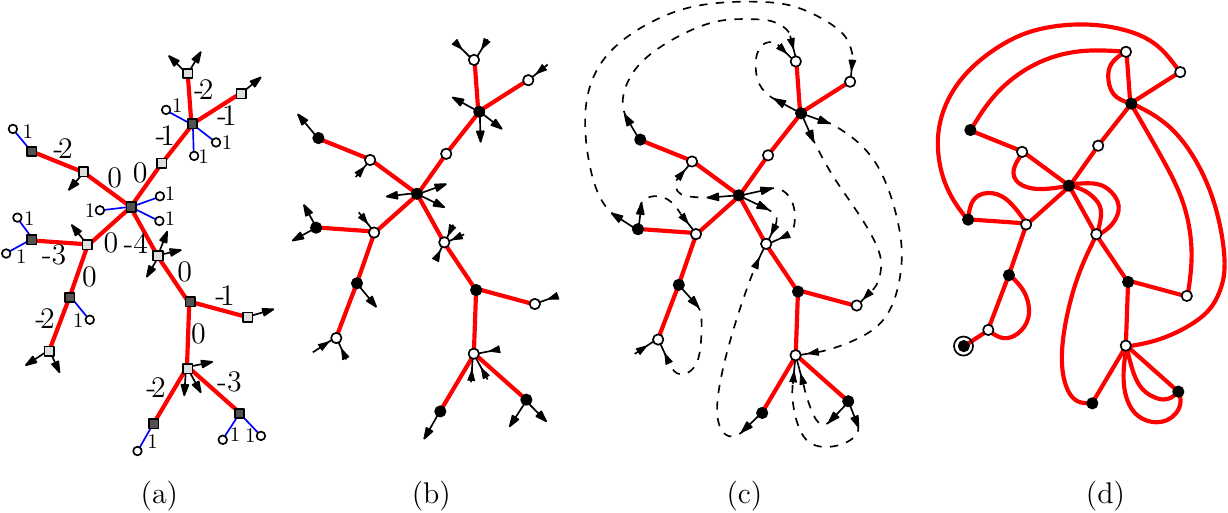}
\end{center}
\caption{(a) A 1-weighted hypermobile $T$. 
(b) The corresponding well-charged blossom tree $T'$ (inbuds and outbuds are represented by ingoing and outgoing arrows). (c)-(d) The closure of $T'$, which is the same as the closure of $T$.}
\label{fig:1leg}
\end{figure}

The bijection in~\cite{BMSc02} associates a 1-leg bipartite map to each well-charged blossom tree using a closure operation; see Figure~\ref{fig:1leg}(b)-(d). More precisely, for a well-charged blossom tree $T$ one considers the cyclic word $w_T$ obtained by walking clockwise around $T$ and encoding outbuds and inbuds by letters $a$ and $\ba$ respectively. Then, the cw-matching outbuds and inbuds of $T$ are joined into edges. Since the charge of $T$ is $1$, there remains 1 unmatched inbud. The result of the closure operation is therefore a 1-leg bipartite map, if one interprets the unmatched inbud as the leg leading to a black vertex of degree $1$. Moreover, it is clear that this closure operation of~\cite{BMSc02} applied to a well-charged blossom tree is equivalent to the closure operation of $\Psi_+$ (as formulated in Section~\ref{sec:alter}) applied to the corresponding $1$-weighted hypermobile. To summarize we obtain:

\begin{prop}
The blossom trees of~\cite{BMSc02} identify to $1$-weighted hypermobiles.
Under this identification the bijection of~\cite{BMSc02} is the same as the case $d=1$ of the bijection of Theorem~\ref{theo:bij_d_plane}. 
\end{prop}

\subsection{The Bousquet-M\'elou Schaeffer bijection for constellations}
For any fixed $p\geq 2$, we call \emph{$p$-constellation} a (planar) hypermap where the degree of each dark face is $p$ and the degree of each light face is a multiple of $p$ (these maps encode certain factorizations in the symmetric group; see~\cite{ZvLa97}). 
In \cite{BMSc00}, Bousquet-M\'elou and Schaeffer have given a bijection 
between dark-rooted $p$-constellations and so-called $p$-Eulerian trees.
 We show here that the bijection in~\cite{BMSc00} is equivalent to the case $d=p$ of the bijection of Theorem~\ref{theo:bij_d_plane} applied to $p$-constellations. Before discussing the equivalence, we show that $p$-constellations have ingirth $p$.
\begin{lem}
A $p$-constellation has \igirth\ $p$.
\end{lem}
\begin{proof}
Let $K$ be a $p$-constellation, and let $C$ be an inward cycle of $K$.
Clearly the length of $C$ equals $A-B$, where $A$ is the total degree
of all light faces inside $C$ and $B$ is the total degree of all
dark faces inside $C$. Since all faces (dark or light) have degree a multiple of $p$,
the length of $C$ is a multiple of $p$, hence is at least $p$. 
\end{proof}

We now explicit the equivalence of the bijection in~\cite{BMSc00} with the the case $d=p$ of Theorem~\ref{theo:bij_d_plane} applied to $p$-constellations.
A \emph{$p$-Eulerian tree} is a bipartite plane tree (with black and white vertices) satisfying:
\begin{compactitem}
\item Each black \emph{inner node} (non-leaf vertex) has degree $p$ and has either $n=1$ or $n=2$ neighbors that are inner nodes. This black vertex is said to be of \emph{type $n\in\{1,2\}$}. 
\item Each white inner node has degree of the form $p\,i$ with $i\geq 1$, and it has $i-1$ neighbors that are black inner nodes of type $1$. 
\end{compactitem}

\begin{figure}
\begin{center}
\includegraphics[width=\linewidth]{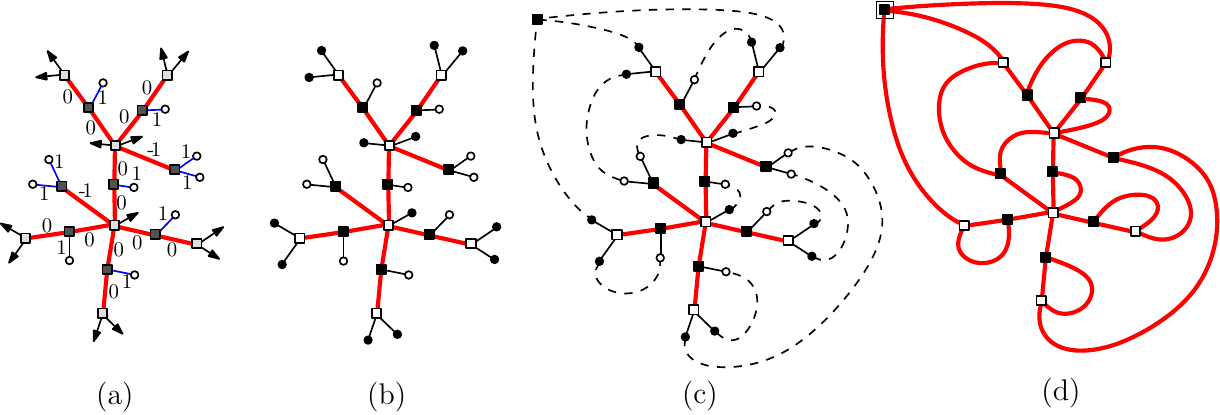}
\end{center}
\caption{(a) A $p$-weighted hypermobile $T$ associated with a dark-rooted $p$-constellation, after all weights have been divided by $p=3$. (b) The corresponding $p$-Eulerian tree $T'$. (c)-(d) The closure of $T'$, which is the same as the closure of $T$.}
\label{fig:eulerian_tree}
\end{figure} 

We first show that $p$-Eulerian trees identify with the $p$-weighted hypermobiles corresponding to $p$-constellations. 
A $p$-weighted hypermobile $T$ corresponds to a $p$-constellation if all dark square vertices have degree $p$ and all light square vertices have degree multiple of $p$. In this case, by Remark~\ref{rk:weight-multiple}, all the edge weights of $T$ are multiple of $p$, and we denote by $T'$ the weighted-hypermobile obtained by dividing every weight by $p$. 
In $T'$ the weight of each round vertex is $1$, the weight of each dark square vertex is $p-2$, and the weight of each light square vertex of degree $p\,i$ is $1-i$.
Since round vertices have weight 1 they are leaves. 
Since a dark square vertex has degree $p$ and weight $p-2$, it has either $p-1$ round neighbors and $n=1$ light square neighbor (and the edge to the
light square neighbor has weight $-1$) or it has $p-2$ round neighbors and $n=2$ light square neighbors (and the edges to the light square neighbors have weight $0$). This dark square vertex is said to be of \emph{type $n\in\{1,2\}$}. 
Since a light square vertex of degree $p\,i$ has weight $1-i$, it has $i-1$ dark square neighbors of type $1$. 
Thus $p$-weighted hypermobiles corresponding to $p$-constellations identify with $p$-Eulerian trees if one interprets buds as black leaves, 
round vertices as white leaves, 
dark square vertices as black inner nodes, 
and light square vertices as white inner nodes. Indeed, if one starts from a $p$-Eulerian tree, the corresponding hypermobile is obtained by giving a weight $w$ to each edge $e=(u,v)$ connecting a black inner node $u$ to a white inner node $v$, where $w=-1$ if $u$ has type $1$ and $w=0$ if $u$ has type $2$. 

The bijection in~\cite{BMSc00} associates a dark-rooted $p$-constellation 
with such a tree $T$ using a closure operation (see Figure~\ref{fig:eulerian_tree}(b)-(d)). More precisely, a counterclockwise walk around
the outer face of $T$ sees a succession of black leaves and white leaves, and we consider the cw-matching when black leaves are interpreted as 
 letters $a$, and white leaves as letters $\ba$. The pairs of cw-matching leaves are joined by edges. 
It can be shown that a $p$-Eulerian tree has an excess of $p$ black leaves over white leaves. Hence after the cw-matching, there remain $p$ unmatched black leaves (all in the outer face) and these are merged into a black vertex of degree $p$ taken as the root-vertex. This yields a vertex-rooted bipartite map where black vertices have degree $p$ and white vertices have degree multiple of $p$. Hence the dual of the obtained bipartite map is a dark-rooted $p$-constellation.

It is clear that the closure mapping (as formulated in Section~\ref{sec:alter}) applied to a $p$-weighted hypermobile of a $p$-constellation is equivalent
to the closing mapping of~\cite{BMSc00} applied to the corresponding $p$-Eulerian
tree. To summarize we obtain:
\begin{prop}
For $p\geq 2$, the $p$-Eulerian trees of~\cite{BMSc00} 
identify to $p$-weighted hypermobiles that are associated with dark-rooted $p$-constellations.
Under this identification the bijection of~\cite{BMSc00} is the same as 
the case $d=p$ of the bijection of Theorem~\ref{theo:bij_d_plane} applied to $p$-constellations.
\end{prop}

\begin{remark} Since the two bijections are the same, the inverse mappings from constellations to decorated trees also coincide. In both cases, the decorated tree is recovered as the complemented dual of a forest: in our case the forest $F$ is made of the directed edges of the canonical $p$-weighted orientation, while in~\cite{BMSc00} the forest $F'$ is the so-called \emph{rank-forest} (see~\cite{BMSc00}, in particular Section~5.2 and Proposition 6.2).
Our rules to obtain the $p$-weighted hypermobile from $F$ can be checked to coincide with the rules given in~\cite{BMSc00} to obtain the $p$-eulerian tree from $F'$. So $F$ is the same as $F'$.
\end{remark}

\subsection{The Bouttier Di Francesco Guitter bijections for Eulerian maps}
In~\cite{BDG04}, Bouttier, Di Francesco and Guitter have given a bijection for
vertex-rooted hypermaps.
In~\cite{BDG07} this bijection was generalized to vertex-rooted hypermaps with 
some ``blocked edges''. We show here that these bijections can be obtained
as specializations of the master bijection $\Phi_0$ (and can be thought of as ``the case $d=0$'' of Theorem~\ref{theo:bij_d_plane}).

\begin{figure}[h!]
\begin{center}
\includegraphics[width=.5\linewidth]{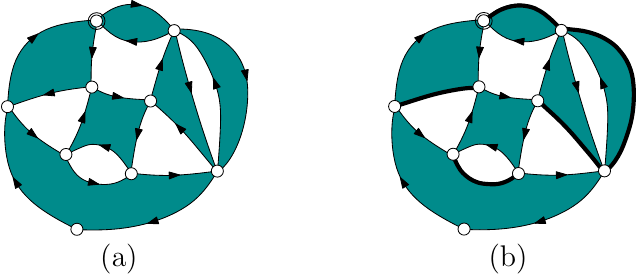}
\end{center}
\caption{(a) A vertex-rooted hypermap endowed with its dark-light orientation. (b)
The same hypermap with some blocked edges, endowed with its dark-light hyperorientation
(blocked edges are 0-way, other edges are 1-way).}
\label{fig:BDG1}
\end{figure}

\begin{figure}[h!]
\begin{center}
\includegraphics[width=.5\linewidth]{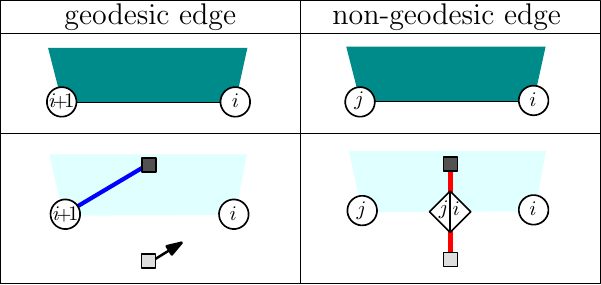}
\end{center}
\caption{Local rules applied to the edges of a vertex-rooted hypermap, according to the distance-labelling.}
\label{fig:local_rule_BDG}
\end{figure}

Let $M$ be a vertex-rooted hypermap, and let $v_0$ be its root-vertex. The hyperorientation $\Om$ of $M$ such that each edge has a dark face on its right is called the \emph{dark-light hyperorientation} of $M$; see Figure~\ref{fig:BDG1}(a). We give to each vertex $v$ of $M$ the \emph{label} $\ell(v)$ equal to the length of a shortest directed path of $\Om$ from $v_0$ to $v$. 
For each edge $e=(u,v)$ (oriented from $u$ to $v$ in $\Om$), the labels of $u$ and $v$ clearly satisfy $\ell(v)\leq\ell(u)+1$. 
We call $e$ \emph{geodesic} if $\ell(v)=\ell(u)+1$ and \emph{non-geodesic} otherwise. 
One associates with $M$ a hypermobile $T$ without buds, but with labels, by applying to each edge the rule indicated in Figure~\ref{fig:local_rule_BDG}. More precisely, $T$ has labels on the round vertices, called \emph{vertex labels}, and on each side of any edge incident to a light square vertex, called \emph{edge labels}.
Moreover, it is easy to see that $T$ satisfies the following properties:
\begin{compactitem}
\item Vertex labels are positive and edge labels are non-negative.
\item In clockwise order around a dark square vertex, any two consecutive labels $\ell,\ell'$ satisfy $\ell'\leq \ell$ if $\ell,\ell'$ are edge-labels on the same edge, $\ell'=\ell+1$ if $\ell'$ is a vertex-label, and $\ell'=\ell$ in the other cases.
\item In clockwise order around a light square vertex, any two consecutive edge-labels $\ell,\ell'$ satisfy $\ell'\geq \ell$ if $\ell,\ell'$ are on the same edge, and $\ell'\leq \ell$ otherwise.
\end{compactitem}
We call \emph{well-labeled mobile} a labeled hypermobile satisfying these conditions; see Figure~\ref{fig:BDG2} for an example.

\begin{figure}
\begin{center}
\includegraphics[width=\linewidth]{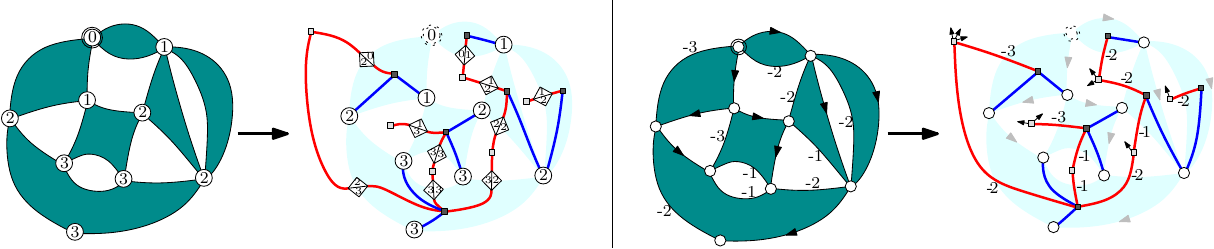}
\end{center}
\caption{Left: the bijection of~\cite{BDG04} between vertex-rooted hypermaps and well-labeled hypermobiles.
 Right: the bijection seen as a specialization of the master bijection $\Phi_0$.}
\label{fig:BDG2}
\end{figure}

Bouttier, Di Francesco and Guitter have shown in~\cite{BDG04} that applying the local rules of Figure~\ref{fig:local_rule_BDG} gives a bijection between vertex-rooted hypermaps and well-labeled mobiles. Now we explain how to reformulate the distance-labelling and the well-labeled mobiles, and the connection with the master bijection $\Phi_0$. 

First, we show that the distance-labelling can be encoded as a weighted hyperorientation; see Figure~\ref{fig:BDG2} right part. 
We call \emph{geodesic hyperorientation} of $M$ the weighted hyperorientation such that each geodesic edge is 1-way with weight 0, and 
 each non-geodesic edge $e=(u,v)$ (with the dark face on its right) is 0-way with weight $\ell(v)-\ell(u)-1$. 
The geodesic hyperorientation satisfies the following conditions:
\begin{compactitem}
\item The weight of an edge $e$ is $0$ if $e$ is 1-way, and is negative if $e$ is 0-way.
\item The weight of a vertex is $0$, and the weight of a face $f$ (light or dark) is $-\deg(f)$.
\end{compactitem}
We call \emph{$0$-weighted} a hyperorientation satisfying these conditions.
Note that the geodesic hyperorientation has two additional properties: it is accessible from $v_0$ and it is acyclic; hence it is in~$\cHz$. 
\begin{lem}\label{lem:0ori}
A vertex-rooted hypermap $M$ has a unique $0$-weighted hyperorientation in $\cHz$; it is its geodesic hyperorientation. 
\end{lem}
\begin{proof}
Let $M$ be a vertex-rooted hypermap, and let $v_0$ be its the root-vertex.
We call \emph{admissible labelling} of $M$ a labelling $L$ of its vertices (each vertex $v$ has a label $L(v)\in\mathbb{Z}$) 
such that $L(v_0)=0$ and for each edge $e=(u,v)$ (with the dark face on its right) $L(v)\leq L(u)+1$. One can associate to such a labelling
a $0$-weighted hyperorientation exactly in the same way as we have done
for the distance labelling. And this actually gives a bijection between
admissible labellings and $0$-weighted hyperorientations of $M$. 
We have already seen that the $0$-weighted hyperorientation associated
with the distance-labelling is in $\cHz$. 
Note that any admissible labelling $L$ satisfies $L(v)\leq \ell(v)$ for all vertices (because the labels increase
by at most $1$ along each edge of a geodesic path ending at $v$). If $L$ is not equal to $\ell$, 
consider a vertex $v$ such that $L(v)<\ell(v)$ (note that $v\neq v_0$) 
and $L(v)$ is the smallest possible. 
Assume there is a  neighbor $v'$ of $v$ such that $L(v')<L(v)$, that is, $L(v')=L(v)-1$. Since $\ell(v')\geq \ell(v)-1$
we reach the contradiction that $L(v')<\ell(v')$.  
 Hence $v$ is not accessible from $v_0$ in
 the $0$-weighted hyperorientation associated with $L$, so the hyperorientation
is not in $\cHz$. 
\end{proof}

Second, we show that the well-labeled mobiles can be encoded as weighted (unlabeled) hypermobiles; see Figure~\ref{fig:BDG2} bottom part. 
From a well-labeled mobile $T$ we construct a weighted hypermobile $\theta(T)$ as follows. 
Give weight $0$ to each edge incident to a round vertex, and give weight $\ell-r-1$ to each edge $e$ incident to a light square vertex $u$, where $\ell$ and $r$ are the edge labels on the left side and right side of $e$ looking from $u$.
In each corner $c$ of $T$ at a light square vertex $u$ between two consecutive edges $e,e'$ (in clockwise order), insert $r-\ell$ buds in the corner $c$ where $r$ is the edge label on the right side of $e$ (looking from $u$), and $\ell$ is the edge label on the left side of $e'$.
Then delete all the labels. 
The obtained hypermobile $\theta(T)$ satisfies the following conditions:
\begin{compactitem}
\item Edges incident to a round vertex have weight $0$ (hence round vertices have weight $0$), while edges incident to a light square vertex have negative weight. 
\item Each square vertex $v$ (light or dark) has weight $-\deg(v)$.
\end{compactitem}
We call \emph{$0$-weighted} a hypermobile satisfying these conditions. Clearly $\theta$ is a bijection between well-labeled mobiles and $0$-weighted hypermobiles. We can now show that the bijection of~\cite{BDG07} can be obtained as a specialization of $\Phi_0$; see Figure~\ref{fig:BDG2}.
\begin{prop}\label{prop:BDG1}
The master bijection $\Phi_0$ yields a bijection between vertex-rooted
hypermaps and $0$-weighted hypermobiles. This bijection coincides with the
Bouttier Di Francesco Guitter bijection, up to the identification of well-labeled
mobiles with $0$-weighted hypermobiles.
\end{prop}
\begin{remark} The bijection~\cite{BDG07}, as reformulated in Proposition~\ref{prop:BDG1}, can be thought of as the ``case $d=0$'' of Theorem~\ref{theo:bij_d_plane}. Indeed, one can think of a vertex-rooted hypermap as a dark-rooted hypermap of degree 0. Then the definition of $0$-weighted hyperorientation coincides with the case $d=0$ of $d$-weighted hyperorientations given in Section~\ref{sec:bij_plane}, except that the weight 0 are authorized on 1-way edges instead of on 0-way edge. Also the definition of $0$-weighted hypermobile coincides with the case $d=0$ of $d$-weighted hypermobile given in Section~\ref{sec:bij_plane}, except that the weight 0 is authorized on edges incident to round vertices instead of on edges incident to light square vertices.
\end{remark}

\begin{proof}
It is easy to prove that $0$-weighted hypermobiles have excess $0$.
Hence the master bijection $\Phi_0$ clearly yields a bijection between $0$-weighted hypermobiles and $0$-weighted vertex-rooted hyperorientations in $\cHz$.
By Lemma~\ref{lem:0ori}, the latter family identifies to the family of vertex-rooted hypermaps.

Now one easily verifies from Figure~\ref{fig:local_rule_BDG} that, if $M$ is a vertex-rooted hypermap and $T$ is the associated well-labeled mobile, then $\theta(T)$ is obtained from the geodesic hyperorientation by applying the local rules of Figure~\ref{fig:local-rule-hyperori}, that is, by applying~$\Phi_0$.
\end{proof}


We now discuss the bijection given in~\cite{BDG07}, which is an extension of the bijection in \cite{BDG04} for vertex-rooted hypermaps with \emph{blocked edges}. Let $M$ be a vertex-rooted hypermap, with $v_0$ the root-vertex, and let $X$ be a subset of the edges of $M$ called \emph{blocked edges}. 
Let $\Om_X$ be the hyperorientation of $M$ where the edges in $X$ are 0-way and the edges not in $X$ are 1-way. The subset $X$ is called \emph{admissible} if $\Om_X$ is accessible from $v_0$. A pair $(M,X)$, with $M$ a vertex-rooted hypermap and $X$ an admissible subset of edges of $M$, is shortly called a (vertex-rooted) hypermap with blocked edges.

The bijection of~\cite{BDG07} proceeds similarly as the one above (which corresponds to the case $X=\emptyset$).
Namely, we give to each vertex $v$ a label equal to the minimal length of the directed paths in $\Om_X$ from $v_0$ to $v$. 
We call $e=(u,v)$ (with the dark face on the right of $e$) \emph{geodesic} if it is not blocked and $\ell(v)=\ell(u)+1$, and \emph{non-geodesic} otherwise. 
Then a labeled mobile is associated with $(M,X)$ by applying the rule of Figure~\ref{fig:local_rule_BDG}, and marking as \emph{blocked} the edges of the mobile corresponding to blocked edges of $M$.
The associated labeled mobiles, called \emph{generalized well-labelled mobiles} satisfy the same conditions as well-labelled mobiles, with the only difference that there can be some blocked edges incident to light square vertices and that the difference between the edge-labels on the two sides of a blocked edge is arbitrary.

\begin{figure}
\begin{center}
\includegraphics[width=\linewidth]{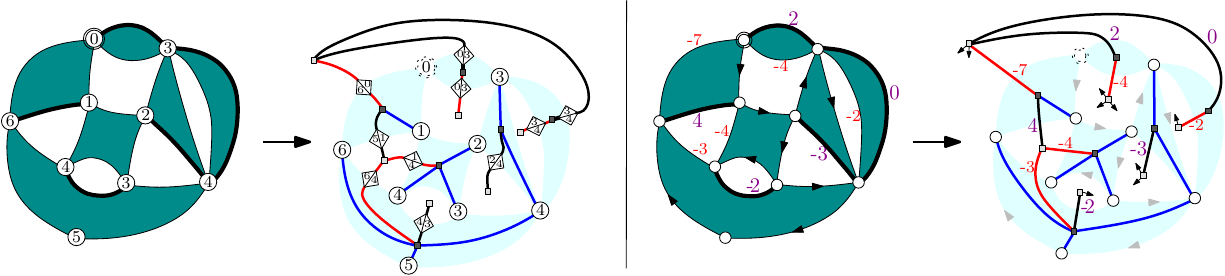}
\end{center}
\caption{Left: The bijection of \cite{BDG07} between vertex-rooted hypermaps with blocked edges 
 and generalized well-labeled mobiles.
 Right: the same bijection seen as a specialization of $\Phi_0$.}
\label{fig:BDG3}
\end{figure}

As above we can encode the distance-labelling by a weighted hyperorientation. More precisely, we define the geodesic hyperorientation as follows:
each geodesic edge is oriented 1-way and given weight $0$, 
each non-geodesic edge $e=(u,v)$ (with the dark face on the right of $e$) 
is oriented 0-way and given weight $\ell(v)-\ell(u)-1$. The geodesic hyperorientation satisfies:
\begin{compactitem}
\item The weight of an edge $e$ is $0$ if $e$ is directed, and is negative if $e$ is non-blocked and 0-way.
\item The weight of a vertex is $0$, and the weight of a face $f$ (light or dark) is $-\deg(f)$.
\end{compactitem}
A hyperorientation satisfying these conditions is called a \emph{generalized $0$-weighted hyperorientation}. 
The geodesic hyperorientation has two additional properties: it is accessible from $v_0$ and it is acyclic; hence it is in $\cHz$. 
\begin{lem}\label{lem:0ori-gen}
A vertex-rooted hypermap with blocked edges $(M,X)$ has a unique generalized $0$-weighted hyperorientation in $\cHz$; it is the geodesic hyperorientation. 
\end{lem}
\begin{proof}
The proof is similar to the proof of Lemma~\ref{lem:0ori}. Let $(M,X)$ be a vertex-rooted hypermap with blocked edges, and let $v_0$ be its root-vertex. 
We call \emph{admissible labelling} of $M$ a labelling $L$ of its vertices such that $L(v_0)=0$ and for each non-blocked edge $e=(u,v)$ (with the dark face on the right of $e$) $L(v)\leq L(u)+1$. As before there is a bijection between the admissible labellings and the generalized $0$-weighted hyperorientations of $M$. Moreover any admissible labelling $L$ which is not the distance-labelling $\ell$ is associated with a hyperorientation which is not accessible, hence not in $\cHz$. 
\end{proof}

We call \emph{generalized $0$-weighted hypermobile} a hypermobile with some 
marked edges incident to light-square vertices, such that the following conditions hold:
\begin{compactitem}
\item Edges incident to a round vertex have weight $0$ (hence each round vertex has weight $0$), and non-marked edges incident to a light square have negative weight. 
\item Each square vertex $v$ (light or dark) has weight $-\deg(v)$.
\end{compactitem}
Similarly as in the case without blocked edges, generalized well-labeled mobiles can be identified to generalized $0$-weighted hypermobiles. 
We now state how the bijection of~\cite{BDG07} can be obtained as a specialization of $\Phi_0$; see Figure~\ref{fig:BDG3}.
\begin{prop}\label{prop:BDG2}
The master bijection $\Phi_0$ yields a bijection between vertex-rooted
hypermaps with blocked edges and generalized $0$-weighted hypermobiles.
This bijection coincides with the bijection of~\cite{BDG07}, up to the identification of generalized well-labeled mobiles with generalized $0$-weighted hypermobiles.
\end{prop}
\begin{proof}
The proof is very similar to the one of Proposition~\ref{prop:BDG1} and is left to the reader. 
\end{proof}


\medskip
\section{Bijections for hypermaps with general cycle-length constraints}\label{sec:bij-charged-maps}
In this section we consider a far-reaching generalization of the girth constraints considered in Section~\ref{sec:bij_plane}, and obtain bijections for hypermaps satisfying these constraints.

We call \emph{charge function} $\si$ of a hypermap $H$ the assignment of a real number $\si(a)$, called \emph{charge}, to each vertex and face $a$ of $H$. 
The pair $(H,\si)$ is called a \emph{charged hypermap}. We call \emph{total charge}, and denote it by $\si_\tot$, the sum of all the charges of the hypermap. 
We will now define some cycle-length constraints on charged hypermaps.
A \emph{light region} of $H$ is a proper subset $R$ of the faces of $H$ such that any face in $R$ sharing an edge with a face not in $R$ is light. 
We say that an edge or a vertex is \emph{strictly inside} a light region $R$ if all its incident faces are in $R$. 
We denote 
$$\si(R)=\sum_{f\textrm{ face inside } R}\si(f)+\sum_{v\textrm{ vertex strictly inside } R}\si(v).$$
The \emph{boundary} of a light region $R$ is the set of edges incident both to a face in $R$ and to a face not in $R$. We denote by $\R$ the boundary of $R$ and by $|\R|$ its cardinality.

\begin{Def}\label{def:girth-condition}
Let $H$ be a hypermap, and let $\si$ be a charge function. 
If the hypermap $H$ is dark-rooted (resp. light-rooted, vertex-rooted), we say that $H$ satisfies the \emph{$\si$-girth condition} if every light region $R$ satisfies $|\R|\geq \si(R)$ with strict inequality if all the outer vertices are strictly inside $R$ (resp. if one of the outer edges is strictly inside $R$, if the root-vertex is strictly inside $R$).
\end{Def}

Various girth constraints can be realized as a $\si$-girth condition by choosing an appropriate charge function $\si$; examples are given in Section~\ref{sec:application-charged-maps}. 

We will now characterize the $\si$-girth condition by the existence of certain hyperorientations. 
\begin{Def} 
Let $(H,\si)$ be a charged hypermap.
If $H$ is light-rooted or vertex-rooted, we call \emph{$\si$-weighted hyperorientation} of $H$ a weighted hyperorientation such that: 
\begin{itemize}
\item[(i)] the weight of 1-way edges is positive, and the weight of 0-way edges is non-positive,
\item[(ii)] the weight of every light face $f$ is $\si(f)-\deg(f)$,
\item[(iii)] the weight of every inner dark face $f$ is $-\si(f)-\deg(f)$,
\item[(iv)] the weight of every vertex $v$ is $\si(v)$.
\end{itemize}
If $H$ is dark-rooted, we call \emph{$\si$-weighted hyperorientation} of $H$, a weighted hyperorientation satisfying (i), (ii), (iii) and
\begin{itemize}
\item[(iv')] the weight of every inner vertex $v$ is $\si(v)$, the weight of every outer vertex $v$ is $\si(v)+1$, the weight of every outer edge is 1, and the weight of the dark outer face $f_0$ is $-\si(f_0)$.
\end{itemize}
\end{Def}

We now state the key result for dark-rooted hypermaps. We say that a charge function $\si$ \emph{fits} a dark-rooted hypermap $H$ if $H$ satisfies the $\si$-girth condition, the charge of every inner vertex is positive, the charge of every outer vertex is 0, the charge of the dark outer face $f_0$ is $-\deg(f_0)$, and $\si_\tot=0$. 
\begin{theo}\label{theo:shifted-orientation-dark}
Let $H$ be a dark-rooted hypermap, and let $\si$ be a charge function.
The hypermap $H$ admits a $\si$-weighted hyperorientation in $\cHm$ if and only if $\si$ fits $H$ and the outer face of $H$ is simple. Moreover, in this case $H$ admits a unique $\si$-weighted hyperorientation in $\cHm$.
\end{theo}
It will be shown in Section~\ref{sec:application-charged-maps} that Theorem~\ref{theo:plane_dori} is a special case of Theorem~\ref{theo:shifted-orientation-dark} corresponding to a particular choice of charge function. We now state the analogous result for light-rooted and vertex-rooted hypermaps. We say that a charge function $\si$ \emph{fits} a light-rooted hypermap $H$ if $H$ satisfies the $\si$-girth condition, the charge of every vertex is positive, the charge of the light outer face $f_0$ is $\deg(f_0)$, and $\si_\tot=0$. 
\begin{theo}\label{theo:shifted-orientation-light}
Let $H$ be a light-rooted hypermap, and let $\si$ be a charge function. 
The hypermap $H$ admits a $\si$-weighted hyperorientation in $\cHp$ if and only if $\si$ fits $H$. Moreover, in this case $H$ admits a unique $\si$-weighted hyperorientation in $\cHp$. 
\end{theo}
 
We say that a charge function $\si$ \emph{fits} a vertex-rooted hypermap if $H$ satisfies the $\si$-girth condition, the charge of every non-root vertex is positive, the charge of the root-vertex is 0, and $\si_\tot=0$. 
\begin{theo}\label{theo:shifted-orientation-0}
Let $H$ be a vertex-rooted hypermap, and let $\si$ be a charge function. 
The hypermap $H$ admits a $\si$-weighted hyperorientation in $\cH_0$ if and only if $\si$ fits $H$. Moreover, in this case $H$ admits a unique $\si$-weighted hyperorientation in $\cH_0$. 
\end{theo}
The proof of Theorems~\ref{theo:shifted-orientation-dark},~\ref{theo:shifted-orientation-light} and~\ref{theo:shifted-orientation-0} are postponed to Section~\ref{sec:proofs}.

We will now obtain bijections for charged hypermaps using the master bijections $\Phi_-$, $\Phi_+$ and $\Phi_0$. We call \emph{fittingly charged hypermap} a charged hypermap such that $\si$ fits $H$. We call \emph{consistently-weighted} a hypermobile with weights in $\br$ such that the weights of edges incident to round vertices are positive, while the weights of edges incident to light square vertices are non-positive. We will now show that fittingly charged hypermaps are in bijection with consistently-weighted hypermobiles. 

We call \emph{charge} of a vertex $u$ of a hypermobile the quantity 
\begin{itemize}
\item $w(u)$ if $u$ is a round vertex,
\item $w(u)+\deg(u)$ if $u$ is a light square vertex,
\item $-w(u)-\deg(u)$ if $u$ is a dark square vertex.
\end{itemize}
We now relate the excess of a hypermobile to the charges.
\begin{lem} 
The excess of a hypermobile of vertex-set $V$ is $\ds -\sum_{v\in V}\si(v)$, where $\sigma(v)$ denotes the charge
of vertex $v$.
\end{lem}

\begin{proof}
Let $T$ be a hypermobile. Let $R$, $D$, and $L$ be respectively the sets of round vertices, dark square vertices, and light square vertices of $T$. Let $b$, $e_R$, and $e_L$ be respectively the number of buds, edges incident to a round vertex, and edges incident to a light square vertex. 
By definition, the excess of $T$ is $e_R-b$. If we denote by $w(u)$ and $\si(u)$ the weight and charge of a vertex $u$, we get 
 \begin{eqnarray*}
\sum_{u\in R}w(u)&=&\sum_{u\in R}\si(u),\\
\sum_{u\in L}w(u)&=& -e_L-b+\sum_{u\in L}\si(u),\\
\sum_{u\in D}w(u)&=& -e_R-e_L-\sum_{u\in D}\si(u).\\
\end{eqnarray*}
Plugging these relations in $\sum_{u\in R}w(u)+\sum_{u\in L}w(u)=\sum_{u\in D}w(u)$ gives $e_R-b=-\sum_{v\in R\cup L\cup D}\si(v)$, as wanted.
\end{proof}

A hyperorientation is called \emph{consistently-weighted} if the weight of every 1-way edge is a positive real number, and the weight of every 0-way edge is a non-positive real number. 
By Theorem~\ref{theo:shifted-orientation-dark}, the set of fittingly charged dark-rooted hypermaps such that the outer face is simple identifies with the set of consistently-weighted hyperorientations in $\cHm$ such that the weight of every outer edge is 1. Moreover, 
\begin{itemize}
\item the charge of an inner vertex $v$ of $H$ is $\si(v)=w(v)$, 
\item the charge of an inner light face $f$ of $H$ is $\si(f)=w(f)+\deg(f)$,
\item the charge of an inner dark face $f$ of $H$ is $\si(f)=-w(f)-\deg(f)$.
\end{itemize}
Hence by applying the master bijection $\Phi_-$ and keeping track of the parameter-correspondences we obtain:
\begin{theo}\label{theo:bij-shifted-dark}
The mapping $\Phi_-$ gives a bijection between the set of fittingly charged dark-rooted hypermaps such that the outer face is simple, and the set of consistently-weighted hypermobiles with negative excess.
Moreover, each light (resp. dark) inner face of degree $\de$ and charge $x$ of the hypermap corresponds to a light (resp. dark) square vertex of degree $\de$ and charge $x$ of the associated charged hypermobile. Also each inner vertex of charge $x$ in the hypermap corresponds to a round vertex of charge $x$ of the associated charged hypermobile. Lastly, the outer degree of the hypermap corresponds to minus the excess of the associated hypermobile.
\end{theo}
We will see below (see Lemma~\ref{lem:sigma-d-condition}) that Theorem~\ref{theo:bij_d_plane} corresponds to a special case of Theorem~\ref{theo:bij-shifted-dark}. We now consider light-rooted hypermaps. Similarly as above, using Theorem~\ref{theo:shifted-orientation-light} and applying the master master bijection $\Phi_+$ we obtain:
\begin{theo}\label{theo:bij-shifted-light}
The mapping $\Phi_+$ gives a bijection between the set of fittingly charged light-rooted hypermaps and the set of consistently-weighted hypermobiles with positive excess.
Moreover, each light (resp. dark) inner face of degree $\de$ and charge $x$ of the hypermap corresponds to a light (resp. dark) square vertex of degree $\de$ and charge $x$ of the associated hypermobile. Also each vertex of charge $x$ in the hypermap corresponds to a round vertex of charge $x$ of the associated hypermobile. Lastly, the outer degree of the hypermap corresponds to the excess of the associated hypermobile.
\end{theo}

Similarly, using Theorem~\ref{theo:shifted-orientation-0} and applying the master master bijection $\Phi_0$ we obtain:
\begin{theo}\label{theo:bij-shifted-0}
The mapping $\Phi_0$ gives a bijection between the set of fittingly charged vertex-rooted hypermaps and the set of consistently-weighted hypermobiles with excess zero.
Moreover, each light (resp. dark) face of degree $\de$ and charge $x$ of the hypermap corresponds to a light (resp. dark) square vertex of degree $\de$ and charge $x$ of the associated hypermobile. Also each non-root vertex of charge $x$ in the hypermap corresponds to a round vertex of charge $x$ of the associated hypermobile. 
\end{theo}

We will use Theorem~\ref{theo:bij-shifted-dark} and~\ref{theo:bij-shifted-light} in the next section to count annular hypermaps. In the remaining part of this section we give a general lemma about $\si$-girth conditions, and then explain how to derive Theorem~\ref{thm:special-charged-maps} stated in the introduction from Theorem~\ref{theo:bij-shifted-0}.

A light region $R$ is said to be \emph{connected} (resp. \emph{simply connected}) if the union of the faces in $R$, and the edges and vertices strictly inside $R$ is a connected (resp. simply connected) subset of the sphere. For instance, the light region in Figure~\ref{fig:inward} is simply connected. When we consider the simple connectedness of a light region containing the outer face, we think of the outer face simply as a marked face of a hypermap on the sphere, so that this face is finite and simply connected. 

\fig{width=.4\linewidth}{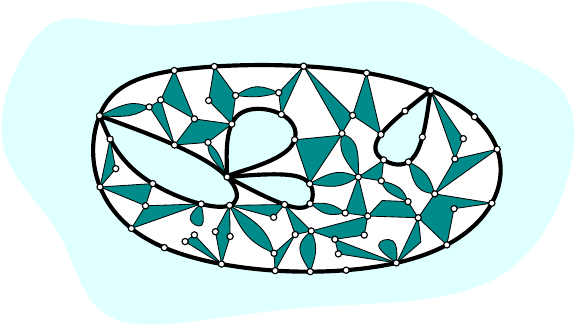}{A simply connected light region.}

The following lemma shows that the $\si$-girth condition can be stated as a condition on simply connected light regions whenever $\si_\tot=0$ (hence in particular when $\si$ is fitting).
\begin{lemma}\label{lem:simply-connected-sufficient}
Let $H$ be a dark-rooted, light-rooted or vertex-rooted hypermap, and let $\si$ be a charge function such that $\si_\tot=0$.
The hypermap satisfies the $\si$-girth condition if and only if the inequalities and strict inequalities stated in Definition~\ref{def:girth-condition} hold for every \emph{simply connected} light region $R$.
\end{lemma}

\begin{remark} We point out that, in general (even if $\si_\tot=0$), the $\si$-girth condition might not be satisfied even if the inequalities and strict inequalities stated in Definition~\ref{def:girth-condition} hold for every light region $R$ whose boundary is a simple cycle. 
\end{remark}

\begin{proof}
We treat the case in which $H$ is dark-rooted (the cases of light-rooted and vertex-rooted hypermaps are proved similarly).
We suppose that any simply connected light region $R$ satisfies $|\R|\geq \si(R)$ with strict inequality if all of the outer vertices are strictly inside $R$. We want to prove that the same property holds for any light region $R$. 
Suppose this does not hold, and take $R_0$ a light region such that the property does not hold and $|\R_0|$ is minimal. If $R_0$ is not connected, then $R_0$ is the disjoint union of two light regions $R_1$ and $R_2$
(with $|\R_1|\geq 1$ and $\R_2|\geq 1$). We have $|\R_0|=|\R_1|+|\R_2|$ and $\si(R_0)=\si(R_1)+\si(R_2)$ which contradicts the minimality of $|\R_0|$ (note that if the outer vertices are strictly inside $R_0$, then they are strictly inside either $R_1$ or $R_2$). Now if $R_0$ is connected but not simply connected, then $R_0$ is the intersection of two light regions $R_1$ and $R_2$ such that every face of $H$ is inside $R_1$ or $R_2$ and every vertex of $H$ is strictly inside $R_1$ or $R_2$. Hence, $\si(R_0)=\si(R_1)+\si(R_2)-\si_\tot=\si(R_1)+\si(R_2)$. Moreover, $\R_0$ is the disjoint union of $\R_1$ and $\R_2$. Thus $|\R_0|=|\R_1|+ |\R_2|$ and this again contradicts the minimality of $|\R_0|$ (note that if the outer vertices are strictly inside $R_0$, then they are strictly inside both $R_1$ and $R_2$). Thus $R_0$ must be simply connected, which is a contradiction.
\end{proof}

We now explain how to get Theorem~\ref{thm:special-charged-maps} from Theorem~\ref{theo:bij-shifted-0}. 
Recall that vertex-rooted maps identify with vertex-rooted hypermaps such that every dark face has degree 2 (see Figure \ref{fig:maps-are-hypermaps}). We can therefore translate the setting of \ref{thm:special-charged-maps} in terms of hypermaps. 
Let $\cC$ be the set of pairs $(H,\si)$ such that $H$ is a vertex-rooted hypermap where every dark face has degree 2, and $\si$ is a fitting charge function with $\si(e)=-2$ for every dark face $e$  and $\si(f)=2$ for every light face $f$. We need to prove the two following claims:\\


\begin{claim}\label{claim1}
The set $\cC$ identifies to the set of partially charged maps considered in Theorem~\ref{thm:special-charged-maps}.
\end{claim}


\begin{claim}\label{claim2}
The weighted hypermobiles associated with the set $\cC$ by the bijection of Theorem~\ref{theo:bij-shifted-0} identify with the suitably weighted mobiles considered in Theorem~\ref{thm:special-charged-maps}.
\end{claim}

We first prove Claim \ref{claim1}. If $M$ is a vertex-rooted map endowed with a partial charge function $\sigma$, we let $H$
be the vertex-rooted hypermap identified to $M$, keeping the same $\si$-values at vertices, and setting $\si(e)=-2$ for every dark face $e$ (of degree $2$, corresponding to an edge of $M$) and $\si(f)=2$ for every light face $f$ (corresponding to a face of $M$). Note that Condition (b) for a partial charge function $\si$ gives $\si_\tot=0$ (by the Euler relation). Thus, proving Claim \ref{claim1} amounts to proving that Condition (a) holds for $\si$ if and only if $M$ satisfies the $\si$-girth condition. 
If $R$ is a set of faces of the map $M$, we consider the set $E(R)$ of edges of $M$ having both incident faces in $R$. 
Thus $\bR=R\cup E(R)$ identifies to a light region of $H$, 
and it is easily seen that if $\bR$ is simply connected then the Euler relation gives 
$$\si(\bR)=2|R|-2|E(R)|+\sum_{v \textrm{ inside }R}\si(v)=2+\sum_{v \textrm{ inside }R}(\si(v)-2).$$
Therefore Condition (a) for a partial charge function $\si$ can be reformulated as: 
``for any simply connected light region of $M$ of the form $\bR=R\cup E(R)$, $|\partial \bR|\leq\si(\bR)$ with strict inequality if the root-vertex $v_0$ is inside $\bR$''. Moreover it is easy to check that if $\bR'=R\cup E'$ with $E'\subseteq E(R)$, then $|\partial \bR'|-\si(\bR')\leq |\partial \bR|-\si(\bR)$. Thus Condition (a) is equivalent to $|\partial \bR|\leq\si(\bR)$ (with strict inequality if $v_0$ inside $\bR$) for any light simply connected region $\bR$ of $M$. This together with Lemma~\ref{lem:simply-connected-sufficient} proves Claim \ref{claim1}.

It only remains to prove Claim \ref{claim2}. First note that the hypermobiles associated with maps have all the dark square vertices of degree 2, hence (upon removing the dark square vertices) these hypermobiles identify with the \emph{mobiles} as defined in the introduction. Hence by Theorem~\ref{theo:bij-shifted-0} the weighted hypermobiles associated with the set $\cC$ are mobiles having excess 0, with weights on half-edges such that 
\begin{compactitem}
\item every half-edge has a positive weight if it is incident to a round vertex and has a non-positive weight otherwise,
\item for every edge $e$, the weights of the two half-edges of $e$ add up to 0, 
\item every light square vertex $v$ has weight $2-\deg(v)$.
\end{compactitem}
These mobiles clearly identify with the suitably weighted mobiles considered in Theorem~\ref{thm:special-charged-maps} upon replacing the weights on half-edges (summing to 0) by non-negative weights on edges. This proves Claim \ref{claim2} and Theorem~\ref{thm:special-charged-maps}.


\section{Applications of the bijection for charged hypermaps to annular hypermaps}\label{sec:application-charged-maps}
In this section we characterize the $\si$-girth condition for particular choices of the charge function~$\si$, and derive from it bijections for annular hypermaps.

Given a real number $d$, we define the charge function $\si_d$ by 
\begin{itemize}
\item $\si_d(v)=d$ for every vertex $v$,
\item $\si_d(f)=d$ for every light face $f$,
\item $\si_d(f)=d-d\cdot\deg(f)$ for every dark face.
\end{itemize}
\begin{lemma}\label{lem:sid}
For any simply-connected light region $R$, $\si_d(R)=d$. Moreover ${\si_d}_\tot=2d$.
\end{lemma}
\begin{proof}
Let $V,E,F,K$ be respectively the set of vertices strictly inside $R$, edges strictly inside $R$, light faces inside $R$, and dark faces inside $R$. By the Euler relation we get,
\begin{equation*}
\si_d(R)=d\,(|V|+|F|+|K|-|E|)=d,
\end{equation*}
because $R$ is simply connected. Similarly, the Euler relation gives ${\si_d}_\tot=2d$.
\end{proof}

We now define a charge function, which will make clear that Theorem~\ref{theo:plane_dori} is a special case of Theorem~\ref{theo:shifted-orientation-dark}. Observe that an inward cycle (as defined in Section~\ref{sec:bij_plane}) is the boundary $C=\R$ of a simply connected light region $R$ not containing the outer face, \emph{such that $C$ is a simple cycle}.
\begin{lemma}\label{lem:sigma-d-condition}
Let $d$ be a positive integer and let $H$ be a dark-rooted hypermap of outer degree $d$. Let $\si$ be the charge function defined by 
\begin{itemize}
\item $\si(v)=d$ for every inner vertex $v$, and $\si(v)=0$ for every outer vertex $v$,
\item $\si(f)=d$ for every light face $f$,
\item $\si(f)=d-d\cdot\deg(f)$ for every inner dark face, and $\si_d(f_0)=-d$ for 
the dark outer face~ $f_0$.
\end{itemize}
The hypermap $H$ satisfies the $\si$-girth condition if and only if $H$ has ingirth $d$ (i.e., every inward cycle $C$ of $H$ has length at least $d$). Moreover in this case, the outer face is simple, and $\si_\tot=0$.
\end{lemma}
Let $\si$ be the charge function of Lemma~\ref{lem:sigma-d-condition}. It is clear that the definition of $d$-weighted hyperorientations coincide with the 
definition of $\si$-weighted hyperorientations. Moreover Lemma~\ref{lem:sigma-d-condition} together with Theorem~\ref{theo:shifted-orientation-dark} implies that $H$ admits a (unique) $\si$-weighted hyperorientation in $\cHm$ if and only if it has ingirth at least $d$. Thus Theorem~\ref{theo:plane_dori} is a special case of Theorem~\ref{theo:shifted-orientation-dark}.

Instead of proving Lemma~\ref{lem:sigma-d-condition}, we will prove a slight extension which will be used for counting hypermaps with given ingirth in Section~\ref{sec:counting-annular}. 
We define an \emph{annular hypermap} as a face-rooted hypermap with a marked inner face (hence $H$ has two distinct marked faces). 
Let $H$ be an annular hypermap, let $f_0$ be its outer face $f_0$, and let $f_1$ be its marked inner face. 
The \emph{separating ingirth} of $H$ is
 the minimal length of the boundary of a light region containing $f_1$ but not $f_0$. Observe that this minimal length is necessarily achieved for a boundary $C$ which is a simple cycle, that is, an inward cycle containing $f_1$. 
We call \emph{separating outgirth} of $H$ the minimal length of the boundary of a light region containing $f_0$ but not $f_1$. This minimal length is necessarily achieved for a boundary $C$ which is a simple cycle. We call \emph{separating outward cycle} a simple cycle which is the boundary of a light region containing $f_0$ but not $f_1$ (so that the separating outgirth is the minimal length of separating outward cycles). 
We call \emph{non-separating ingirth} of $H$ the minimal length of the boundary of a \emph{simply connected} light region containing neither $f_0$ nor $f_1$. Observe that this minimal length is \emph{not} necessarily achieved for a boundary $C$ which is simple (it could be that $C$ is the union of two simple cycles). 

\begin{lemma}\label{lem:sigma-condition-annular-dark}
Let $d,e$ be positive integers. Let $H$ be an annular hypermap with a 
 dark outer face $f_0$ of degree $e$ and a marked inner face $f_1$. 
We consider the charge function $\si$ defined by
\begin{itemize}
\item $\si(v)=d$ for every inner vertex $v$, and $\si(v)=0$ for every outer vertex~$v$,
\item $\si(f)=d$ for every non-marked light face~$f$, 
\item $\si(f)=d-d\cdot\deg(f)$ for every non-marked inner dark face, and $\si(f_0)=-e$ for the outer face~$f_0$,
\item $\si(f_1)=e$ if the marked inner face $f_1$ is light, and $\si(f_1)=e-d\cdot\deg(f_1)$ if $f_1$ is dark.
\end{itemize}
The hypermap $H$ satisfies the $\si$-girth condition if and only if $H$ has non-separating ingirth at least $d$, and separating ingirth $e$. In this case $\si_\tot=0$ and the outer face is simple.
\end{lemma}
Observe that Lemma~\ref{lem:sigma-d-condition} corresponds to the special case $e=d$ of Lemma~\ref{lem:sigma-condition-annular-dark} (up to forgetting the marked face which plays no particular role).

\begin{proof}
Let $\si'$ be the charge function defined by $\si'(f_0)=d\,e-d-e$, $\si'(f_1)=-d+e$, $\si'(v)=-d$ if $v$ is an outer vertex, and $\si'(a)=0$ for any inner vertex and any non-marked non-root face $a$. We have $\si=\si_d+\si'$. Hence by Lemma~\ref{lem:sid}, we get $\si(R)=d+\si'(R)$ for any simply connected light region $R$. Note also that if $H$ has separating ingirth $e$, then it implies that the outer face is simple. In this case, there are $e$ outer vertices hence $\si'_\tot=-2d$, and since by Lemma~\ref{lem:sid} ${\si_d}_\tot=2d$ we get $\si_\tot=0$.

Now assume that $H$ satisfies the $\si$-girth condition. For any separating inward cycle $C$, we consider the corresponding light region $R$ and get $|C|\geq \si(R)=d+\si'(R)=d+\si'(f_1)=e$. Similarly and for any non-separating inward cycle $C$ we get $|C|\geq \si(R)=d+\si'(R)=d$. Thus $H$ has non-separating ingirth at least $d$, and separating ingirth $e$. 

Conversely assume that $H$ has non-separating ingirth at least $d$, and separating ingirth $e$. We want to prove that $H$ satisfies the $\si$-girth condition. Since $\si_\tot=0$, Lemma~\ref{lem:simply-connected-sufficient} implies that we can focus on simply connected light regions of $H$. For a simply connected light region $R$, we know $\si(R)=d+\si'(R)$, and need to prove $|\R|>\si(R)$. 
First suppose that $R$ does not contain $f_0$. We get $\si(R)=e$ if $f_1\in R$ and $\si(R)=d$ otherwise. Moreover $\R$ contains an inward cycle $C$, thus by hypothesis, $|\R|\geq |C|\geq \si(R)$. 
Suppose now that $R$ contains $f_0$. Let $b$ be the number of outer vertices incident to $\R$ (the other outer vertices are all strictly inside $R$). We get $\si(R)= bd-d$ if $f_1\in R$ and $\si(R)=bd-e$ otherwise. If $b=0$ then $\si(R)<0$, hence $|\R|>\si(R)$ holds trivially. Suppose now that $b>0$. In this case the light region $R'=R\setminus\{f_0\}$ is the disjoint union of $b$ simply connected light regions $R_1,\ldots,R_b$, and $\R'$ is the disjoint union of their boundaries $\R_1,\ldots,\R_b$. Each  boundary $\R_i$ contains an inward cycle, so $|\R_i|\geq e$ if $R_i$ contains $f_1$ and $|\R_i|\geq d$ otherwise. Since $\ds |\R|=|\R'|-e=\sum_{i=1}^b|\R_i|-e$ we get $|\R|\geq bd-d=\si(R)$ if $R$ contains $f_1$ and $|\R|\geq bd-e=\si(R)$ otherwise. This completes the proof that $H$ satisfies the $\si$-girth condition.
\end{proof}

We now give a similar result for light-rooted hypermaps.
\begin{lemma}\label{lem:sigma-condition-annular-light}
Let $d,e$ be positive integers. Let $H$ be an annular hypermap with a light outer face face $f_0$ of degree $e$ and a marked inner face $f_1$. 
We consider the charge function $\si$ defined by
\begin{itemize}
\item $\si(v)=d$ for every vertex $v$, 
\item $\si(f)=d$ for every non-marked inner light face $f$, and $\si(f_0)=e$ for the outer face~$f_0$,
\item $\si(f)=d-d\cdot\deg(f)$ for every non-marked inner dark face, 
\item $\si(f_1)=-e$ if the marked face $f_1$ is light, and $\si(f_1)=-e-d\cdot\deg(f_1)$ if $f_1$ is dark.
\end{itemize}
Then $\si_\tot=0$, and the hypermap $H$ satisfies the $\si$-girth condition if and only if $H$ has non-separating ingirth at least $d$, and separating outgirth $e$ and such that the only outward cycle of length $e$ is the contour of the outer face.
\end{lemma}

\begin{proof}
We have $\si=\si_d+\si'$, where $\si'(f_0)=-d+e$, $\si'(f_1)=-d-e$, and $\si'(a)=0$ for any vertex and any non-marked non-root face $a$.
Lemma~\ref{lem:sid} gives $\si_\tot={\si_d}_\tot+\si'(f_0)+\si'(f_1)=0$.
Using Lemmas~\ref{lem:simply-connected-sufficient} and~\ref{lem:sid}, we easily see that the $\si$-girth condition translates into the following condition for any simply connected region $R$:
\begin{itemize}
\item[(i)] $|\R|\geq d$ if $R$ contains neither $f_0$ nor $f_1$,
\item[(ii)] $|\R|\geq e$ if $R$ contains $f_0$ but not $f_1$, with strict inequality if $R\neq \{f_0\}$,
\item[(iii)] $|\R|\geq -e$ if $R$ contains $f_1$ but not $f_0$,
\item[(iv)] $|\R|> -d$ if $R$ contains both $f_0$ and $f_1$,
\end{itemize}
The conditions (iii) and (iv) are void, while the conditions (i) and (ii) are clearly equivalent to the fact that $H$ has non-separating ingirth at least $d$, separating outgirth $e$, and the contour of the outer face is the only separating outward cycle of length $e$. 
\end{proof}

We will now use Lemmas~\ref{lem:sigma-condition-annular-dark} and~\ref{lem:sigma-condition-annular-light} in conjunction with Theorems~\ref{theo:bij-shifted-dark} and~\ref{theo:bij-shifted-light} to get bijections with classes of hypermobiles. 
For any integers $d,e$ (where $e$ is allowed to be negative), 
we call \emph{$(d,e)$-weighted hypermobile} a consistently weighted hypermobile (that is, a weighted hypermobile such that edges incident to a round vertex have positive weight, while edges incident to a light square vertex have non-positive weight), with a marked square vertex, such that 
\begin{itemize}
\item every round vertex has weight $d$.
\item every unmarked light square vertex $v$ has weight $d-\deg(v)$,
\item every unmarked dark square vertex $v$ has weight $d\cdot\deg(v)-d-\deg(v)$,
\item the marked square vertex $v$ has weight $e-\deg(v)$ if $v$ is light and $d\cdot\deg(v)-e-\deg(v)$ if $v$ is dark.
\end{itemize}
By similar arguments as in Claim~\ref{claim:excess-d-weighted-hypermobile} one can check that 
a $(d,e)$-weighted hypermobile always has excess $-e$.

Let $\cB^{(d,e)}$ be the family of annular hypermaps with a dark outer face of 
degree $e$, non-separating ingirth at least $d$ and separating ingirth $e$. 
Let $\cB'^{(d,e)}$ be the set of fittingly charged annular hypermaps $(H,\si)$
with a dark outer face of degree $e$, 
where the charge $\si$ is defined as in Lemma~\ref{lem:sigma-condition-annular-dark}.
 By Lemma~\ref{lem:sigma-condition-annular-dark} we can identify the sets $\cB^{(d,e)}$ and $\cB'^{(d,e)}$. Moreover, by Theorem~\ref{theo:bij-shifted-dark}, the mapping $\Phi_-$ gives a bijection between $\cB'^{(d,e)}$ and the family of $(d,e)$-weighted hypermobiles. Thus we obtain the following result (see Figure~\ref{fig:bij_ingirth_d_annular_dark} for an example).
\begin{thm} \label{thm:bij_annular_dark}
For $d$ and $e$ positive integers, the family $\cB^{(d,e)}$ 
is in bijection with the family of $(d,e)$-weighted hypermobiles. 
 Each light (resp. dark) inner face in the hypermap corresponds to a light (resp. dark) square vertex of the same degree in the associated hypermobile. Moreover, the marked inner face corresponds to the marked square vertex. 
\end{thm} 
\begin{remark}
Theorem~\ref{thm:bij_annular_dark} generalizes Theorem~\ref{theo:plane_dori} which corresponds to the case $e=d$ (by forgetting the marked inner face). Indeed the ingirth of a hypermap with a marked inner face is the minimum of its separating and non-separating ingirths. This theorem also generalize the bijection established for so-called \emph{annular maps} (planar maps with a root-face and an additional marked face) in \cite{BFgir}. 
\end{remark}
\begin{remark} 
It would be possible to extend Theorem~\ref{thm:bij_annular_dark} to the case where there is a marked inner vertex $v_1$ instead of a marked inner face $f_1$. In that case the separating ingirth
is the minimal boundary-length of a light region not containing $f_0$ but containing $v_1$ in its strict interior. 
In the associated hypermobiles, the marked vertex would be round instead of square, and its weight would be $e$ 
instead of $d$. 
It would also be possible to extend Theorem~\ref{thm:bij_annular_dark} to the case where there is a root-vertex $v_0$ instead of a root-face $f_0$: this would correspond to the case $e=0$ (no constraint on the 
separating ingirth) and we would apply $\Phi_0$ instead of
 $\Phi_-$.  These extensions can be obtained from the charged-map setup in a way similar to the results proved in this section.
\end{remark}

\begin{figure}
\begin{center}
\includegraphics[width=\linewidth]{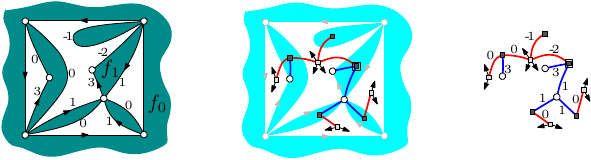}
\end{center}
\caption{The bijection of Theorem~\ref{thm:bij_annular_dark} on an example (case $d=3$, $e=4$, with a dark marked inner face $f_1$).} 
\label{fig:bij_ingirth_d_annular_dark}
\end{figure}

\begin{figure}
\begin{center}
\includegraphics[width=\linewidth]{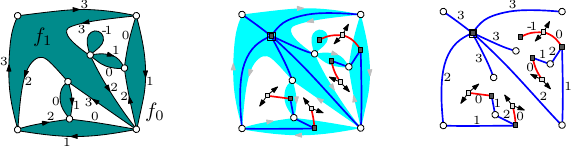}
\end{center}
\caption{The bijection of Theorem~\ref{thm:bij_annular_light} on an example (case $d=3$, $e=4$, with a dark marked inner face $f_1$).} 
\label{fig:bij_ingirth_d_annular_light}
\end{figure}

Similarly let $\cC^{(d,e)}$ be the set of annular hypermaps with a light outer face of degree $e$, non-separating ingirth at least $d$, separating outgirth $e$, and such that the only outward cycle of length $e$ is the contour of the outer face. 
Let $\cC'^{(d,e)}$ be the set of fittingly charged annular hypermaps $(H,\si)$ with a light
outer face of degree $e$, where the charge $\si$ is defined as in Lemma~\ref{lem:sigma-condition-annular-light}.
 By Lemma~\ref{lem:sigma-condition-annular-light} we can identify the sets $\cC^{(d,e)}$ and $\cC'^{(d,e)}$, and by Theorem~\ref{theo:bij-shifted-light}, the mapping $\Phi_+$ gives a bijection between $\cC'^{(d,e)}$ and the family of $(d,-e)$-weighted hypermobiles. 
 In conclusion we obtain (see Figure~\ref{fig:bij_ingirth_d_annular_light} for an example):
\begin{thm} \label{thm:bij_annular_light}
For $d$ and $e$ positive integers, the family $\cC^{(d,e)}$ is in bijection 
with the family of $(d,-e)$-weighted hypermobiles. 
 Each light (resp. dark) inner face in the hypermap corresponds to a light (resp. dark) square vertex of the same degree in the associated hypermobile. Moreover, the marked inner face corresponds to the marked square vertex. 
\end{thm} 


\section{Counting annular hypermaps according to girth parameters}\label{sec:counting-annular}
 In Section~\ref{sec:count_plane_hypermaps} we have counted plane hypermaps with control on the ingirth and the face degrees, but with the restriction that the outer degree is equal to the ingirth. Here we will drop this restriction. Our strategy (as in our previous article~\cite{BFgir} dealing with maps) is to consider annular hypermaps instead of plane hypermaps, and use a canonical decomposition of annular hypermaps into two hypermaps that can be counted bijectively.

Recall that an \emph{annular hypermap} is a hypermap with two marked faces: one called the 
outer face and the other called the marked inner face.
An annular hypermap is \emph{corner-rooted} by marking a corner in the outer face and a corner in the marked inner face. 
 Let $\cAad$ be the family of corner-rooted annular hypermaps of separating girth $e$
and non-separating ingirth at least $d$. 
Let $\cRad$ (resp. $\cSad$) be the family of corner-rooted annular hypermaps such that
the underlying (unrooted) annular hypermap is in the family $\cB_{d,e}$
 (resp. $\cC_{d,e}$) defined in Section~\ref{sec:application-charged-maps}. 
\begin{lem}\label{lem:cut_C}
There is an $e$-to-1 correspondence between $\cAad$ and the Cartesian product $\cSad\times\cRad$. 
\end{lem}

\begin{proof}
We will first define the \emph{canonical cycle} of an annular hypermap $H\in \cAad$.
 For any cycles $C_1,C_2$ that are contours of some light regions $R_1,R_2$ of a hypermap $H$, we denote by $\cap(C_1,C_2)$ (resp. $\cup(C_1,C_2)$) the contour
of the light region $R_1\cap R_2$ (resp. $R_1\cup R_2$). It is easy to see that 
$$
|\cup(C_1,C_2)|+|\cap(C_1,C_2)|=|C_1|+|C_2|.
$$
Thus, if $C_1$ and $C_2$ are separating inward cycles of length $e$, then $\cup(C_1,C_2)$ and $\cap(C_1,C_2)$ are both separating inward cycles of length $e$ (since $H$ has separating girth $e$). Thus $H$ has a separating inward cycle $C$ of length $e$ which is the outermost (that is, its light region contains the light region of any separating inward cycle of length $e$), and we call it the \emph{canonical cycle} of $H$. 

We now define the $e$-to-1 correspondence between $\cAad$ and $\cSad\times\cRad$. 
Let $\cAadb$ be the set of pairs $(H,v)$ where $H\in \cAad$ and $v$ is a vertex on the canonical cycle. 
For $(H,v)\in \cAadb$, with $f_0$ the outer face and $f_1$ the marked inner face of $H$, 
we denote by $\phi(H,v)$ the pair $(I,J)$ of
corner-rooted annular hypermaps obtained by cutting along $C$: the marked inner face of $I$ is $f_0$ and the outer
face of $I$ is delimited by $C$, while the marked inner face of $J$ is $f_1$
and the outer face of $J$ is delimited by $C$ (the marked corners in the
faces delimited by $C$ are at $v$). It is immediate to check that $I\in\cSad$, and $J\in\cRad$. Hence $\phi$ is a mapping from $\cAadb$ to $\cSad\times\cRad$. 

It remains to prove that $\phi$ is a bijection, which we do by exhibiting the inverse mapping. For $(I,J)\in \cSad\times\cRad$, we let $\psi(I,J)$ be the pair $(H,v)$, where $H$ is the corner-rooted annular hypermap obtained by patching
the outer face of $I$ with the outer face of $J$ so that their marked outer corners coincide, defining $v$ as their common incident vertex after patching, and defining the outer face of $H$ as the marked inner face of $I$. It is clear that $\psi\circ\phi=\Id$ and we need to prove $\phi\circ\psi=\Id$.
Hence, we need to prove that if $(H,v)=\psi(I,J)$ then $H\in \cAad$ and the cycle $C'$ of $H$ resulting from merging the outer face of $I$ with the outer face of $J$ is the canonical cycle $C$ of $H$. Note that $|C'|=e$ and $|C|\leq e$. 
Moreover, since $\cap(C,C')$
is a separating inward cycle of $J$, we get $|\cap(C,C')|\geq e$. And since $\cup(C,C')$ is a separating outward cycle of $I$, we get $|\cup(C,C')|\geq e$ with
equality if and only if $\cup(C,C')=C'$. Thus $|C|+|C'|=|\cup(C,C')|+|\cup(C,C')|\geq 2e$, and finally $|C|=|C'|=e$. This implies that the separating ingirth of $H$ is $e$, and moreover $\cup(C,C')=C'$ which implies that $C'=C$ (since $C$ is the outermost separating inward cycle of length $e$). It only remains to show that $H$ has non-separating girth at least $d$. 
Let $\hat{R}$ be a light region of $H$ not containing the inner marked face $f_1$ and let $\hat{C}$
be the contour of $\hat{R}$. 
If $\cap(C,\hat{C})$ encloses no face, then $\hat{C}$ completely belongs to $I$, so
that $|\hat{C}|\geq d$. 
Otherwise, $\cap(C,\hat{C})$ completely belongs to $J$ and
 is the contour of a non-empty light region not 
containing $f_1$, 
hence $|\cap(C,\hat{C})|\geq d$. 
Moreover
 $\cup(C,\hat{C})$ is a separating outward cycle of $I$, 
hence $|\cup(C,\hat{C})|\geq e$.
Thus $|\hat{C}|=|\cup(C,\hat{C})|+|\cap(C,\hat{C})|-|C|\geq d$. Thus $H$ has non-separating ingirth $d$ and $H\in\cAad$, which completes the proof that $\phi$ is a bijection.
\end{proof}

For $k,\ell\geq 1$, we define $\cAbw$ as the family of corner-rooted annular hypermaps of
separating ingirth $e$, non-separating ingirth at least $d$, 
where the outer face is dark of degree $k$ and the marked inner face is light of degree $\ell$.
Let $\Abw\equiv\Abw(x_1,x_2,\ldots;y_1,y_2,\ldots)$ be the generating function
of $\cAbw$ where $x_i$ and $y_i$ mark respectively the number of unmarked inner light and dark faces of degree $i$. 
We define the families $\cAbb$, $\cAwb$, $\cAww$ (depending on the types, light or dark, of the outer face and of the marked inner face) and their associated generating functions similarly.
Let $\cRadb$ (resp. $\cSadb$) be the subfamily of $\cRad$ (resp. $\cSad$) for which the marked inner face is dark of degree $k$. 
Let $\Radbk\equiv \Radbk(x_1,x_2,\ldots;y_1,y_2,\ldots)$ (resp. $\Sadbk\equiv \Sadbk(x_1,x_2,\ldots;y_1,y_2,\ldots)$) be the generating function of $\cRadb$ (resp. $\cSadb$) where $x_i$ and $y_i$ mark respectively the number of light and dark unmarked inner faces of degree $i$. We define the families $\cRadw$, $\cSadw$ and their generating functions similarly.
Lemma~\ref{lem:cut_C} gives $$A^{\ast k,\star\ell}_{d,e}=\frac{1}{e}C^{\ast k}_{d,e}B^{\star\ell}_{d,e}.$$
for $\ast\in\{\blacklozenge,\lozenge\}$ and $\star\in\{\blacklozenge,\lozenge\}$.

We now use Theorems~\ref{thm:bij_annular_dark} and~\ref{thm:bij_annular_light} to determine $B^{\star\ell}_{d,e}$ and $C^{\ast k}_{d,e}$.
Theorem~\ref{thm:bij_annular_dark} gives a bijection between $\cB_{d,e}$ and the family of $(d,e)$-weighted hypermobiles. It is easily seen that marking a corner in the marked inner face of an annular hypermap in $\cB_{d,e}$ corresponds to marking a corner at the marked square vertex of the associated $(d,e)$-weighted hypermobile. Thus, there is an $e$-to-$1$ correspondence between $\cRad$ and the family $\cT_{d,e}$ of $(d,e)$-weighted hypermobiles with a marked corner at the marked square vertex (the factor $e$ correspond to choosing the marked corner in the outer face of the annular hypermap). Moreover the degree and color of the marked inner face of the hypermap corresponds to the degree and color of the marked vertex of the hypermobile. Thus by decomposing hypermobiles in $\cT_{d,e}$ at their root-vertex (which yields a sequence of planted $d$-hypermobiles) we get 
$$
\Radwk=e[u^e]\W(u)^k,\ \ \ 
\Radbk=e[u^{-e}]L(u)^k,
$$
where $\W(u)$ and $L(u)$ are defined by~\eqref{eq:LM}. Similarly, Theorem~\ref{thm:bij_annular_light} leads to 
$$
\Sadwk=e[u^{-e}]\W(u)^k,\ \ \ 
\Sadbk=e[u^e]L(u)^k.
$$
We therefore obtain the following result.
\begin{theo}\label{theo:count_ann}
For $e,d,k,\ell\geq 1$, the generating functions of corner-rooted annular maps have the 
following expressions:
\begin{eqnarray*}
&&\Abw=e[u^e]L(u)^k[v^e]\W(v)^{\ell},\ \ \ \Awb=e[u^{-e}]\W(u)^k[v^{-e}]L(v)^{\ell},\\
&&\Abb=e[u^e]L(u)^k[v^{-e}]L(v)^{\ell},\ \ \ \Aww=e[u^{-e}]\W(u)^k[v^e]\W(v)^{\ell},
\end{eqnarray*}
where $L(u)$ and $\W(u)$ are specified by \eqref{eq:LM}, \eqref{eq:2} and \eqref{eq:3}.
\end{theo}

\begin{remark}
Under the specialization $y_2=1,y_i=0\ \mathrm{for}\ i\neq 2$, the generating function 
$\Aww$ counts corner-rooted annular maps with control on the separating girth, the non-separating
girth, and the face degrees. 
Hence Theorem~\ref{theo:count_ann} gives an extension to annular 
hypermaps of the counting results obtained in~\cite{BFgir} for annular maps.

Moreover, it is easy to see that the generating function $F_d$ defined in Section~\ref{sec:count_plane_hypermaps} is related to $A^{\blacklozenge d,\star \ell}_{d,d}$ by 
$\ell\frac{\partial F_d}{x_{\ell}}=A^{\blacklozenge d,\lozenge\ell}_{d,d}$ and $A^{\blacklozenge d,\blacklozenge\ell}_{d,d}=\ell\frac{\partial F_d}{y_{\ell}}$. Hence the expressions for the derivatives of $F_d$ given in Theorem~\ref{theo:count_plane_hypermaps} are a special case of Theorem~\ref{theo:count_ann}.
\end{remark}

For any sets $\Delta,\Delta'$, the generating function $A^{\ast k,\star\ell}_{d,e,\Delta,\Delta'}$ of hypermaps in $\vcA^{\ast k,\star\ell}_{d,e,\Delta,\Delta'}$ with inner light faces having degrees in $\Delta$ and inner dark faces having degrees in $\Delta'$ is obtained by setting $x_i=0\ \mathrm{for}\ i\notin \Delta,\ y_i=0\ \mathrm{for}\ i\notin \Delta'$. This is an algebraic series as soon as $\Delta,\Delta'$ are both finite. 
For instance, for $d=4$, $e=2$, $\Delta=\{4\}$, $\Delta'=\{3\}$, we have
$$
A^{\blacklozenge 4,\lozenge 2}_{4,2}=2(4L_2+6L_3^2)(1+W_0)^2,
$$ 
where the series $\{L_0,L_1,L_2,L_3,L_4,W_0,W_1,W_2,W_3,W_4\}$ (already considered in the example
of Section~\ref{sec:count_plane_hypermaps}) are specified by
\begin{eqnarray*}
&&L_0=x_4(1+W_0)^3,\ L_1=W_1^3+2W_1W_2+W_3,\ L_2=W_1^2+W_2,\ L_3=W_1,\ L_4=1,\\
&&W_0=2y_3L_2L_3,\ W_1=y_3(2L_1L_3+L_2^2),\ W_2=2y_3L_1L_2,\ W_3=y_3L_1^2,\ W_4=2y_3L_1.
\end{eqnarray*}

\bigskip

\section{Proof of Theorems~\ref{theo:master_bij1} and~\ref{theo:master_bij2} about the master bijection}\label{sec:proof_master_bij}
In this section we prove Theorems~\ref{theo:master_bij1} and~\ref{theo:master_bij2} about the three master bijections $\Phi_+$, $\Phi_-$ and $\Phi_0$. The proofs for the three bijections are similar. We give a detailed proof for $\Phi_+$ in Section~\ref{subsec:Phi+} and a more succinct proof for $\Phi_-$ and $\Phi_0$ in Sections~\ref{subsec:Phi-} and~\ref{subsec:Phi0}. 

\subsection{Proof for $\Phi_+$}\label{subsec:Phi+}
Let $\cJp$ be the family of light-rooted hyperorientations such that all outer edges are 1-way. Note that $\cHp$ is a subset of $\cJp$. 
We now extend the definition of the mapping $\Phi_+$ to $\cJp$. For $H\in\cJp$, we define $\Phi_+(H)$ as the map obtained from $H$ by placing a dark (resp. light) square vertex in each dark (resp. light) face, then applying the local rule of Figure~\ref{fig:local-rule-hyperori} to each edge of $H$,
and then deleting the edges of $H$ and the light square vertex corresponding to the outer face (see Figures~\ref{fig:Phinotmin} and~\ref{fig:Phinotacc} for examples). 
\begin{lem}\label{lem:iffp}
Let $H$ be an hyoeroriented hypermap in $\cJp$, and let $T=\Phi_+(H)$. Then, $T$ is a hypermobile if and only if $H\in\cHp$. \\
Moreover, in this case the following property holds for each inner 1-way edge $e$ of $H$:
\begin{itemize}
\item[$(\spadesuit)$]
Let $u,v$ be the square vertices in the faces incident to $e$, and let $C$ be the cycle contained in $T\cup\{e^*\}$, where $e^*$ is the edge joining $u$ and $v$ across $e$. Then $e$ is oriented from the outside of $C$ to the inside of $C$ (across $e^*$).
\end{itemize}
\end{lem}

\begin{proof}
First observe that $T$ is a hypermobile if and only if it is a tree (since the local conditions of hypermobiles are satisfied by $T$). Let $N_v$, $N_e$, $N_f$ be the numbers of vertices, edges, and faces of $H$. 
The map $T$ has $E=N_e$ edges (because each edge of $H$ yields an edge in $T$), and $V=N_v+N_f-1$ vertices (the $-1$ accounts for the deletion of the light square vertex in the outer face of $H$). The Euler relation for $H$ gives $N_v-N_e+N_f=2$, hence $E=V-1$. Thus $T$ is a hypermobile if and only if it is acyclic. 

\begin{figure}
\begin{center}
\includegraphics[width=\linewidth]{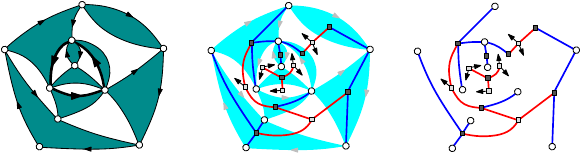}
\end{center}
\caption{If $H$ has a counterclockwise circuit $C$ (shown in bold line on the leftmost picture), then $\Phi(H)$ has a cycle outside of $C$.}
\label{fig:Phinotmin}
\end{figure}

\begin{figure}
\begin{center}
\includegraphics[width=.95\linewidth]{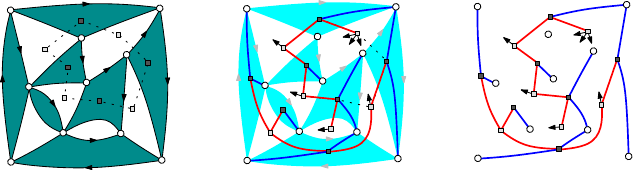}
\end{center}
\caption{If $H$ is not accessible from the outer boundary, there is a cycle $C$ in the dual of $H$ such that
all edges dual to edges on $C$ are either 0-way or 1-way from the inside to the outside of $C$; then 
$\Phi(H)$ has a cycle in the area exterior to $C$ (including $C$).}
\label{fig:Phinotacc}
\end{figure}

Now we prove that if $H\notin\cHp$, then $T$ has a cycle. For $H\notin\cHp$, either $H$ has a counterclockwise circuit or $H$ is not accessible from the outer vertices. Suppose first that $H$ has a counterclockwise circuit $C$ (see Figure~\ref{fig:Phinotmin}). 
Let $n_v $ and $n_e$ be the numbers of vertices and edges of $H$ that are on $C$ or outside of $C$, and let $n_f$ be the number of faces of $H$ that are outside of $C$. Note that the Euler relation (applied to $H$ where everything strictly inside $C$ is erased) yields $n_v-n_e+n_f=1$. 
Let $K$ be the submap of $T$ made of all its vertices on $C$ or outside of $C$ and all its edges outside of $C$. Since all edges on $C$ are counterclockwise, the submap $K$ has $E=n_e$ edges (because each edge on $C$ yields an edge of $T$ outside of $C$), and $V=n_v+n_f-1$ vertices (the $-1$ accounts for the deletion of the light square vertex in the outer face). Hence, $E=V$, so that $K$ has a cycle, and $T$ is not a tree. 
Suppose now that $H$ is not accessible from the outer vertices (see Figure~\ref{fig:Phinotacc}). We consider the \emph{dual map} $H^*$ which is obtained by placing a vertex $f^*$ of $H^*$ in each face $f$ of $H$, and drawing an edge $e^*$ of $H^*$ from $f_1^*$ to $f_2^*$ across each edge $e$ of $H$ separating the faces $f_1$ and $f_2$. An \emph{outward cocycle of $H$} is a sequence $D=e_1,\ldots,e_k$ of edges such that the dual edges $D^*=e_1^*,\ldots,e_k^*$ form a simple cycle of $H^*$, and for all $i\in\{1,\ldots,k\}$ the edge $e_i$ is either 0-way or 1-way toward the outside of $D^*$. It is not hard to prove that because $H$ is not accessible it has an outward cocycle $D=e_1,\ldots,e_k$ (to prove the existence of $D$ start by considering the set of vertices of $H$ that are reachable from the outer vertices). Let $n_v^*$, $n_e^*$ be the number of vertices and edges of $H^*$ that are on $D^*$ or outside of $D^*$, and let $n_f^*$ be the number of faces of $H^*$ that are outside of $D^*$. By the Euler relation applied to $H^*$ (where everything strictly inside $D^*$ is erased), $n_v^*-n_e^*+n_f^*=1$. Let $K$ be the submap of $T$ made of all its vertices on $D^*$ or outside of $D^*$ and all its edges on $D^*$ or outside of $D^*$. Since all edges in $D$ are 0-way or are 1-way from the inside to the outside of $D^*$, the submap $K$ has $E=n_e^*$ edges (because each edge in $D$ yields an edge of $T$ on $D^*$ or outside of $D^*$), and $V=n_v^*+n_f^*-1$ vertices (the $-1$ accounts for the deletion of the light square vertex in the outer face of $H$). Hence $E=V$, so that $K$ has a cycle, and $T$ is not a tree.

Next we prove that, if $H\in\cHp$, then $T$ is a hypermobile. We suppose by contradiction that $H\in\cHp$ and $T$ has a cycle $C$. We first consider the case where all vertices on $C$ are squares. In this case, the edges dual to edges on $C$ form a cocycle of 0-way
edges, so the (non-empty) set of vertices of $H$ inside $C$ is unreachable from the outer vertices of $H$,
a contradiction. We now suppose that there is a round vertex $u_0$ on $C$. Let $v_0$ be the (dark square) vertex following $u_0$ in clockwise order around $C$, and let $e_0$ be the edge of $H$ following the edge $\{u_0,v_0\}$ clockwise around $u_0$; see Figure~\ref{fig:prisoner_cycle1}. 
By the local rule of Figure~\ref{fig:local-rule-hyperori}, $e_0$ is 1-way toward $u_0$, and by accessibility of $H$, $e_0$ is the 
ending edge of some directed path $P_0$ starting from some outer vertex of $H$. Let $\widetilde{P_0}$ be the last portion
of $P_0$ inside $C$, and let $u_1\in C$ be the starting vertex of $\widetilde{P_0}$. Note that $u_1\neq u_0$, otherwise
$\widetilde{P_0}$ would form a counterclockwise circuit. By the same argument as for $u_0$, the next vertex $v_1$ after $u_1$ 
in clockwise order around $C$ is a dark square, and denoting by $e_1$ the next edge after $\{u_1,v_1\}$ in clockwise order around $u_1$,
there is a path $\widetilde{P_1}$ inside $C$ that starts from a vertex $u_2\in C$ and ends at $e_1$. Note that $u_2$
is not on the portion of $C$ going clockwise from $u_0$ to $u_1$ (otherwise it would yield a counterclockwise cycle in $H$). Continuing iteratively
we reach a contradiction, because at each step $i$, the vertex $u_i$ has to avoid a strictly growing portion of $C$; see Figure~\ref{fig:prisoner_cycle1}. 

\begin{figure}[h]\begin{center}
\subfigure[]{\label{fig:prisoner_cycle1}
\includegraphics[width=.3\linewidth]{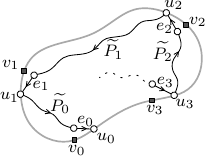}
}
\quad\quad\quad\quad\quad\quad
\subfigure[]{\label{fig:prisoner_cycle2}
\includegraphics[width=.3\linewidth]{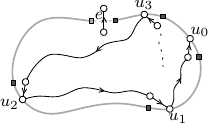}
}
\end{center}
\caption{(a) If $H$ is accessible, then the existence of a cycle in $\Phi(H)$ implies the existence of a counterclockwise circuit in $H$. (b) Proof of the property $(\spadesuit)$.}
\end{figure}

Lastly, the proof of property $(\spadesuit)$ follows the exact same line of argument as above. Assuming by contradiction that $H\in\cHp$ but that $(\spadesuit)$ does not hold for an edge $e$ we consider two cases. First if all vertices of $C$ are square, then the dual of the edges of $C$ are 0-way, so the inside of $C$ is unreachable from the outer vertices, giving a contradiction. Second, if there is a round vertex $u_0\in H$ on $C$, then one can construct a sequence $u_0,u_1,u_2,\ldots$ of vertices on $C$ such that $u_i$ has to avoid a strictly growing portion of $C$, again giving a contradiction; see Figure~\ref{fig:prisoner_cycle2}. 
\end{proof}

Next we prove that the mappings $\Phi_+$ and $\Psi_+$ are inverse bijections.

\begin{lemma}\label{lem:remains-tree}
Let $T$ be a hypermobile of positive excess, and let $H=\Psi_+(T)$ be the closure of $T$. 
Then $T$ is a tree covering all the vertices of $H$ and all the square vertices placed in the inner faces of $H$. 
\end{lemma}

\begin{proof}
To prove the lemma, it is convenient to see the closure mapping $\Phi_+$ as done ``step by step''. Let $\hT$ be the outerplanar map associated with $T$. Starting from $\hT$, define a \emph{local closure} as the operation of gluing a cw-outer edge $e_1$ with a ccw-outer edge $e_2$ such that $e_1$ and $e_2$ are consecutive edges in clockwise order around the outer face; see Figure~\ref{fig:invariant}. Then $H$ is obtained as the result of performing local closure operations greedily until there remains no pair to glue. At each step of the closure, we call \emph{floating} a vertex which is the origin of a ccw-outer edge. 
We now claim that \emph{at each step of the closure, $T$ is a tree covering all the vertices of the partially closed map except all the floating vertices}. Indeed this property is true for $\hT$. Moreover, it remains true through local closures because each local closure identifies a floating vertex with another vertex, and the resulting vertex is floating if both vertices are floating; see Figure~\ref{fig:invariant}.
\end{proof}

\begin{figure}
\begin{center}
\includegraphics[width=10cm]{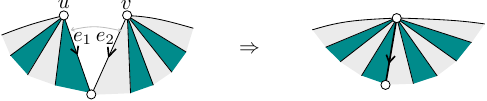}
\end{center}
\caption{A local closure glues a cw-outer edge, with a consecutive ccw-outer edge. This identifies a floating vertex $v$ with another vertex $u$.}\label{fig:invariant}
\end{figure}

\begin{cor}\label{cor:closurep}
Let $T\in\cTp$, and let $H=\Psi_+(T)$. Then $H$ is in $\cHp$, and $\Phi_+(H)=T$. 
Moreover, the excess of $T$ equals the outer degree of $H$. 
\end{cor}

\begin{proof}
Since the excess $\eps$ of $T$ is positive, after doing the closure of $\hT$ there remains $\epsilon$ cw-outer edge. Thus $H$ is in $\cJp$ and has outer degree $\eps$. Moreover it is clear that, while superimposing $T$ and $H$,  we have the local rules indicated in Figure~\ref{fig:local-rule-mobile} (since these rules are true for the outerplanar map $\hT$ and are preserved by the closure). Since these rules are also those of Figure~\ref{fig:local-rule-hyperori} (disregarding the incidences with the outer face), we conclude that $T=\Phi_+(H)$. Moreover, by Lemma~\ref{lem:remains-tree}, $T$ is a tree, hence a hypermobile. Thus by Lemma~\ref{lem:iffp}, $H$ is in $\cHp$. 
\end{proof}

\begin{figure}
\begin{center}
\includegraphics[width=10cm]{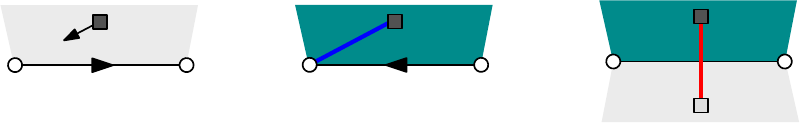}
\end{center}
\caption{The local rules for the configuration of $T$ for each incidence of an inner face with an edge of $H$.}\label{fig:local-rule-mobile}
\end{figure}

\begin{lem}\label{lem:openp}
Let $H\in\cHp$, and let $T=\Phi_+(H)$. Then $T$ is in $\cTp$, and $\Psi_+(T)=H$. 
\end{lem}

\begin{proof}
We have proved in Lemma~\ref{lem:iffp} that $T$ is a hypermobile. It remains to show that $\Psi_+(T)=H$. 
First of all, we claim that there exists a ``planar matching'' of the outer edges of the outerplanar map $\hT$ of $T$ such that gluing the outer edges of $\hT$ according to this matching yields $H$. Indeed to obtain the outerplanar map $\hT$ from $H$, one can apply the following operations illustrated on Figure~\ref{fig:open_into_cactus}:
\begin{compactitem}
\item[(i)] Replace each 1-way inner edge of $H$ by a pair of parallel 1-way edges, thereby creating a new face of degree $2$. 
\item[(ii)] For each 1-way edge $e$ with a new face on its right, detach from the origin $v$ of $e$ the sector between $e$ and the next 1-way edge $e'$ incident to $v$ in counterclockwise order around $v$; note that $e'$ has on its left  
either a new face or the outer face  
(see Figure~\ref{fig:open_into_cactus}). 
\end{compactitem}

\begin{figure}[h!]
\begin{center}
\includegraphics[width=\linewidth]{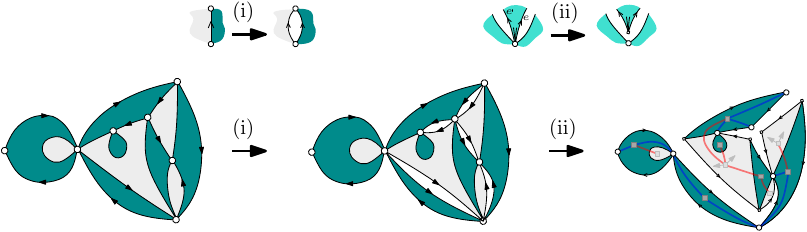}
\end{center}
\caption{Going from an hyperoriented hypermap $H$ (in $\cHp$) to the outerplanar map $\hT$ of the hypermobile $T=\Phi_+(H)$.}
\label{fig:open_into_cactus}
\end{figure} 

In order to prove that $\Psi_+(T)=H$ it remains to prove that the ``planar matching'' of the outer edges of $\hT$ giving $H$ corresponds to the cw-matching of these edges. This is essentially what property $(\spadesuit)$ in Lemma~\ref{lem:iffp} ensures. Indeed, consider a cw-outer edge $e'$ and a ccw-outer edge $e''$ of $\hT$ glued into an edge $e$ of $H$, and the sequence $e_1,e_2,\ldots,e_{n}$ of outer edges of $\hT$ appearing between $e'$ and $e''$ in clockwise order around the outer face of $\hT$. We need to prove that the sequence $e_1,e_2,\ldots,e_{n}$ is a parenthesis word (when cw-edges are interpreted as $a$'s and ccw-outer edges are interpreted as $\ba$'s). By the property $(\spadesuit)$ applied to $e$, all the outer edges $e_1,e_2,e_3,\ldots,e_n$ are glued into edges of $H$ which are \emph{inside} the cycle $C$ contained in $T\cup\{e^*\}$, hence they are all matched. Moreover, if $e_i$ and $e_j$ are matched into an edge $\tilde{e}$ of $H$, the property $(\spadesuit)$ applied to $\tilde{e}$ ensures that $i<j$ (since the cycle $\tilde{C}$ contained in $T\cup\{\tilde{e}^*\}$ lies inside $C$). Thus the sequence $e_1,e_2,\ldots,e_{n}$ is a parenthesis word. Therefore the ``planar matching'' of the outer edges of $\hT$ giving $H$ corresponds to the cw-matching of these edges, that is, $H=\Psi_+(T)$.
\end{proof}

Corollary~\ref{cor:closurep} and Lemma~\ref{lem:openp} conclude the proof of Theorems~\ref{theo:master_bij1} and~\ref{theo:master_bij2} for $\Phi_+$.

\subsection{Proof for $\Phi_-$}\label{subsec:Phi-}
The proof for $\Phi_-$ follows very similar lines. We highlight here the main differences. 
Let $\cJm$ be the set of dark-rooted hyperorientations such that the root face contour is simple and every outer edge is ccw-outer, and each incidence of an inner edge $e$ with an outer vertex $v$ is such that $e$ is either 0-way or 1-way out of $v$. 
 Note that $\cHm$ is a subset of $\cJm$. We now extend the definition of the mapping $\Phi_-$ to $\cJm$. For $H\in\cJm$, we define $\Phi_-(H)$ as the map obtained from $H$ by placing a dark (resp. light) square vertex in each dark (resp. light) face, then applying the local rule of Figure~\ref{fig:local-rule-hyperori} to each edge of $H$, and finally deleting the edges of $H$, the dark square vertex $v_0$ corresponding to the outer face, the outer vertices of $H$ and the edges between these vertices and $v_0$.

\begin{lem}\label{lem:iffm}
Let $H$ be an hyperoriented hypermap in $\cJm$, and let $T=\Phi_-(H)$. Then, $T$ is a hypermobile if and only if $H\in\cHm$. 
Moreover, in this case the property $(\spadesuit)$ holds for each inner 1-way edge $e$ of $H$.
\end{lem}

\begin{proof}
The proof is very similar to the proof of Lemma~\ref{lem:iffp}. As in Lemma~\ref{lem:iffp}, the Euler relation implies that $\Phi_-(H)$ is a hypermobile if and only if it is acyclic and the outer face is simple. 
Next one shows that if $H\notin\cHm$ then $\Phi_-(H)$ has a cycle. The proof is as for $\Phi_+$: consider either a counterclockwise cycle or outward cocycle $C$ of $H$, and can prove using the Euler relation that there is a cycle of $T$ in the region of $H$ outside of $C$. The only difference is that when applying the Euler relation, one needs to consider the subgraph $K$ of $T$ made of all its \emph{inner} vertices outside of $C$ and all its edges outside of $C$.
Lastly one shows that if $H\in\cHm$ then $\Phi_-(H)$ is acyclic, and property $(\spadesuit)$ holds exactly as for $\Phi_+$. 
\end{proof}

Next we prove that $\Phi_-$ and $\Psi_-$ are inverse of each other. 

\begin{lemma}\label{lem:remains-treem}
Let $T$ be a hypermobile of negative excess, and let $H=\Psi_-(T)$ be the closure of $T$. 
Then $T$ is a tree covering all the \emph{inner} vertices of $H$ (but none of the outer vertices)
 and all the square vertices placed in the inner faces of $H$. 
\end{lemma}

\begin{proof}
The proof of Lemma~\ref{lem:remains-treem} is the same as the proof of Lemma~\ref{lem:remains-tree}.
\end{proof}

\begin{cor}\label{cor:closurem}
Let $T\in\cTm$, and let $H=\Psi_-(T)$. Then $H$ is in $\cHm$, and $\Phi_-(H)=T$. 
Moreover, the excess of $T$ equals minus the outer degree of $H$. 
\end{cor}

\begin{proof}
Since the excess $\eps$ of $T$ is negative, after doing the closure operations on $T$ there remain $-\epsilon$ ccw-outer edge. Moreover since $T$ covers none of the outer vertices of $H$, each incidence of an inner edge $e$ of $H$ with an outer vertex $v$ is such that $e$ is either 0-way or 1-way out of $v$. Thus $H$ is in $\cJm$ and has outer degree $-\eps$. 
Moreover it is clear that superimposing $T$ and $H$ we have the local rules indicated in Figure~\ref{fig:local-rule-mobile} (since these rules are true for the outerplanar map $\hT$ and are preserved by the closure), hence $T=\Phi_-(H)$. Lastly, by Lemma~\ref{lem:remains-tree}, $T$ is a tree, hence a hypermobile. Thus by Lemma~\ref{lem:iffm}, $H$ is in $\cHm$. 
\end{proof}

\begin{lem}\label{lem:openm}
Let $H\in\cHm$, and let $T=\Phi_-(H)$. Then $T$ is in $\cTm$, and $\Psi_-(T)=H$. 
\end{lem}
\begin{proof}
The proof of Lemma~\ref{lem:openm} is the same as the proof of Lemma~\ref{lem:openp}. 
\end{proof}
Corollary~\ref{cor:closurem} and Lemma~\ref{lem:openm} conclude the proof of Theorems~\ref{theo:master_bij1} and~\ref{theo:master_bij2} for $\Phi_-$.

\subsection{Proof for $\Phi_0$}\label{subsec:Phi0}
The proof for $\Phi_0$ is again very similar. We define $\cJz$ as the family
of vertex-rooted hyperorientations such that for each incidence of an edge $e$ with the 
root-vertex $v_0$, $e$ is either 0-way or 1-way out of $v_0$. 
 We extend the definition of the mapping $\Phi_0$ to $\cJz$: for $H\in\cJz$, we define $\Phi_0(H)$ as the map obtained from $H$ by placing a dark (resp. light) square vertex in each dark (resp. light) face, then applying the local rule of Figure~\ref{fig:local-rule-hyperori} to each edge of $H$, and finally deleting the edges of $H$, and the root vertex $v_0$.
In a similar way as for $\Phi_-$, one proves:
\begin{lemma}\label{lem:iffz}
Let $H$ be an hyperoriented hypermap in $\cJz$, and let $T=\Phi_0(H)$. Then, $T$ is a hypermobile if and only if $H\in\cHz$. 
Moreover, in this case the following property holds for each inner 1-way edge $e$ of $H$:
\begin{itemize}
\item[$(\clubsuit)$]
Let $u,v$ be the square vertices in the faces incident to $e$, and let $C$ be the (unique) cycle contained in $T\cup\{e^*\}$, where $e^*$ is the edge joining $u$ and $v$ across $e$. Then $e$ is oriented from the region delimited by $C$ containing the root-vertex, to the other region delimited by $C$ (across $e^*$).
\end{itemize} 
\end{lemma}

Then the proof that $\Phi_0$ and $\Psi_0$ are inverse mappings is similar to the case $\Phi_-$. It implies Theorems~\ref{theo:master_bij1} and~\ref{theo:master_bij2} for $\Phi_0$.\\





\section{Proofs of Theorems~\ref{theo:plane_dori},~\ref{theo:shifted-orientation-dark},~\ref{theo:shifted-orientation-light}, and~\ref{theo:shifted-orientation-0} about canonical orientations}\label{sec:proofs}
Theorems~\ref{theo:shifted-orientation-dark},~\ref{theo:shifted-orientation-light}, and~\ref{theo:shifted-orientation-0} state that a hypermap $H$ admits a (unique) $\si$-weighted orientation in $\cHm$, $\cHp$, or $\cH_0$ if and only if the charge function $\si$ fits $H$. Recall that Theorem~\ref{theo:shifted-orientation-dark} actually generalizes Theorem~\ref{theo:plane_dori} about plane hypermaps (see Lemma~\ref{lem:sigma-d-condition}). In this section, we prove Theorems~\ref{theo:shifted-orientation-dark},~\ref{theo:shifted-orientation-light}, and~\ref{theo:shifted-orientation-0}. The proof is organized as follows.
\begin{itemize}
\item In Section~\ref{sec:girth-necessary}, we prove the necessity of the fitting condition in Theorems~\ref{theo:shifted-orientation-dark},~\ref{theo:shifted-orientation-light}, and~\ref{theo:shifted-orientation-0}.
\item In Section~\ref{sec:hyperflow}, we develop some tools useful for proving the existence of constrained hyperorientations.
\item In Section~\ref{sec:proof-case-d-light}, we prove Theorem~\ref{theo:shifted-orientation-light} in the case where every light face has charge equal to its degree.
\item In Section~\ref{sec:endproof-outer-light}, we complete the proof of Theorem~\ref{theo:shifted-orientation-light} by reduction to the case treated in Section~\ref{sec:proof-case-d-light}.
\item In Section~\ref{sec:endproof-outer-dark}, we complete the proof of Theorems~\ref{theo:shifted-orientation-dark} and~\ref{theo:shifted-orientation-0} by reduction to Theorem~\ref{theo:shifted-orientation-light}.
\end{itemize}

\subsection{Necessity of the fitting condition in Theorems~\ref{theo:shifted-orientation-dark},~\ref{theo:shifted-orientation-light}, and~\ref{theo:shifted-orientation-0}}\label{sec:girth-necessary}
In this subsection we prove the following lemma. 
\begin{lem}\label{lem:girth-necessary}
If a dark-rooted (resp. light-rooted, vertex-rooted) hypermap $H$ admits a $\si$-weighted orientation in $\cHm$ (resp. $\cHp$, $\cH_0$), then $\si$ fits $H$. Moreover if $H$ is dark-rooted, then the contour of the outer face is simple.
\end{lem}

\begin{proof}
Let $H$ be a dark-rooted, light-rooted, or vertex-rooted hypermap, and let $\si$ be a charge function such that $H$ admits a $\si$-weighted hyperorientation $\Om$ in $\cHm$, $\cHp$, or $\cH_0$. We denote by $w(a)$ the weight of a vertex, edge, or face $a$ of $H$ in $\Om$. 

We first suppose that $H$ is dark-rooted and prove that \emph{the contour of the outer face $f_0$ of $H$ is simple, the charge of every inner vertex is positive, the charge of every outer vertex is 0, the charge of the dark outer face $f_0$ is $-\deg(f_0)$, and $\si_\tot=0$}. 
By definition of $\cHm$ the contour of $f_0$ is a simple cycle, and since the weight of each outer edge is 1 in $\Om$, the weight of the outer face is $w(f_0)=\deg(f_0)$. Since, by definition, $w(f_0)=-\si(f_0)$, we get $\si(f_0)=-\deg(f_0)$. Moreover, by definition, the weight of any outer vertex $v$ is $w(v)=1=\si(v)+1$, hence $\si(v)=0$. Consider now an inner vertex $v$. Since the orientation $\Om\in \cHm$ is accessible from the outer vertices there is a 1-way edge $e$ directed toward $v$, hence $w(v)\geq w(e)>0$. It only remains to prove that $\si_\tot=0$. Let $V$, $F$, and $K$ be respectively the set of vertices, light faces, and dark faces of $H$. By definition, 
$$\sum_{v\in V}w(v)+\sum_{f\in F}w(f)=\sum_{k\in K}w(k),$$
and since $\Om$ is $\si$-weighted we get
$$\bigg(\sum_{v\in V}\si(v)\bigg)+\deg(f_0)+\bigg(\sum_{f\in F}\si(f)-\deg(f)\bigg)=\bigg(\sum_{k\in K}-\si(k)-\deg(k)\bigg)+\deg(f_0).$$
Hence, $\ds \si_\tot=\sum_{v\in V}\si(v)+\sum_{f\in F}\si(f)+\sum_{k\in K}\si(k)=0$.

With similar arguments, one proves that if $H$ is light-rooted then \emph{the charge of every vertex is positive, the charge of the light outer face $f_0$ is $\deg(f_0)$, and $\si_\tot=0$}, and if $H$ is vertex-rooted then \emph{the charge of every non-root vertex is positive, the charge of the root-vertex is $0$, and $\si_\tot=0$}. 

It only remains to prove that $H$ satisfies the $\si$-girth condition. We first suppose that $H$ is dark-rooted. 
Let $R$ be a light region. Let $V$, $E$, $F$ and $K$ be respectively the set of vertices strictly inside $R$, edges strictly inside $R$, light faces inside $R$, and dark faces inside $R$. We want to prove
 \begin{equation}\label{eq:difference-weights}
|\R|\geq \si(R):=\sum_{v\in V}\si(v)+\sum_{f\in F}\si(f)+\sum_{k\in K}\si(k),
\end{equation}
with strict inequality if every outer vertex is strictly in $R$.

Because $\Om$ is $\si$-weighted we get 
\begin{eqnarray*}
\sum_{v\in V}\si(v)&=&-b+\sum_{v\in V}w(v),\\
\sum_{f\in F}\si(f)&=& |E|+|\R|+\sum_{f\in F}w(f),\\
\sum_{k\in K}\si(k)&=& -|E|+\textbf{1}_{f_0\in R}\cdot \deg(f_0) -\sum_{k\in K}w(k),\\
\end{eqnarray*}
where $b$ is the number of outer vertices in $V$, and $f_0$ is the dark outer face. Hence
$$\si(R)=|\R|-b+\textbf{1}_{f_0\in R}\cdot \deg(f_0)+\sum_{v\in V}w(v)+\sum_{f\in F}w(f)-\sum_{k\in K}w(k),$$
and the requirement \eqref{eq:difference-weights} becomes
\begin{equation}\label{eq:difference-weights2}
\sum_{k\in K}w(k)-\sum_{v\in V}w(v)-\sum_{f\in F}w(f)\geq \textbf{1}_{f_0\in R}\cdot \deg(f_0)-b,
\end{equation}
Moreover we have
$$\sum_{k\in K}w(k)- \sum_{v\in V}w(v)-\sum_{f\in F}w(f)=x-y\geq x,$$
where $x$ is the sum of the (positive) weights of the 1-way edges in $E$ oriented toward vertices incident to edges in $\R$, and $y$ is the sum of the (non-positive) weights of the 0-way edges in $\R$. 
If $f_0\notin R$, then $b=0$ and the inequality \eqref{eq:difference-weights2} holds because $x\geq 0$. If $f_0\in R$ and $b=\deg(f_0)$ (i.e. every outer vertex is strictly inside $R$), then inequality \eqref{eq:difference-weights2} is strict because $x>0$ (indeed, since $\Om$ is accessible from the outer vertices of $H$, there exists a 1-way edge in $E$ oriented toward vertices of $\R$). Lastly suppose that $f_0\in R$ and $b<\deg(f_0)$. 
Because $f_0$ is a dark face, all the edges incident to $f_0$ are in $E$, and because $\Om\in\cHm$ these edges are 1-way and have weight 1.
Thus for each outer vertex $v$ on $\R$ there is an edge in $E$ of weight 1 oriented toward $v$. Hence $x\geq \deg(f_0)-b$ which is the number of outer vertices on $\R$. This proves \eqref{eq:difference-weights2} and completes that proof that $H$ satisfies the $\si$-girth condition when $H$ is dark-rooted.

The case where $H$ is light-rooted (resp. vertex-rooted) is similar. Indeed, by the same arguments, we see that the $\si$-girth condition a light region $R$ becomes the following requirement: 
\begin{equation*}\label{eq:difference-weights3}
\sum_{k\in K}w(k)-\sum_{v\in V}w(v)-\sum_{f\in F}w(f)\geq 0,
\end{equation*}
with strict inequality if one of the outer edges is strictly inside $R$ (resp. if the root vertex is strictly inside $R$). This is easily seen to hold with arguments similar to the ones above. The only point that requires a special argument is that the equality is strict if $H$ is light-rooted and one of the outer edges is strictly inside $R$. For this particular case, we need to prove that the sum $x$ of weights of the 1-way edges in $E$ oriented toward vertices incident to edges of $\R$ is positive. This holds, because if one of the outer vertices $v$ is strictly inside $R$ then $x>0$ because the vertices on $\R$ are accessible from $v$, while if none of the outer vertices is strictly in $R$, then the outer edge $e$ strictly inside $R$ is a 1-way edge in $E$ oriented toward a vertex of $\R$ (indeed, $e$ is 1-way because $\Om\in \cHp$). 
\end{proof}

\subsection{A preliminary result about $\al$-hyperflows}\label{sec:hyperflow}
In this subsection we prove a result akin to the mincut-maxflow theorem for the \emph{hyperflows} of bipartite graphs. This result will then be used in Section~\ref{sec:proof-case-d-light}.

Throughout this subsection, we fix a (finite, undirected) bipartite graph $G=(X\sqcup Y,E)$ where every edge $e\in E$ joins a vertex in $X$ to a vertex in $Y$. We call \emph{hyperflow} of $G$, a function $\varphi$ from the edge set $E$ to the set $\rp$ of non-negative real numbers. Let $P$ be a directed path, or cycle, of $G$ and let $P_X$ be the subset of edges of $P$ oriented toward a vertex in $X$. Given a hyperflow $\varphi$ of $G$, we say that $P$ is \emph{$\varphi$-positive} if $\varphi(e)>0$ for every edge $e\in P_X$. A $\varphi$-positive path is represented in Figure~\ref{fig:hyperflow}(a).
For a vertex $x_0\in X$, we say that a hyperflow $\varphi$ is \emph{accessible from $x_0$} if for all $x\in X$ there is a $\varphi$-positive path from $x_0$ to $x$. For instance, The hyperflow represented in Figure~\ref{fig:hyperflow}(a) is accessible from $x_0$.

\fig{width=.8\linewidth}{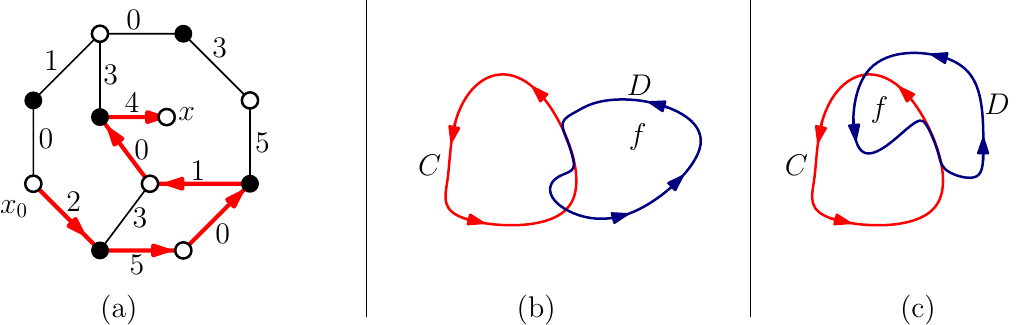}{(a): A bipartite graph endowed with a hyperflow $\varphi$, and a $\varphi$-positive directed path from $x_0$ to $x$. The vertices in $X$ and $Y$ are represented in light and dark respectively and the value of $\varphi$ is indicated on each edge. (b),(c): The cycles $C$ and $D$ in the proof of Lemma~\ref{lem:unique-minimal}.}

Let $\varphi$ be a hyperflow of $G=(X\sqcup Y,E)$. We call \emph{$\varphi$-flow at a vertex $v\in X\sqcup Y$} the sum 
$$\phi(v):=\sum_{e\in E\textrm{ incident to }v}\varphi(e).$$ 
Given a function $\al$ from $X\sqcup Y$ to $\rp$, we say that $\varphi$ is an $\al$-hyperflow if the $\varphi$-flow at every vertex $v\in X\sqcup Y$ is equal to $\al(v)$. We now establish a criterion for the existence of an accessible $\al$-hyperflow:

\begin{lem} \label{lem:existence-hyperflow}
Let $\al$ be a function from $X\sqcup Y$ to $\rp$. 
For a subset $A\subseteq X$, let us denote 
\begin{equation*}
\al(A):= \sum_{x\in X}\al(x)-\sum_{y\in Y_A}\al(y),
\end{equation*}
where $Y_A$ denotes the set of vertices in $Y$ having all their neighbors in $A$. Then there exists an $\al$-hyperflow of $G$ if and only if
$$\forall A\subseteq X,~~ \al(A)\geq 0,$$ 
with equality for $A=X$. 
Moreover for any vertex $x_0\in X$ and any $\al$-hyperflow $\varphi$, the hyperflow $\varphi$ is accessible from $x_0$ if and only if $\al(A)>0$ for all non-empty subset $A\subset X$ not containing $x_0$.
\end{lem}

\begin{proof}
First suppose that there exists an $\al$-hyperflow $\varphi$ of $G$. In this case, for all $A \subseteq X$, 
$$\ds \sum_{x\in A}\al(x)=\sum_{e \textrm{ incident to } A}\varphi(e)\geq \sum_{y\in Y_A}\al(y),$$
with equality if $A=X$. Hence $\al(A)\geq 0,$ with equality for $A=X$. \\
 
We will now prove that an $\al$-hyperflow exists whenever $\al(A)\geq 0$ for all $A \subseteq X$, with equality for $A=X$. 
We make an induction on $|X\cup Y\cup E|$. The property is trivial when $E=\emptyset$, hence for the induction step we can assume $E\neq \emptyset$.
We consider an edge $e_0\in E$ with endpoints $x_0\in X$ and $y_0\in Y$. For $\eps\geq 0$ we denote by $\al_\eps$ the function from $X\sqcup Y$ to $\rp$ defined by: $\al_\eps(x_0)=\al(x_0)-\eps$, $\al_\eps(y_0)=\al(y_0)-\eps$ and $\al_\eps(z)=\al(z)$ for all $z\neq x_0,y_0$. Observe that if $\varphi$ is an $\al_\eps$-hyperflow of $G$, then $\varphi'$ defined by $\varphi'(e_0)=\varphi(e_0)+\eps$ and $\varphi'(e)=\varphi(e)$ for all $e\neq e_0$ is an $\al$-hyperflow of $G$. Hence it suffices to prove that there exists an $\al_\eps$-hyperflow of $G$ for some $\eps\geq 0$. We choose $\eps$ maximal such that $\al_\eps(x_0)\geq 0$, $\al_\eps(y_0)\geq 0$, and $\al_\eps(A)\geq 0$ for all $A\subseteq X$. Clearly, $\al_\eps(X)=\al(X)=0$, and $\al_\eps(A)\geq 0$ for all $A\subseteq X$. Moreover, we have either $\al_\eps(x_0)=0$, or $\al_\eps(y_0)=0$ or $\al_\eps(A)= 0$ for some $A\neq \emptyset, X$. 
Suppose first $\al_\eps(x_0)=0$. In this case we consider the subgraph $G'$ obtained from $G$ by deleting $x_0$ and the incident edges, and we denote by $\al'$ the restriction of $\al_\eps$ to $G'$. Clearly $\al'(A)\geq 0$ for all $A\subseteq X\setminus \{x_0\}$, with equality for $A=X\setminus \{x_0\}$. Hence by the induction hypothesis, there exists an $\al'$-hyperflow of $G'$ and this gives an $\al_\eps$-hyperflow of $G$ (by setting the flow on edges incident to $x_0$ to be 0), 
and hence an $\al$-hyperflow of $G$. The case $\al_\eps(y_0)=0$ is similar. Suppose lastly that $\al_\eps(A)=0$ for some subset $A\neq \emptyset, X$. 
Let $\ov{A}=X\setminus A$ and $\ov{Y_A}=Y\setminus Y_A$.
Let $G_1$ (resp. $G_2$) be the graph with vertex set $A\cup Y_A$ (resp. $\ov{A}\cup\ov{Y_A}$) and edge set $E_1$ (resp. $E_2$) made of all the edges with both endpoints in $A\cup Y_A$ (resp. $\ov{A}\cup\ov{Y_A}$). Observe that the graph $G_1\cup G_2$ is simply obtained from $G$ by deleting the set $E_0$ of edges having both endpoints in $A\cup \ov{Y_A}$; see Figure~\ref{fig:induction-flow}.
We denote by $\al'$ and $\al''$ respectively the restriction of $\al_\eps$ to $G_1$ and $G_2$. 
Observe that for all $B\subseteq A$, the set of vertices of $G_1$ with all their neighbors in $B$ is $Y_{B\cup \ov{A}}\cap Y_A=Y_B$. 
Thus
$$\al'(B)=\sum_{x\in B}\al_\eps(x)-\sum_{y\in Y_B}\al_\eps(y)=\al_\eps(B).$$ 
Hence $\al'(B)\geq 0$ for all $B\subseteq A$, with equality for $B=A$. Hence, by the induction hypothesis, there exists a $\al'$-hyperflow $\varphi_1$ of $G_1$.
Now for $B\subseteq \ov{A}$, the set of vertices of $G_2$ with all their neighbors in $B$ is $Y_{A\cup B}\cap \ov{Y_A}$. Hence
\begin{eqnarray*}
\al''(B)&=&\sum_{x\in B}\al_\eps(x)-\sum_{Y_{B\cup A}\cap \ov{Y_A}}\al_\eps(y)\\
&=&\left(\sum_{x\in A\cup B}\al_\eps(x)-\sum_{x\in A}\al_\eps(x)\right)-\left(\sum_{y\in Y_{A\cup B}}\al_\eps(y)-\sum_{y\in Y_A}\al_\eps(y)\right)=\al_\eps(A\cup B).
\end{eqnarray*}
Hence $\al''(B)\geq 0$ for all $B\subseteq \ov{A}$, with equality for $B=\ov{A}$. Hence, by the induction hypothesis, there exists an $\al''$-hyperflow $\varphi_2$ of $G_2$. We now consider the hyperflow $\varphi$ of $G$ defined by $\varphi(e)=0$ if $e\in E_0$, $\varphi(e)=\varphi_1(E)$ if $E$ in $E_1$, and $\varphi(e)=\varphi_2(e)$ if $e\in E_2$. It is clear that $\varphi$ is an $\al_\eps$-hyperflow. This completes the proof by induction.

\fig{width=.3\linewidth}{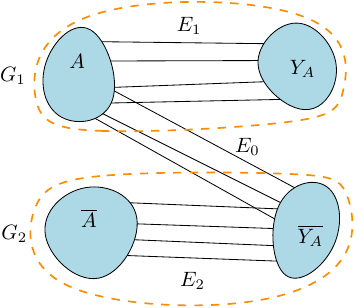}{The bipartite graph $G=(X\sqcup Y,E)$, and the subgraphs $G_1$ and $G_2$. We have $X=A\cup \ov{A}$, $Y=Y_A\cup \ov{Y_A}$ and $E=E_0\cup E_1\cup E_2$.}

It remains to prove that an $\al$-hyperflow $\varphi$ is accessible from a vertex $x_0\in X$ if and only if $\al(A)>0$ for all non-empty subset $A\subset X$ not containing $x_0$.
Suppose first that $\varphi$ is accessible from $x_0$ and let $A\subset X$ be a non-empty subset not containing $x_0$. Let $P$ be a $\varphi$-positive path from $x_0$ to a vertex $x\in A$. Let $e_0$ be the first edge of $P$ incident to a vertex in $A$. This edge of $P$ is directed from its endpoint $y\in Y$ to its endpoint $a\in A$, hence $\varphi(e_0)>0$. Moreover $y\notin Y_A$, hence 
$$\ds \sum_{x\in A}\al(x)=\sum_{e \textrm{ incident to } A}\varphi(e)\geq \varphi(e_0)+\sum_{y\in Y_A}\al(y).$$
Thus $\al(A)\geq \varphi(e_0)>0$, as wanted.
Suppose now that $\varphi$ is not accessible from $x_0$. Consider the set $A$ of vertices $x\in X$ such that there exists no $\varphi$-positive path from $x_0$ to $x$. This definition implies that every edge $e$ incident to a vertex $x\in A$ and a vertex $y\in Y\setminus Y_A$ satisfies $\varphi(e)=0$. Thus
$$\ds \sum_{x\in A}\al(x)=\sum_{e \textrm{ incident to } A}\varphi(e)=\sum_{e \textrm{ incident to } Y_A}\varphi(e)= \sum_{y\in Y_A}\al(y).$$
Hence $\al(A)=0$ for a non-empty set $A\subset X$ not containing $x_0$.
\end{proof}

\begin{remark}\label{rk:flow}
In the literature, $\al$-hyperflows are also known
as \emph{$b$-matchings}~\cite[Chap.~21]{Sh}. Our existence criterion in 
Lemma~\ref{lem:existence-hyperflow} can be checked
to be equivalent to Corollary 21.1b from~\cite{Sh} (we have provided our own proof and terminology
for completeness and convenience). About efficiently computing an $\al$-hyperflow of $G=(V,E)$, when $\al$ only has 
integer values the problem can easily be reduced to that of 
finding a perfect matching in a bipartite graph $G'=(V',E')$ associated to $G$ (each vertex $v\in G$ is turned into $\al(v)$ copies in $G'$, and for each edge $(u,v)\in G$, there is an edge in $G'$ between every copy of $u$ and every copy of $v$). 
The algorithm of Hopcroft and Karp~\cite{hopcroft1973n} yields a perfect matching of $G'$ in time $O(\sqrt{|V'|}|E'|)$,
which is $O(c\sqrt{|V|}|E|)$, with $c=(\sum_{v\in V}\al(v))^{1/2}\sum_{(u,v)\in E}\al(u)\al(v)$. A detailed 
survey on complexity results (for the general case of flow values in $\rp$) is given in~\cite[Chap.~21]{Sh}. 
\end{remark} 

Suppose that a bipartite graph $G=(X\sqcup Y, E)$ is embedded (i.e., drawn without edge crossings) in the plane. 
In this case, a directed cycle $C$ of $G$ is called \emph{counterclockwise} if the outer face of $G$ lies to the right of $C$.
A hyperflow $\varphi$ of $G$ is called \emph{minimal} if there is no $\varphi$-positive counterclockwise directed cycle of $G$. The hyperflow $\varphi$ represented in Figure~\ref{fig:hyperflow}(a) is \emph{not} minimal because there is a $\varphi$-positive counterclockwise directed cycle of length 4. 

\begin{lem}\label{lem:unique-minimal}
Let $G=(X\sqcup Y, E)$ be a bipartite graph embedded in the plane. If $\al$ is a function from $X\sqcup Y$ to $\rp$ such that there exists an $\al$-hyperflow of $G$, then there exists a unique minimal $\al$-hyperflow of $G$.
\end{lem}

\begin{proof}
We first prove the existence of a minimal $\al$-hyperflow. We first define the operation of \emph{pushing} a cycle. 
Let $\varphi$ be an $\al$-hyperflow, and let $C$ be a $\varphi$-positive counterclockwise directed cycle. Let $C_X$ (resp. $C_Y$) be the subset of edges of the directed cycle $C$ oriented toward a vertex in $X$ (resp. $Y$). Let $m=\min\{\varphi(e),e\in C_X\}$ and let $\psi$ be the hyperflow defined by $\psi(e)=\varphi(e)-m$ if $e\in C_X$, $\psi(e)=\varphi(e)+m$ if $e\in C_Y$, and $\psi(e)=\varphi(e)$ if $e$ is not in $C$. Observe that $\psi$ is an $\al$-hyperflow.
We say that $\psi$ is the $\al$-hyperflow obtained from $\varphi$ by \emph{pushing} the cycle $C$. We will now prove that \emph{the minimal $\al$-hyperflow can be obtained from any $\al$-hyperflow by repeatedly pushing counterclockwise directed cycles}. For an $\al$-hyperflow $\varphi$ we consider the total number $N(\varphi)$ of faces which are \emph{enclosed in} (i.e., separated from the outer face by) a $\varphi$-positive counterclockwise directed cycle. By definition, an $\al$-hyperflow $\varphi$ is minimal if and only if $N(\varphi)=0$. Hence it is sufficient to show that for any non-minimal $\al$-hyperflow $\varphi$ there is an $\al$-hyperflow $\psi$ obtained from $\varphi$ by pushing a $\varphi$-positive counterclockwise directed cycle such that $N(\psi)<N(\varphi)$. Let $\varphi$ be a non minimal $\al$-hyperflow, and let $C$ be a $\varphi$-positive counterclockwise directed cycle $C$ enclosing a maximal number of faces. We consider the $\al$-hyperflow $\psi$ obtained from $\varphi$ by pushing the cycle $C$. Now consider a face $f$ not enclosed by a $\varphi$-positive counterclockwise directed cycle. 
If $f$ is enclosed by a $\psi$-positive counterclockwise directed cycle $D$, then $D$ must have an edge in $C_Y$. But this would imply the existence of a $\varphi$-positive counterclockwise directed cycle $D'\subset C\cup D$ enclosing $f$ and all the faces inside $C$: see Figure~\ref{fig:hyperflow}(b). This is impossible by the choice of the cycle $C$. Consider now a face $f$ inside $C$ and incident to an edge of $C$. This face cannot be inside a $\psi$-positive counterclockwise directed cycle $D$, otherwise $D$ would cross $C$, and there would be again a $\varphi$-positive counterclockwise directed cycle $D'\subset C\cup D$ enclosing more faces than $C$: see Figure~\ref{fig:hyperflow}(c). This is impossible by the choice of the cycle $C$. Thus $N(\psi)<N(\varphi)$ as wanted. This proves the existence to a minimal $\al$-hyperflow.

We now prove the uniqueness of the minimal $\al$-hyperflow. Suppose that $\varphi$ and $\psi$ are distinct $\al$-hyperflows. We want to show that they are not both minimal. Let $e_1$ be an edge such that $\varphi(e_1)<\psi(e_1)$. Let $x_1\in X$ and $y_1\in Y$ be the endpoints of $e$. Since 
$$\sum_{e \textrm{ incident to }y_1}\varphi(e)= \al(y_1)=\sum_{e \textrm{ incident to }y_1}\psi(e),$$
there exists an edge $e_1'\neq e_1$ incident to $y_1$ such that $\varphi(e_1')>\psi(e_1')$. Continuing in this way, one find a directed path made of edges $e_1,e_1',e_2,e_2',e_3,e_3',\ldots$ such that $\varphi(e_i)<\psi(e_i)$ and $\varphi(e_i')>\psi(e_i')$. This path will eventually intersect itself, so we get a directed simple cycle $C$ of $G$ such that $C$ is $\psi$-positive and the directed cycle $C'$ obtained by reversing $C$ is $\varphi$-positive. Either $C$ or $C'$ is counterclockwise, hence $\varphi$ and $\psi$ are not both minimal.
\end{proof}

\begin{remark}\label{rk:felsner} 
When $\al$ has only integer values and $G$ has at least one $\al$-hyperflow, 
more can be said on the structure of the set $K$ of $\al$-hyperflows of $G$ 
such that all flow-values are integers. 
By a result of Felsner and Knauer~\cite[Sec.4.2]{felsner2009uld} 
(extending an earlier result by Khuller et al.~\cite{khuller1993lattice}),
the set $K$ carries the structure of a distributive lattice (their result is formulated on flows
of directed graphs with prescribed flow-excess at each vertex, which are equivalent to our formulation
of $\al$-hyperflows upon orienting all the edges from black to white vertices); and naturally the minimum element
in the lattice is the minimal $\al$-hyperflow. This is an extension of a well-known result of Propp~\cite{propp2002lattice} and Felsner~\cite{Felsner:lattice} on $\al$-orientations of planar maps (an $\al$-orientation is an orientation 
where every vertex $v$ has outdegree $\al(v)$): Propp and Felsner have shown that, if non-empty, the set of 
$\al$-orientations of a map embedded in the plane is a distributive lattice, the minimum element of which is 
the unique $\al$-orientation with no clockwise cycle. 

About algorithmic aspects, it should be doable to compute the minimal $\al$-hyperflow in linear time \emph{once an $\al$-hyperflow is computed} (which has superlinear complexity as we have seen in Remark~\ref{rk:flow}), 
by extending the approach described in~\cite{brandes2000linear} for $\al$-orientations.  
\end{remark}

Lastly we prove an additional technical lemma about the minimal hyperflow.

\begin{lem}\label{lem:minimal-greater-on-outer} 
Let $G=(X\sqcup Y, E)$ and $\al$ be as in Lemma~\ref{lem:unique-minimal}, 
and let $\varphi_0$ be the minimal $\al$-hyperflow of $G$.
Let $x\in X$, $y\in Y$, and let $a=(x,y)$ be an edge of $G$ such that the face on the right of $a$ (when oriented from $x$ to $y$) is the outer face. If there is an $\al$-hyperflow $\varphi$ such that $\varphi(a)>0$, then $\varphi_0(a)>0$.
\end{lem}

\begin{proof}
Let $\varphi$ be a $\al$-hyperflow such that $\varphi(a)>0$.
It was shown in the proof of Lemma~\ref{lem:unique-minimal}, that the minimal $\al$-hyperflow $\varphi_0$ can be obtained from $\varphi$ by repeatedly pushing counterclockwise directed cycles. Moreover, because the face on the right of $a$ is the outer face, for any counterclockwise directed cycle $C$, the edge $a$ belongs to the subset $C_Y$ of edges of $C$ oriented toward a vertex in $Y$. Thus pushing cycles will only increase the value of the hyperflow on $a$, so  
$\varphi_0(a)\geq \varphi(a)>0$.
\end{proof}

\subsection{Proof of Theorem~\ref{theo:shifted-orientation-light} when the charge of every light face is equal to its degree}\label{sec:proof-case-d-light}
This subsection is devoted to the proof of the following result.
\begin{prop} \label{prop:case-degd}
Let $H$ be a light-rooted hypermap. Let $\si$ be a charge function which fits $H$ and such that every light face has charge equal to its degree. Then $H$ admits a unique $\si$-weighted hyperorientation in $\cHp$. 
\end{prop}

Throughout this subsection $(H,\si)$ is a charged hypermap satisfying the hypotheses of Proposition~\ref{prop:case-degd}. We say that a weighted hyperorientation of $H$ is \emph{$\rp$-weighted} if 0-way edges have weight 0, and 1-way edges have positive real weights. In fact, the $\si$-weighted hyperorientations of $H$ are precisely the $\rp$-weighted hyperorientations such that 
\begin{compactitem}
\item every vertex has weight $\si(v)$,
\item every inner dark face $f$ has weight $-\si(f)-\deg(f)$. 
\end{compactitem}
We will prove Proposition~\ref{prop:case-degd} in two steps. First we will establish the existence of a certain $\al$-hyperflow in a related graph $G_H$ using Lemma~\ref{lem:existence-hyperflow}, and then we will use this $\al$-hyperflow to define a $\si$-weighted hyperorientation of $H$.

We call \emph{star graph} of $H$ the bipartite graph $G_H$ (embedded in the plane) obtained as follows: for each dark face $h$ of $H$, place a vertex $y$ of $G_H$ inside $h$ and draw an edge $e$ of $G_H$ going from $y$ to each  corner of $h$. The construction is illustrated in Figure~\ref{fig:star-graph2}. We denote by $X$ the vertex set of $H$, and by $Y$ the remaining set of vertices of $G_H$ (which are placed inside the dark faces of $H$).

\fig{width=.6\linewidth}{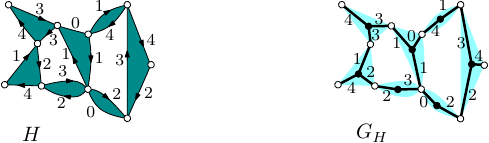}{A hypermap $H$ and the associated star graph $G_H$. The bipartite map $G_H$ is endowed with a hyperflow $\varphi$, while $H$ is endowed with the $\rp$-weighted hyperorientation $\Ga(\varphi)$.}

Given a hyperflow $\varphi$ of $G_H$, we define an $\rp$-weighted hyperorientations $\Ga(\varphi)$ of $H$ as follows: 
\begin{compactitem}
\item for every edge $e$ of $G_H$, we give weight $\varphi(e)$ to the edge $e'$ of $H$ preceding $e$ clockwise around the endpoint of $e$ in $X$
\item we orient $e'$ 1-way if $\varphi(e)>0$ and 0-way otherwise.
\end{compactitem}
The mapping $\Ga$ is illustrated in Figure~\ref{fig:star-graph2}. It is clear that $\Ga$ is a bijection between the hyperflows of $G_H$ and the $\rp$-weighted hyperorientations of $H$. Moreover, the $\varphi$-flow at a vertex $v$ of $G_H$ is equal to the weight of the corresponding vertex or dark face of $H$ in the hyperorientation $\Ga(\varphi)$. This proves the following result. 
\begin{lemma} \label{lem:bij-hyperflow}
The mapping $\Ga$ is a bijection between the $\si$-weighted hyperorientations of $H$ and the $\al$-hyperflows of $G_H$, where 
$\al$ is the function defined on $X\sqcup Y$ by 
\begin{itemize}
\item for every vertex $x\in X$, $\al(x)=\si(x)$,
\item for every vertex $y\in Y$, $\al(y)=-\si(f_y)-\deg(f_y)$, where $f_y$ is the dark face of $H$ containing~$y$.
\end{itemize}
\end{lemma}

We will now prove the existence of a minimal $\al$-hyperflow for $G_H$ by using Lemmas~\ref{lem:existence-hyperflow} and~\ref{lem:unique-minimal}.
\begin{lem}\label{lem:si-flow} 
Let $\al$ be the function defined in Lemma~\ref{lem:bij-hyperflow}.
Then $G_H$ admits a unique minimal $\al$-hyperflow $\varphi$. Moreover this hyperflow is accessible from every outer vertex of $H$.
\end{lem}

\begin{proof}
For $A\subseteq X$, we denote by $Y_A$ the set of vertices of $G_H$ placed in the inner dark faces of $H$ having all of their incident vertices in $A$ and we let
\begin{eqnarray*}
\al(A):=\sum_{x\in A}\al(x)-\sum_{y\in Y_A}\al(y).
\end{eqnarray*}
By Lemmas~\ref{lem:existence-hyperflow} and~\ref{lem:unique-minimal}, the existence and uniqueness of $\varphi$ are granted provided $\al(A)\geq 0$ for all $A\subseteq X$ with equality for $A=X$. 
We denote by $G_A=(A\cup Y_A,E_A)$ the subgraph of $G_H$ \emph{induced} by $A\cup Y_A$ (that is, $E_A$ is the set of edges of $G_H$ with both endpoints in $A\cup Y_A$). See for instance Figure~\ref{fig:subset-star-graph2}(a). 
Since $\al(X)$ is linear over the connected components of the subgraph $G_A$, it is sufficient to prove $\al(A)\geq 0$ when $G_A$ is connected (with equality for $A=X$). 
\fig{width=.5\linewidth}{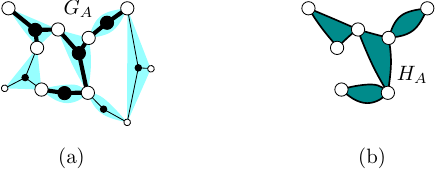}{(a) A subgraph $G_A=(A\cup Y_A,E_A)$ of the star graph $G_H$: the vertices in $A\cup Y_A$ are represented by big discs and the edges in $E_A$ are represented in bold lines. (b) The hypermap $H_A$.}

Let $A\subseteq X$ be such that $A\neq \emptyset$ and $G_A$ is connected. Observe that $G_A$ is the star graph of a hypermap $H_A$ with a light outer face: the set of vertices of $H_A$ is $A$ and the set of dark faces of $H_A$ is the set of inner dark faces of $H$ having all their incident vertices in $A$. See for instance Figure~\ref{fig:subset-star-graph2}(b). Let $D_A$ and $L_A$ be the set of dark and light faces of $H_A$. 
By definition of $\al$, we get 
$$\al(A)=\sum_{x\in A}\si(x)+\sum_{f\in D_A}(\si(f)+\deg(f))=\sum_{x\in A}\si(x)+\sum_{f\in D_A}\si(f)+\sum_{\ell \in L_A}\deg(\ell).$$
Now every face $\ell\in L_A$ corresponds to a light region of $H$, thus the $\si$-girth condition gives
$$\deg(\ell)\geq \sum_{x\textrm{ vertex of }H \textrm{ strictly inside }\ell}\si(x)+\sum_{f\textrm{ face of }H \textrm{ inside } \ell}\si(f),$$
with strict inequality for the outer face $\ell_0$ of $H_A$ if $\ell_0$ is not equal to the outer face of $H$. Thus, 
$$\sum_{\ell \in L_A}\deg(\ell)\geq \sum_{x\textrm{ vertex of }H\textrm{ not in }A}\si(x)+\sum_{f\textrm{ face of }H \textrm{ not in } D_A}\si(f),$$
with strict inequality if the outer face of $H_A$ is not equal to the outer face of $H$.
This gives 
$$\al(A)\geq \si_\tot=0.$$
Moreover, if one of the outer vertices is not in $A$, then one of the dark faces incident to the outer edges is not in $D_A$, hence the inequality is strict: $\al(A)> \si_\tot=0$. 
Thus, by Lemmas~\ref{lem:existence-hyperflow} and~\ref{lem:unique-minimal}, the graph $G_H$ admits a unique minimal $\al$-hyperflow $\varphi$, and $\varphi$ is accessible from every outer vertex.
\end{proof}

Next we use lemma~\ref{lem:si-flow} to prove the following.
\begin{lemma}\label{lem:almost-cop}
The hypermap $H$ admits a unique $\si$-weighted hyperorientation $\Om$ which is both minimal and accessible from the outer vertices. 
\end{lemma}

\begin{proof}
\textbf{Existence.}
By Lemma~\ref{lem:si-flow}, the bipartite graph $G_H$ admits a unique minimal $\al$-hyperflow $\varphi$.
By Lemma~\ref{lem:bij-hyperflow}, we know that $\Om=\Ga(\varphi)$ is a $\si$-weighted hyperorientation of $H$. We now want to prove that $\Om$ is both minimal and accessible from the outer vertices. 

We first prove that $\Om$ is minimal. Suppose, by contradiction, that there is a simple counterclockwise directed cycle $C$ of $\Om$ distinct from the outer face. We will show that in this case, there is a $\varphi$-positive counterclockwise directed cycle $D$ of $G_H$; see Figure~\ref{fig:ccw-cycle-GH}. Let $e_1,\ldots, e_k$ be the oriented edges of $C$. Let $f_i$ be the dark face of $H$ incident to $e_i$ (which is on the right of $e_i$). 
Let $x_i,x_i'\in X$ be the origin and end of $e_i$, and let $y_i\in Y'$ be the vertex of $G_H$ in the dark face $f_{i}$. Let $a_i,a_i'$ be the oriented edges $(x_i,y_i)$ and $(y_i,x_i')$; see Figure~\ref{fig:ccw-cycle-GH}. The edges $a_1,a_1',\ldots,a_k,a_k'$ form a counterclockwise directed cycle $D$ of $G_H'=(X\sqcup Y',E')$, hence it contains a simple counterclockwise directed cycle $D'$. By definition, the hyperflow $\varphi'(a_i')$ is equal to the weight of $e_i$ which is positive. Hence the counterclockwise directed cycle $D'$ is $\varphi$-positive. This is a contraction since the hyperflow $\varphi$ is minimal.

\fig{width=.28\linewidth}{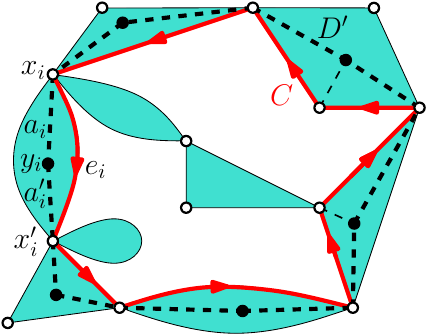}{A counterclockwise directed cycle $C$ of $H$ (thick lines) and the corresponding $\varphi$-positive counterclockwise directed simple cycle $D'$ of $G_H$ (thick dashed lines).}

We now want to prove that $\Om$ is accessible from every outer vertex of $H$. 
Let $u_0$ be an outer vertex of $H$, and let $v$ be an inner vertex. We want to exhibit a directed path of $H$ from $u_0$ to $v$. We know (by Lemma~\ref{lem:si-flow}) that the hyperflow $\varphi$ of $G_H$ is accessible from $u_0$, hence for every vertex $u$ of $H$ there exists a $\varphi$-positive path of $G_H$ from $u_0$ to $u$. For a vertex $u\neq u_0$ of $H$, we consider the set $A_u$ of all the edges of $G_H$ incident to $u$ which are part of a $\varphi$-positive simple directed path of $G_H$ from $u_0$ to $u$. For a 1-way edge $e$ of $H$ having origin $u\neq u_0$, we denote by $\th(e)$ the edge of $A_u$ preceding $e$ in clockwise direction around $u$, and we denote by $\pi(e)$ the edge of $H$ following $\th(e)$ around $u$; see Figure~\ref{fig:proof-accessibility}(a). By definition, $\varphi(\th(e))>0$ hence $\pi(e)$ is a 1-way edge of $H$ directed toward $u$ in $\Om$. Moreover, there exists no $\varphi$-positive (simple) directed path of $G_H$ from $u_0$ to $u$ ending between $e$ and $\pi(e)$ in clockwise direction around $u$. 
We now construct a directed path of $H$ ending at $v$ as follows\footnote{Our construction corresponds to the so-called \emph{leftmost path} which has proved useful for other bijective results on maps.}. First we choose an edge $a\in A_v$ and denote by $e_0$ the edge of $H$ following $a$ clockwise around $v$. The edge $e_0$ is a 1-way edge oriented toward~$v$. Then we define some edges $e_1,e_2,\ldots$ as follows: for all $i\geq 0$, if the origin of the 1-way edge $e_i$ is distinct from $u_0$ we define $e_{i+1}=\pi(e_{i})$. We will now prove that there exists $i$ such that the origin of $e_i$ is $u_0$ (so that $e_i,e_{i-1},\ldots,e_0$ is a directed path from $u_0$ to $v$). Suppose the contrary. In this case, there must exist integers $i<j$ such that the origin of $e_j$ is the end of $e_i$, and we consider the least such $j$. The edges $e_i,e_i+1,\ldots,e_{j}$ form a simple directed cycle $C$ of $H$, which is not the outer face of $H$. And since $\Om$ has no counterclockwise directed cycle, except for the outer faces of $H$, the cycle $C$ is directed clockwise. The situation is represented in Figure~\ref{fig:proof-accessibility}(b). Let $u$ be the end of $e_i$ (also the origin of $e_j$). By definition of $e_i$, the edge of $G_H$ preceding $e_i$ around $u$ is part of a $\varphi$-positive directed path $P$ of $G_H$ from $u_0$ to $u$. Hence, the path $P$ must intersect the cycle $C$. We denote by $w$ the first vertex of $C$ on the path $P$ from $u_0$ to $u$, and by $e_k$ the edge of $C$ with origin $w$ (with $i\leq k<j$). Note that the directed path $P$ arrives at $w$ between $e_{k+1}=\pi(e_k)$ and $e_k$ in clockwise direction around $w$. This is impossible by definition of $\pi$. This completes the proof that there is a directed path from $u_0$ to $v$ in $\Om$. Hence the hyperorientation $\Om$ is accessible from every outer vertex of $H$.

\fig{width=.7\linewidth}{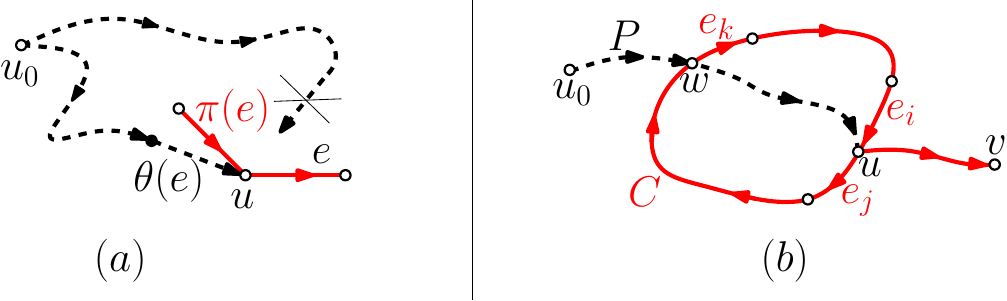}{(a) Definition of the edge $\pi(e)$, for a 1-way edge $e$ of $H$ with origin $u\neq u_0$. The $\varphi$-positive paths of $G_H$ are represented in dashed lines. (b) The cycle $C=\{e_i,e_{i+1},\ldots,e_j\}$ of $H$, and the $\varphi$-positive path $P$ of $G_H$ (represented in dashed line).}

\ni \textbf{Uniqueness.}
Suppose that $\wOm$ is a $\si$-weighted hyperorientation of $H$ which is minimal and accessible from every outer vertex of $H$.
We want to prove that $\wOm=\Om$. 
It suffices to prove that the hyperflow $\wt{\varphi}:=\Ga^{-1}(\wOm)$ of $G_H$ is equal to $\varphi$. By Lemma~\ref{lem:bij-hyperflow}, we know that $\wt{\varphi}$ is an $\al$-hyperflow of $G_H$. Hence by Lemma~\ref{lem:unique-minimal}, it suffices to prove that $\wt{\varphi}$ is minimal. Suppose, by contradiction that $\wt{\varphi}$ is not minimal, and consider a simple $\wt{\varphi}$-positive counterclockwise cycle $D$ of $G_H'$. We will now exhibit a counterclockwise directed cycle $C$ of $\wOm$. 
For a vertex $x$ of $H$ on the cycle $D$ we consider the edge $a_x$ of $D$ oriented toward $x$, and the edge $e_x$ of $H$ following $a_x$ clockwise around $x$; see Figure~\ref{fig:ccw-cycle-H}(a). Since $\varphi(a_x)>0$, the edge $e_x$ is 1-way oriented toward $x$ in $\wOm=\Ga(\wt{\varphi})$. The origin $x'$ of $e_x$ is either on the cycle $D$ or strictly inside $D$. If $x'$ is strictly inside $D$, we consider a directed path of $\wOm$ going from an outer vertex of $H$ to $x'$ (we know that such a path exists since $\wOm\in\cHp$). We extract from this path a directed path of $\wOm$ starting at a vertex on the cycle $D$, staying strictly inside $D$ and ending at $x'$. We denote by $Q_x$ the directed path of $\Om$ made of $P_x$ followed by the edge $e_x$, with the convention that $P_{x}$ is empty if the origin $x'$ of $e_x$ is on $D$. With this convention, for all $x$ of $H$ on the cycle $D$, the directed path $Q_x$ of $\wOm$ starts at a vertex of $D$, stays strictly inside $D$ and ends at $x$. We now consider a vertex $x_0$ of $H$ on $D$, and for all $i\geq 0$ we denote by $x_{i+1}$ the origin of $Q_{x_i}$. The infinite path $\cup_{i=0}^\infty Q_{x_i}$ stays inside $D$ and must intersect itself. Let $n$ be the largest integer such that the path $Q=\cup_{i=0}^{n-1} Q_{x_i}$ from $x_n$ to $x_0$ is simple. Since $Q$ is simple, it cuts the interior of the cycle $D$ into two regions. Moreover, the edge $e_{x_n}$ is easily seen to be in the region on the left of $Q$. Therefore the path $\cup_{i=0}^n Q_{x_i}$ contains a counterclockwise cycle; see Figure~\ref{fig:ccw-cycle-H}(b). This implies that $\wOm$ is not minimal, a contradiction. 
\end{proof}

\fig{width=.7\linewidth}{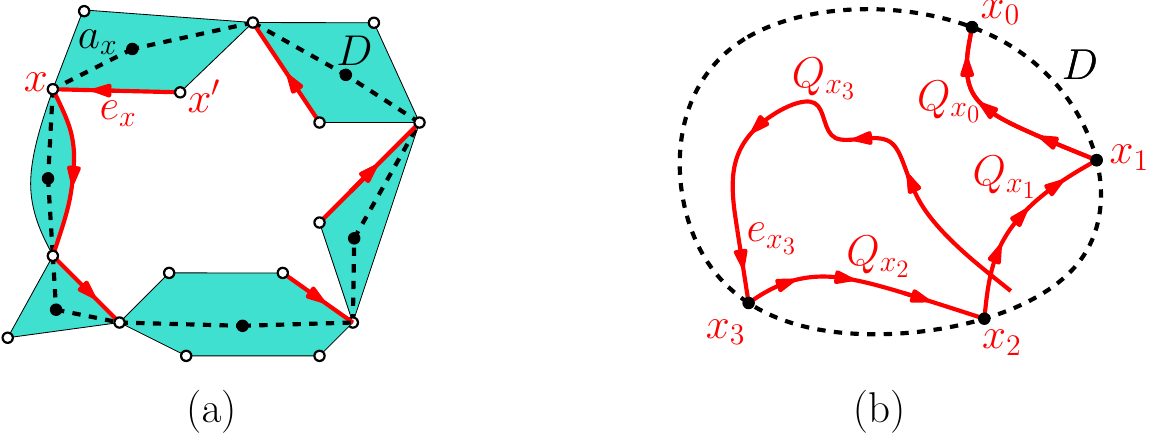}{(a) The counterclockwise cycle $D$ of $G_H$ (dashed lines), a vertex $x$ of $H$ on $D$ and the 1-way edge $e_{x}$ of $H$. (b) The directed paths $Q_{x_0},Q_{x_1},\ldots$ of $H$ inside the cycle $D$ of $G_H$ and a counterclockwise cycle $C$ of $H$ contained in $\cup_{i=0}^3 Q_{x_i}$.}

\ni \textbf{Proof of Proposition~\ref{prop:case-degd}}. 
We now complete the proof of Proposition~\ref{prop:case-degd}. 
By Lemma~\ref{lem:almost-cop} there is a unique $\si$-weighted hyperorientation $\Om$ which is both minimal and accessible from the outer vertices. In order to complete the proof of Proposition~\ref{prop:case-degd} we need to prove that $\Om$ is in $\cHp$, that is, we need to prove that the outer face of $H$ is a clockwise directed cycle. Hence \emph{it suffices to prove that every outer edge of $H$ has positive weight} (hence is 1-way). 

Let $e_0$ be an outer edge of $H$. We denote by $w(a)$ the weight of a vertex, edge or face $a$ in $\Om$. We want to prove $w(e_0)>0$. 
Let us first treat the case where $e_0$ is a loop. Let $f_0$ be the light outer face and let $f$ be the dark inner face incident to $e_0$. The light region $R=\{f_0,f\}$ satisfies $|\R|=\deg(f_0)-1$ and $\si(R)=\si(f_0)+\si(f)=\deg(f_0)+\si(f)$. Thus the $\si$-girth condition gives $-1>\si(f)$. Hence $w(e_0)=w(f)=-1-\si(f)>0$ as wanted.

We now suppose that $e_0$ is not a loop and want to prove $w(e_0)>0$. We consider the hypermap $H'$ obtained from $H$ by adding two edges $e',e''$ with the same endpoints as $e_0$ in the outer light face of $H$ as indicated in Figure~\ref{fig:proof-cop4}(a). In $H'$, the edges $e_0$ and $e'$ enclose an inner light face $f'$ of degree 2, while the edges $e'$ and $e''$ enclose an inner dark face $f''$ of degree 2. Let 
$$\ds \eps=\frac{1}{2}\min_{R}\left(|\R|-\si(R)\right).$$ 
where $R$ ranges over all light regions containing strictly at least one of the outer edges. 
Since $H$ satisfies the $\si$-girth condition, we have $\eps>0$. Let $f$ be the dark face incident to $e_0$ and let $\si'$ be the charge function of $H'$ defined by $\si'(f)=\si(f)+\eps$, $\si'(f')=2$, $\si'(f'')=-2-\eps$, and $\si'(a)=\si(a)$ for any vertex or face $a\notin\{f,f',f''\}$ of $H'$. 

\fig{width=\linewidth}{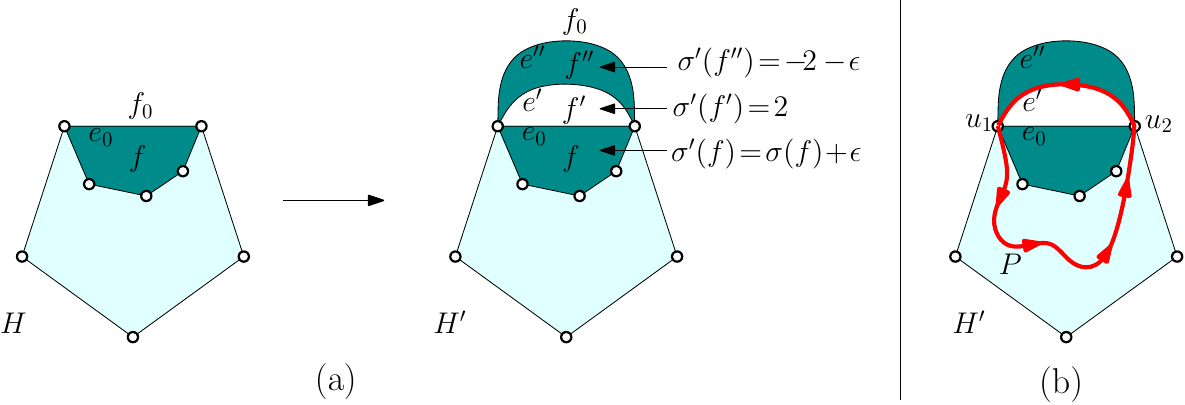}{(a) 
(a) The hypermap $H'$ obtained from $H$ by adding two edges $e',e''$ with the same endpoints as $e_0$ (and conveniently 
 assigning charges for the new faces and the dark face incident to $e_0$). (b) If $e'$ was 1-way, by the accessibility properties of $\cHp$, it would yield a $P$ forming with $e'$ a counterclockwise cycle (shown in red), a contradiction.
.}

\begin{claim}
The charge function $\si'$ fits $H'$.
\end{claim}

\begin{proof}
It is easy to see that the charge $\si'(v)=\si(v)$ of every vertex is positive, the charge $\si'(f_0)=\si(f_0)$ of the light outer face $f_0$ is $\deg(f_0)$, and $\si'_\tot=\si_\tot=0$. It remains to prove that $H'$ satisfies the $\si'$-girth condition. 
Let $R'$ be a light region of $H'$. We need to prove $|\R'|\geq \si'(R')$, with strict inequality if $R'$ strictly contains an outer edge.

First suppose that $f''\in R'$. In this case $f_0,f'\in R'$ (because $R'$ is a light region). Let $R$ be the light region of $H$ obtained from $R'$ by removing $f',f''$. If $f\in R$, then $R$ strictly contains the outer edge $e_0$ so that
$$|\R'|=|\R|\geq \si(R)+2\eps=\si'(R')+2\eps,$$
while if $f\notin R$, then
$$|\R'|=|\R|\geq \si(R)=\si'(R')+\eps.$$

Next suppose that $f''\notin R'$ and $f\notin R'$. If $f'\notin R'$ then $R'$ is a light region of $H$ and we get
$$|\R'|\geq \si(R')=\si'(R'),$$
 with strict inequality if $R'$ strictly contains an outer edge.
If $f'\in R'$ then we consider the light region $R$ of $H$ obtained from $R'$ by removing $f'$. 
We get 
$$|\R'|=|\R|+2\geq \si(R')+2=\si'(R),$$
 with strict inequality if $R'$ (hence $R$) strictly contains an outer edge.

Lastly suppose that $f''\notin R'$ and $f\in R'$. In this case $f'\in R'$. If $f_0\in R$, we consider the light region $R$ of $H$ obtained from $R'$ by removing $f'$. Since $R$ strictly contains the outer edge $e_0$ we get 
$$|\R'|=|\R|+2\geq \si(R)+2\eps+2=\si'(R')+\eps$$
If $f_0\notin R$, then we consider the light region $R$ of $H$ obtained from $R'$ by removing $f'$ and adding $f_0$. 
Since $R$ strictly contains the outer edge $e_0$ we get 
$$|\R'|=|\R|+2-\deg(f_0)\geq \si(R)+2\eps+2-\deg(f_0)=\si'(R')+\eps.$$
\end{proof}

Since $H'$ satisfies the $\si'$-girth condition, Lemma~\ref{lem:almost-cop} implies that $H'$ admits a $\si'$-weighted hyperorientation $\Om'$ which is minimal and accessible from every outer vertex. Let $e_1$ be the edge preceding $e_0$ in clockwise order around $f$. We define a weighted hyperorientation $\wt{\Om}$ of $H$ by setting $\wt{w}(e_0)=w'(e_0)+w'(e'')$, $\wt{w}(e_1)=w'(e_1)+w'(e')$, and $\wt{w}(e)=w'(e)$ for all edge $e\neq e_0,e_1$ of $H$ (as usual the edge in $\wt{\Om}$ are 1-way if they have positive weight and 0-way otherwise). We denote by $w'(a)$ and $\wt{w}(a)$ respectively the weight of a vertex, edge or face $a$ in $\Om'$ and $\wt{\Om}$. Recall that all the weights $w'(a)$ are non-negative, hence the weights $\wt{\Om}(a)$ are non-negative.

\begin{claim}\label{claim:outer-edge-1way}
The hyperorientation $\wt{\Om}$ is $\si$-weighted. Moreover $\wt{w}(e_0)>0$.
\end{claim}

\begin{proof}
It is easily seen that the weight of every vertex is the same in $\Om'$ and $\wt{\Om}$ (also for the endpoints of $e_0$). Moreover the weight of every face of $H$ is the same in $\Om'$ and $\wt{\Om}$ except for the dark face $f$. For the dark face $f$ we have 
$$\wt{w}(f)=w'(f)+w'(e')+w'(e'')=w'(f)+w'(f'')=-\si'(f)-\deg(f)-\si'(f'')-2=-\si(f)-\deg(f),$$
as wanted. Since $\Om'$ is $\si'$-weighted, this shows that $\wt{\Om}$ is $\si$-weighted.

It remains to show that $\wt{w}(e_0)>0$. It suffices to show $w'(e'')>0$. Suppose by contradiction that $w'(e'')=0$. 
In this case $w'(e')=w'(f'')=\eps>0$, so that $e'$ is 1-way and $e''$ is 0-way in the hyperorientation $\Om'$. Let $u_2$ and $u_1$ be the origin and end of $e'$ as indicated in Figure~\ref{fig:proof-cop4}(b). Since $\Om'$ is accessible from the outer vertex $u_1$, there is a directed path $P$ from $u_1$ to $u_2$. This path does not use the outer edge $e''$ which is 0-way, hence the path $P$ together with $e'$ form a counterclockwise directed cycle; see Figure~\ref{fig:proof-cop4}(b). This is a contradiction since $\Om'$ is minimal. 
\end{proof}

We know $w(\wt{e_0})>0$ and want to prove $w(e_0)>0$. For this we use Lemma~\ref{lem:minimal-greater-on-outer}.
Since the hyperorientation $\wt{\Om}$ of $H$ is $\si$-weighted, we know by Lemma~\ref{lem:bij-hyperflow} that the hyperflow $\wt{\varphi}=\Gamma^{-1}\left(\wt{\Om}\right)$ is an $\al$-hyperflow of $G_H$. Let $a_0$ be the edge of $G_H$ preceding $e_0$ clockwise around the outer vertex $u_2$ of $H$. By definition of $\Gamma$, $\wt{\varphi}(a_0)=\wt{w}(e_0)>0$. Hence by Lemma~\ref{lem:minimal-greater-on-outer}, the minimal $\al$-hyperflow of $G_H$ satisfies $\varphi(a_0)>0$. Moreover the hyperorientation $\Om$ is equal to $\Gamma(\varphi)$ (see the proof of Lemma~\ref{lem:almost-cop}). Thus $w(e_0)=\varphi(a_0)>0$.
This completes the proof of Proposition~\ref{prop:case-degd}.

\subsection{Proof of Theorem~\ref{theo:shifted-orientation-light}}\label{sec:endproof-outer-light}
In this section we complete the proof of Theorem~\ref{theo:shifted-orientation-light}. 
We consider a light-rooted hypermap $H$, and a charge function $\si$ fitting $H$. We want to prove that $H$ admits a unique $\si$-weighted hyperorientation in $\cHp$. Our strategy is as follows. First, we will construct a related hypermap $\Hk$ and a fitting charge function $\sik$ satisfying the condition of Proposition~\ref{prop:case-degd}. This grants the existence of a unique $\sik$-weighted hyperorientation $\Omk$ in $\cHp$ for $\Hk$. We will then construct from $\Omk$ a hyperorientation $\Om$ of $H$, and prove that it is the unique $\si$-weighted hyperorientation of $H$ in $\cHp$.

Let $k$ be an integer greater than 
$$1+|E_0|+\sum_{a\in A}|\si(a)|,$$
where $E_0$ is the set of edges of $H$, and $A$ is the set of all vertices and faces of $H$.
Let $H_k$ be the hypermap obtained from $H$ by subdividing every edge into a path of length $k$. Hence, every face of degree $\de$ of $H$ corresponds to a face of degree $k\de$ of $H_k$. We now consider a hypermap $\Hk$ obtained by adding a dark face of degree $k(k-1)\de$, called a \emph{sea-star}, inside each inner light face of degree $k\de$ of $H_k$; see Figure~\ref{fig:sea-star2}. More precisely, $\Hk$ is obtained by adding the sea-stars inside the inner light faces of $H_k$ in such a way that every inner light faces of $\Hk$ has degree $k$ and is incident to one edge of $H_k$ and $k-1$ edges of a sea-star. We call \emph{sea-edges} the edges of $\Hk$ incident to sea-stars. For an edge $e=(u,v)$ of $H_k$, we call \emph{sea-arc} associated with $e$ the path of $\Hk$ made of the $k-1$ sea-edges around the face of $\Hk$ incident to $e$. 
For an edge $e=(u,v)$ of $H$, we call \emph{sea-path} associated with $e$ the path of $\Hk$ from $u$ to $v$ (of length $k(k-1)$) made of the $k$ sea-arcs associated with the edges of $H_k$ subdividing $e$.

We define a charge function $\sik$ of $\Hk$ as follows:
\begin{itemize}
\item $\sik(v)=k\si(v')$ if $v$ is a vertex of $\Hk$ corresponding to a vertex $v'$ of $H$ and $\sik(v)=k$ otherwise,
\item $\sik(f)=\deg (f)$ if $f$ is a light face (hence $\sik(f)=k$ for every inner light face $f$),
\item $\sik(f)=k\si(f')-k^2\de +k\de$ if $f$ is a dark face of $\Hk$ of degree $k\de$ corresponding to a dark face $f'$ of $H$ (of degree $\de$),
\item $\sik(f)=k\si(f')-k^2(k-1)\de$ if $f$ is a sea-star of degree $k(k-1)\de$ corresponding to a light face $f'$ of $H$ (of degree $\de$).
\end{itemize}

\fig{width=.7\linewidth}{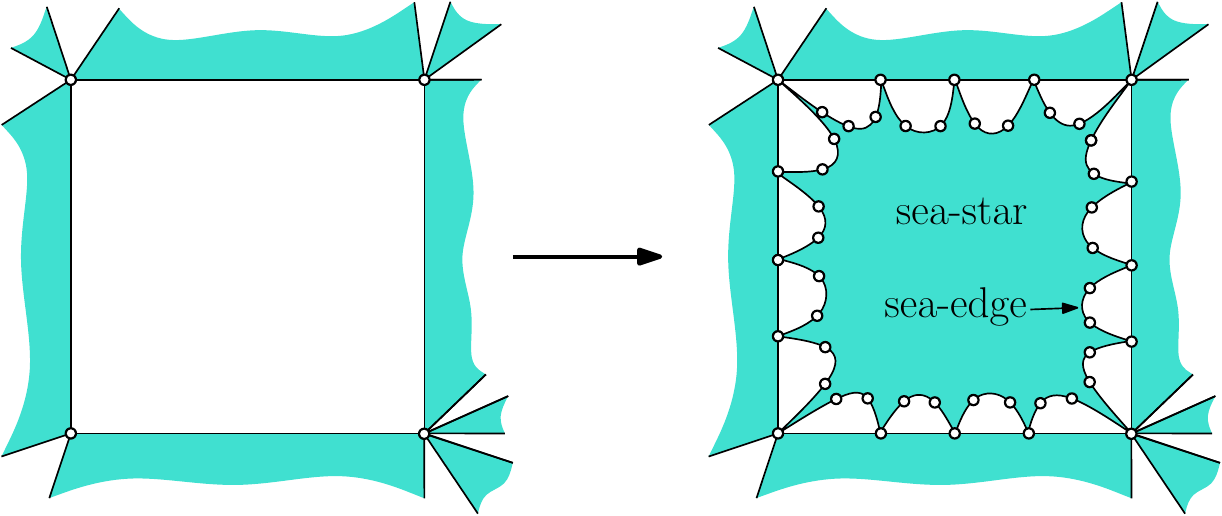}{Left: A face $f$ of $H$. Right: the face $f$ after subdividing each edge into a path of length $k=4$, and adding the sea-star of $\Hk$ inside $f$.}

\begin{claim}\label{claim:girth-Hk} 
The charge function $\sik$ fits $\Hk$.
\end{claim}

\begin{proof}
First observe that the charge $\sik(v)=k\si(v)$ is positive for every vertex $v$, and the charge of the outer face is equal to its degree. We now show that $\sik_\tot=0$. Let $V_0$, $E_0$, $F_0$, $S_0$, and $K_0$ be respectively the set of vertices, edges, light faces, sea-stars, and non-sea-star dark faces of $\Hk$. Let $E_1$ and $E_2$ be respectively the set of edges of $\Hk$ incident to sea-stars, and to non-sea-star dark faces of $\Hk$. We have
\begin{eqnarray*}
\sik_\tot&=&\sum_{v\in V_0}\sik(v)+\sum_{f\in F_0}\sik(f)+\sum_{f\in S_0}\sik(f)+\sum_{f\in K_0}\sik(f)\\
&=& k\bigg(|V_0|-|V_0'|+\sum_{v'\in V_0'}\si(v')\bigg)+k\bigg(|F_0|+\deg(f_0')-1\bigg)\\
&&+k\bigg(-|E_1|+\sum_{f'\in S_0'}\si(f')\bigg)+k\bigg(|E_0'|-|E_2|+\sum_{f'\in K_0'}\si(f')\bigg)\\
&=& k\bigg(|V_0|+|F_0|-|E_0|-|V_0'|+|E_0'|-1+\si_\tot\bigg)\\
\end{eqnarray*}
where $f_0'$ is the outer face of $H$, and $V_0'$, $E_0'$, $S_0'$, $K_0'$ are respectively the set of vertices, edges, inner light faces, and dark faces of $H$ (the last identity uses $|E_0|=|E_1|+|E_2|$ and $\si(f_0')=\deg(f_0')$). The Euler relation gives
$$|V_0|+|F_0|-|E_0|=-|S_0|-|K_0|+2=-|S_0'|-|K_0'|+2=|V_0'|-|E_0'|+1,$$
because $|S_0|=|S_0'|$ and $|K_0|=|K_0'|$. Moreover $\si_\tot=0$, hence $\sik_\tot=0$.

It remains to prove that $\Hk$ satisfies the $\sik$-girth condition. Let $R$ be a light region of $\Hk$. We want to prove $|\R|\geq \sik(R)$ with strict inequality if one of the outer edges is strictly contained in $R$. By Lemma~\ref{lem:simply-connected-sufficient} we can assume that the light region $R$ is simply connected.
Let $V$, $E$, $F$, $S$, and $K$ be respectively the set of vertices strictly inside $R$, edges strictly inside $R$, light faces inside $R$, sea-stars inside $R$, and non-sea-star dark faces inside $R$. Similarly 
as in the above computation of $\sik_\tot$, we have 
\begin{eqnarray*}
\sik(R)&=&\sum_{v\in V}\sik(v)+\sum_{f\in F}\sik(f)+\sum_{f\in K}\sik(f)+\sum_{f\in S}\sik(f)\\
&=& k\,\left(|V|+|F|-|E|-|V'|+|E'|+\textbf{1}_{f_0\in R}\cdot(\deg(f_0')-1)+\sum_{a\in V'\cup S'\cup K'}\si(a)\right),
\end{eqnarray*}
where $f_0$ is the outer face of $\Hk$, $V'$ is the set of vertices of $H$ corresponding to vertices of $\Hk$ in $V$, $S'$ is the set of inner light faces of $H$ corresponding to sea-stars in $S$, $K'$ is the set of dark faces of $H$ corresponding to dark faces in $K$, and $E'$ is the set of edges of $H$ incident to faces in $K'$.
Since $R$ is simply connected, the Euler relation gives $\ds |V|+|F|-|E|=-|S'|-|K'|+1$, hence
$$\sik(R)=k\,\bigg(|E'|-|V'|-|S'|-|K'|+1+\textbf{1}_{f_0\in R}(\deg(f_0')-1)+\sum_{a\in V'\cup S'\cup K'}\si(a)\bigg).$$
(In particular, by the choice of $k$, $\sik(R)<k(k-1)$.) 

We now prove $|\R|\geq \sik(R)$ with strict inequality if one of the outer edges is strictly contained in $R$. 
Suppose first that $R$ contains an inner light face $f$ but contains none of the two dark faces incident to $f$. Since $R$ is connected, we have $R=\{f\}$ and $|\R|=k=\si(R)$. Next, suppose that $R$ contains an inner light face and the incident non-sea-star dark face $f$, but not the incident sea-star $s$. Since $f\in R$ all the incident light faces are in $R$ (because $R$ is a light region) hence $C$ contains an entire sea path. Thus 
$|\R|\geq k(k-1)>\sik(R)$
by the choice of $k$. Lastly suppose that for every light face $f$ in $R$, the sea-star incident to $f$ is also in $R$. In this case, we consider the light region $R'$ of $H$ defined by $R'=K'\cup S'$ if $f_0\notin R$ and $R'=K'\cup S'\cup\{f_0'\}$ if $f_0\in R$. We have $|\R|=k|\R'|$. Moreover $V'$, $E'$, $S'$, and $K'$ are respectively the sets of vertices strictly inside $R'$, edges strictly inside $R'$, inner light faces inside $R'$, and dark faces inside $R'$, so that the Euler relation gives 
$$|E'|-|V'|-|S'|-|K'|+1-\textbf{1}_{f_0\in R'}=0.$$
Hence, using $\deg(f_0')=\si(f_0')$ we get
$$\sik(R)=k\bigg(\textbf{1}_{f_0\in R'}\si(f_0')+\sum_{a\in V'\cup S'\cup K'}\si(a)\bigg)=k\,\si(R').$$
Thus $|\R|=k|\R'|\geq k\,\si(R')=\sik(R)$ with strict inequality if one of the outer edges is strictly contained in $R$.
\end{proof}

By Claim~\ref{claim:girth-Hk} and Proposition~\ref{prop:case-degd}, the hypermap $\Hk$ admits a unique $\sik$-weighted hyperorientation $\Omk$ in $\cHp$. We now establish a few properties of $\Omk$. We denote by $w(a)$ the weight of an edge or face $a$ of $\Hk$ in the hyperorientation $\Omk$. Note that all the weights are non-negative because every light face of $\Hk$ has degree equal to its charge. 
Let $a$ be an inner edge of $H_k$ and let $P$ be the associated sea-arc. The $k-2$ first edges of $P$ are forced to have weight $k$, and we denote by $w'(a)$ the weight of the last edge of $P$; see Figure~\ref{fig:sea-path2}. Let $f$ be an inner light face of $H$ of degree $\de$, let $f_k$ be the corresponding light face of $H_k$ and let $s$ be the corresponding sea-star of $\Hk$. For the edges $e_1,\ldots,e_{k\de}$ incident to $f_k$ we get 
\begin{eqnarray}\label{eq:weight-sea-star}
w'(e_1)+\ldots+w'(e_{k\de})&=&w(s)-k^2(k-2)\de~=~-\sik(s)-\deg(s)-k^2(k-2)\de \nonumber\\
&=&-k\si(f)+k^2(k-1)\de-k(k-1)\de-k^2(k-2)\de\nonumber\\
&=&-k\si(f)+k\de.
\end{eqnarray}

\fig{width=.6\linewidth}{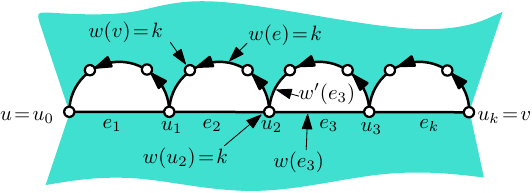}{The sea-path associated with an inner edge $e=(u,v)$ of $H$ for $k=4$. The edges with weights $w(e_i)$ and $w'(e_i)$ are indicated for $i=3$.}

\begin{claim}\label{claim:config-sea-path}
Let $e$ be an inner edge of $H$. 
Let $e_1,\ldots,e_k$ be the edges of $H_k$ subdividing $e$ in clockwise order around the incident dark face.
There exists $j\in\{1,\ldots,k\}$ such that $w(e_i)=k$ for all $i<j$, $w(e_i)=0$ for all $i>j$. Moreover $w'(e_1)=0$, and $w'(e_{i+1})=k-w(e_{i})$ for all $i\in\{1,\ldots,k-1\}$. 
\end{claim}
The situation described by Claim~\ref{claim:config-sea-path} is represented in Figure~\ref{fig:reduction-sea-path2} (first line).

\begin{proof} 
We let $e_{i}=(u_{i-1},u_i)$, with $u_0=u$, and $u_k=v$. Since for all $i\in\{1,\ldots,k-1\}$ the weight of the vertex $u_i$ is $k$, we get $w'(e_{i+1})=k-w(e_{i})$. Since $\Om$ is minimal, the weights $w(e_i)$ and $w'(e_i)$ cannot both be positive (otherwise the incident light face would be oriented counterclockwise). 
Thus if $w(e_i)\neq k$ for $i<k$, then $w'(e_{i+1})\neq 0$, hence $w(e_{i+1})=0$. This proves the existence of $j\in\{1,\ldots,k\}$ such that $w(e_i)=k$ for all $i<j$, and $w(e_i)=0$ for all $i>j$. Lastly, suppose by contradiction that $w'(e_1)>0$. In this case $w(e_1)=0$, and $w'(e_i)=k$ for all $i\in\{2,\ldots,k\}$. Thus 
$$w'(e_1)+\ldots+w'(e_{k})=w'(e_1)+k(k-1)>k(k-1).$$
By our choice of $k$ this contradicts \eqref{eq:weight-sea-star}.
\end{proof}

We now associate with the weighted hyperorientation $\Omk$ of $\Hk$ a weighted hyperorientation $\bOm$ of $H$; see Figure~\ref{fig:reduction-sea-path2} (note that, by these rules, an edge $e\in\bOm$ is 1-way iff the associated $j$ defined in Claim~\ref{claim:config-sea-path} is equal to $k$).   
Let $e$ be an edge of $H$ and let $e_1,\ldots,e_k$ be the edges of $H_k$ subdividing $e$ in clockwise order around the incident dark face. 
We define the weight $\bw(e)$ of $e$ in $\bOm$ to be 
$$\bw(e)=\frac{\sum_{i=1}^k w(e_i)}{k}-(k-1)$$ 
and we orient $e$ 1-way if the weight is positive and 0-way otherwise. 
Note that for any outer edge $e$, the edges $e_1,\ldots,e_{k-1}$ are all 1-way of weight $k$ and $e_k$ is also 1-way (because $\Omk\in\cHp$), so that 
$$\bw(e)=w(e_k)/k>0,$$ 
hence $e$ is 1-way.

\fig{width=\linewidth}{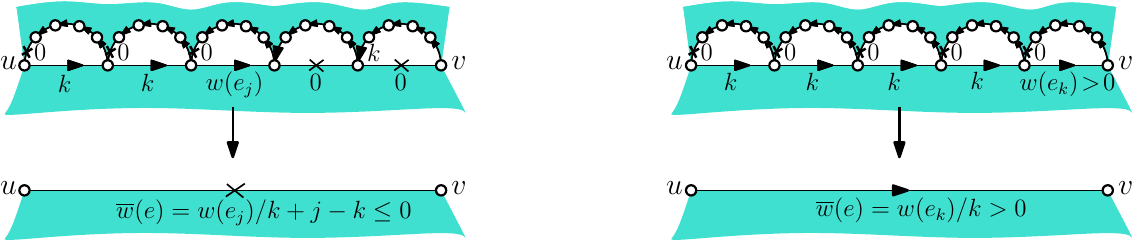}{Top part: the possible configurations of weights along a sea-path in the hyperorientation $\Omk$ of $\Hk$, as described by Claim~\ref{claim:config-sea-path}. Bottom part: the corresponding weight in the hyperorientation $\bOm$ of $H$. In the case $j<k$ (left) one gets 
$\bw(e)=\frac{\left(\sum_{i=1}^k w(e_i)\right)}{k}-(k-1)=w(e_j)/k-(j-1)\leq 0$, while in the case $j=k$ one gets $\bw(e)=w(e_k)/k> 0$.
}

We will now complete the proof of Theorem~\ref{theo:shifted-orientation-light} by proving the following claim.

\begin{claim}\label{lem:exists-kde-orient}
The hyperorientation $\bOm$ is the unique $\si$-weighted hyperorientation of $H$ in $\cHp$.
\end{claim}

\begin{proof}
We denote by $w(a)$ (resp. $\bw(a)$) the weight of a vertex, edge or face $a$ in the hyperorientation $\Omk$ of $\Hk$ (resp. $\bOm$ of $H$). If $e$ is an inner edge of $H$, adopting the notation of Claim~\ref{claim:config-sea-path} gives 
$$\bw(e)=\frac{1}{k}\Big(w(e_k)-\big(\sum_{i=1}^k w'(e_i)\big)\Big).$$
Moreover, since $w(e_k)>0$ if and only if $\sum_{i=1}^k w'(e_i)=0$, we get 
\begin{equation}\label{eq:bweight-positive}
\max(\bw(e),0)=\frac{w(e_k)}{k},
\end{equation}
and 
\begin{equation}\label{eq:bweight-negative}
\min(\bw(e),0)=-\frac{\sum_{i=1}^k w'(e_i)}{k}.
\end{equation}
Now if $e$ is an outer edge $e$, then the edges $e_1,\ldots,e_{k-1}$ are all 1-way of weight $k$ and $e_k$ is also 1-way (because $\Omk\in\cHp$), so that 
$$\max(\bw(e),0)=\bw(e)=w(e_k)/k,$$
and $e$ is oriented 1-way in $\Omk$. 

We will now prove that $\bOm$ is $\si$-weighted. 
Let $u$ be a vertex of $H$. We observe that Claim~\ref{claim:config-sea-path} (more precisely, the statement $w'(e_1)=0$ in this claim) implies that no sea-edge is oriented 1-way toward $u$. Thus, the weight $w(u)$ is equal to the sum of the weights of the edges of $H_k$ oriented toward $u$. Hence \eqref{eq:bweight-positive} gives
$$\bw(u)=\sum_{e\textrm{ oriented toward } u \textrm{ in }H}\max(\bw(e),0)=\sum_{e'\textrm{ oriented toward } u \textrm{ in }H_k}\frac{w(e')}{k}=\frac{w(u)}{k}=\frac{\sik(u)}{k}=\si(u),$$
as wanted.
Now let $f$ be a light inner face of $H$ of degree $\de$, and let $f'$ be the corresponding face in $H_k$. By \eqref{eq:bweight-negative}, the weight of $f$ in $\bOm$ is 
$$\bw(f)=\sum_{e\textrm{ incident to } f \textrm{ in }H}\min(\bw(e),0)=-\sum_{e_i\textrm{ incident to } f' \textrm{ in }H_k}\frac{w'(e_i)}{k}.$$
Hence, \eqref{eq:weight-sea-star} gives $\bw(f)=\si(f)-\de$ as wanted.
Now, let $f$ be a dark inner face of $H$ of degree $\de$, and let $f'$ be the corresponding face of $\Hk$. The weight of $f$ in $\bOm$ is 
\begin{eqnarray*}
\bw(f)&=&\sum_{e\textrm{ incident to } f}\bw(e)=\left(\sum_{e'\textrm{ incident to } f'}\frac{w(e')}{k}\right)-(k-1)\de= \frac{w(f')}{k}-(k-1)\de\\
&=&\frac{-\sik(f')-\deg(f')}{k}-(k-1)\de=-\si(f)-\de,
\end{eqnarray*}
as wanted. Thus $\bOm$ is $\si$-weighted.

Next we prove that $\bOm$ is in $\cHp$.
As noted above, the outer edges of $H$ are 1-way in $\bOm$ (hence they form a clockwise directed cycle), hence it remains to prove that $\bOm$ is minimal and accessible from outer vertices. 
For an edge $e=(u,v)$ of $H$ we consider the subgraph $G_e$ of $\Hk$ made of the path subdividing $e$ together with the sea-path associated with $e$. In the hyperorientation $\Omk$ of $G_e$ the sea-path cannot be used to go from $u$ to $v$ nor from $v$ to $u$ because of Claim~\ref{claim:config-sea-path} (more precisely, the statement $w'(e_1)=0$ in this claim). Moreover, the path subdividing $e$ is oriented from $u$ to $v$ in $\Omk$ if and only if $e$ is oriented 1-way from $u$ to $v$ in the hyperorientation $\bOm$; see Figure~\ref{fig:reduction-sea-path2}. Thus for any vertices $v_1,v_2$ of $H$, there is a directed path from $v_1$ to $v_2$ in the hyperorientation $\Omk$ of $\Hk$ if and only if there is a directed path from $v_1$ to $v_2$ in the hyperorientation $\bOm$ of $H$. Since $\Omk$ is in $\cHp$, we conclude that in the hyperorientation $\bOm$ of $H$ every vertex is accessible from every outer vertex. Moreover any simple directed cycle in the hyperorientation $\bOm$ of $H$ corresponds to a directed simple cycle in the hyperorientation $\Omk$ of $\Hk$. Hence the minimality of the hyperorientation $\Omk$ implies the minimality of $\bOm$. Thus $\bOm$ is in $\cHp$.

Lastly we prove that there does not exist a $\si$-weighted hyperorientation $\bOm'\in\cHp$ of $H$ distinct from $\bOm$. 
Suppose the contrary. By inverting the construction represented in Figure~\ref{fig:reduction-sea-path2} (using the fact that $\bw(e)\geq -k+1$ by our choice of $k$), one can associate with $\bOm'$ a hyperorientation $\Om'\neq \Omk$ of $\Hk$ satisfying the properties described in Claim~\ref{claim:config-sea-path}. 
It is then easy to see using the same relations as above that $\Om'$ is $\sik$-weighted. 
Moreover, by the properties highlighted in the previous paragraph, it is easily seen that $\Om'$ is minimal and accessible from every outer vertex. Thus we obtain a $\sik$-weighted hyperorientation $\Om'\neq \Omk$ in $\cHp$. This is impossible because this contradicts the uniqueness property of Proposition~\ref{prop:case-degd}.
\end{proof}

Claim~\ref{lem:exists-kde-orient} proves that if a charge function $\si$ fits a light-rooted hypermap $H$, then $H$ admits a unique $\si$-weighted hyperorientation in $\cHp$. This, together with Lemma~\ref{lem:girth-necessary}, completes the proof of Theorem~\ref{theo:shifted-orientation-light}.

\subsection{Proof of Theorems~\ref{theo:shifted-orientation-dark} and~\ref{theo:shifted-orientation-0}}\label{sec:endproof-outer-dark}
In this section we prove Theorems~\ref{theo:shifted-orientation-dark} and~\ref{theo:shifted-orientation-0} by a reduction to Theorem~\ref{theo:shifted-orientation-light}. 

We start with Theorem~\ref{theo:shifted-orientation-dark}. Let $H$ be a dark-rooted hypermap with a simple outer face, and let $\si$ be a charge function fitting $H$. 
We want to prove that $H$ admits a unique $\si$-weighted hyperorientation in $\cHm$. 
Let $H'$ be the light-rooted hypermap obtained from $H$ by adding a dark face of degree 2 along each of the outer edges of $H$, and changing the outer face color into light, as indicated in Figure~\ref{fig:reduction-outer-dark4}. Observe that the outer faces $f_0$ of $H$ and $f_0'$ of $H'$ have the same degree. 
We call \emph{outer digons} of $H'$ the added dark faces. We define a charge function $\si'$ of $H'$ by setting $\si'(f_0')=\deg(f_0')$, $\si'(f)=-3$ if $f$ is an outer digon, $\si'(v)=1$ if $v$ is an outer vertex, and $\si'(a)=\si(a)$ if $a$ is any inner vertex or inner face of $H$.

\fig{width=.9\linewidth}{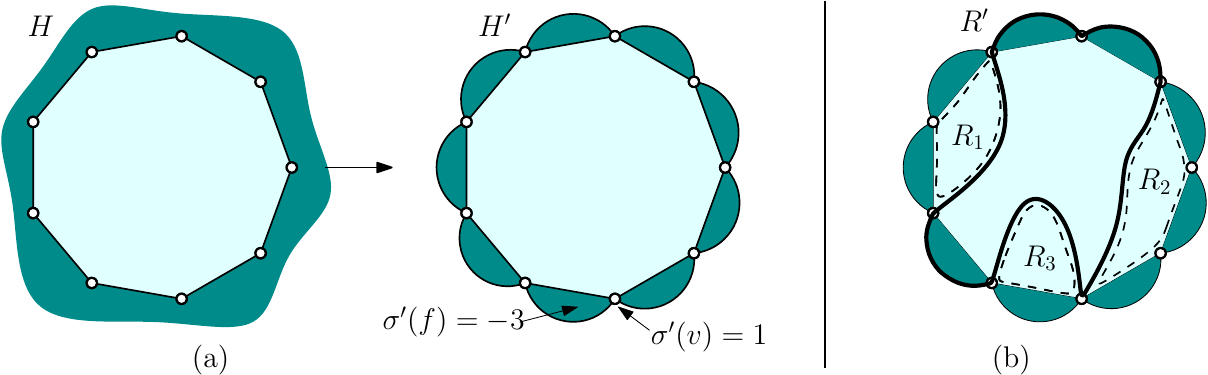}{(a) The hypermap $H'$ obtained from $H$ by adding a dark face of degree 2 along each outer edge. (b) The contour of a simply connected light region $R$ of $H'$ containing the outer face $f_0'$ but not all outer vertices. Here $\deg(f_0)=9$, $b=6$, $d=6$ and $k=3$}

\begin{claim}
The charge function $\si'$ fits $H'$.
\end{claim}
\begin{proof}
First observe that the charge $\si'(v)$ of any vertex $v$ is positive, and $\si'(f_0')=\deg(f_0')$. Moreover, 
$$\si'_\tot=\si_\tot-\si(f_0)+\si'(f_0')+\sum_{v\textrm{ outer vertex}}\si'(v)+\sum_{f\textrm{ outer digon}}\si'(f)=0$$
because $\si_\tot=0$, $-\si(f_0)=\si'(f_0')=\deg(f_0)$ and there are $\deg(f_0)$ outer vertices and outer digons.

We now prove that $H'$ satisfies the $\si'$-girth condition.
Let $R$ be a light region of $H'$. We want to prove $|\R|\geq \si'(R)$ with strict inequality if an outer edge is strictly contained in $R$. By Lemma~\ref{lem:simply-connected-sufficient} we can assume that $R$ is simply connected. 
If $f_0'\notin R$, then none of the outer digons is in $R$. Hence in this case $R$ is a light region of $H$, and $|\R|\geq \si(R)=\si'(R)$ as wanted. We now assume that $f_0'\in R$. First suppose that every outer vertex is strictly inside $R$. In this case, all the outer digons are in $R$, and we consider the light region $R$ of $H$ obtained from $R$ by replacing the outer face $f_0'$ and the outer digons by $f_0$. We get $|\R|>\si(R)=\si'(R)$ as wanted. 
Now assume that $f_0'\in R$ but $b>0$ outer vertices are incident to $\R$ (so that $\deg(f_0)-b$ outer-vertices are strictly inside $R$). Let $d$ be the number of outer digons in $R$. By deleting from $R$ the outer face $f_0'$ and the $d$ outer digons in $R$, one get a light region of $H$ which decomposes as a disjoint union of $k$ (non-empty) simply connected light regions $R_1,R_2,\ldots,R_k$; see Figure~\ref{fig:reduction-outer-dark4}(b). The number $k$ is determined by $k+(\deg(f_0)-d)=b$. Indeed the contour $D$ of the outer face of $f_0$ of $H$ decomposes into $b$ paths (joining consecutive vertices incident to $\R$) which are either edges of one of the $\deg(f_0)-d$ digons not in $R$, or part of the boundary of one of the light regions $R_1,R_2,\ldots,R_k$ (recall that $R$ is simply connected so that the light regions $R_1,\ldots, R_k$ corresponding to different paths of $D$ are distinct).
We have 
$$|\R|=\sum_{i=1}^k|\R_i|+\deg(f_0)-2d,$$
and
$$\si'(R)=\sum_{i=1}^k\si(R_i)+\si'(f_0')-3d+\deg(f_0)-b=\sum_{i=1}^k\si(R_i)+2\deg(f_0)-3d-b,$$
hence
$$|\R|-\si'(R)=\bigg(\sum_{i=1}^k|\R_i|-\si(R_i)\bigg)+b+d-\deg(f_0)=\bigg(\sum_{i=1}^k|\R_i|-\si(R_i)\bigg)+k\geq k.$$
Thus, $|\R|\geq \si'(R)$ and if one of the outer edges is strictly inside $R$, then $k\geq 1$ and $|\R|>\si'(R)$.
\end{proof}

Since $\si'$ fits $H'$, Theorem~\ref{theo:shifted-orientation-light} ensures that $H'$ has a unique $\si'$-weighted hyperorientation $\Om'$ in $\cHp$. Let $\Om$ be the hyperorientation of $H$ such that the weights and orientations of the inner edges of $H$ are the same as in $\Om'$, and the outer edges of $H$ form a counterclockwise directed cycle of 1-way edges of weight 1.

\begin{claim}\label{claim:reduction-works}
The hyperorientation $\Om$ is the unique $\si$-weighted hyperorientation of $H$ in $\cHm$. 
\end{claim}

\begin{proof}
Let $w'(a)$ be the weight of a vertex, edge or face in $\Om'$. Let $D_1,\ldots,D_{\deg(f_0)}$ be the outer digons of $H'$ in clockwise order, and let $e_i,e_i'$ be the outer and inner edges incident to $D_i$ respectively. We will first prove that $w'(e_i)=1$ and $w'(e_i')=0$ for all $i\in\{1,\ldots,\deg(f_0)\}$.
First note that for all $i\in\{1,\ldots,\deg(f_0)\}$, the weight condition on the outer digon $D_i$ gives $w'(e_i)+w'(e_i')=w'(D_i)=-\si'(D_i)-2=1$. Moreover, since the weight of every outer vertex $u$ is $w'(u)=\si'(u)=1$, we get $w'(e_i)\leq 1$ and $w'(e_i')\leq 1$. Hence $w'(e_i)\geq 0$ and $w'(e_i')\geq 0$, and $w'(e_{i-1})+w'(e_{i}')\leq 1$ for all $i\in\{1,\ldots,\deg(f_0)\}$ with the convention that $e_{0}=e_{\deg(f_0)}$. Hence $w'(e_{i-1})\leq w'(e_i)$ for all $i\in\{1,\ldots,\deg(f_0)\}$. Thus $w'(e_{i-1})=w'(e_{i})$ and $w'(e_{i-1}')=w'(e_{i}')$ for all $i\in\{1,\ldots,\deg(f_0)\}$. 
Moreover, the hyperorientation $\Om'$ has no counterclockwise directed cycle (since $\Om'\in \cHp$), hence $w'(e_i')=0$ for all $i\in\{1,\ldots,\deg(f_0)\}$, and $w'(e_i)=1$. 

Since $w'(e_i)=1$ and $w'(e_i')=0$ for all $i\in\{1,\ldots,\deg(f_0)\}$, the weight of any vertex, or face of $H$ is the same in $\Om$ as in $\Om'$. Moreover, the weight of every outer vertex and outer edge of $H$ in $\Om$ is 1. Thus $\Om$ is $\si$-weighted. Moreover, because the hyperorientation $\Om'$ is minimal and accessible from the outer vertices, the hyperorientation $\Om$ is also minimal and accessible from the outer vertices. Thus $\Om$ is in $\cHm$. 

Lastly, suppose there is another $\si$-weighted hyperorientation $\wt{\Om}\neq \Om$ of $H$ in $\cHm$. We then consider the hyperorientation $\wt{\Om}'$ of $H'$ defined as follows: the weight of the inner edges of $H$ in $\wt{\Om}'$ are the same as in $\wt{\Om}'$, while the weight of the edges $e_i,e_i'$ of the outer digon in $\wt{\Om}'$ are $\wt{w}'(e_i)=1$ and $\wt{w}'(e_i')=0$ for all $i\in\{1,\ldots,\deg(f_0)\}$. It is easily seen that $\wt{\Om}'$ is a $\si$-weighted hyperorientation of $H'$ distinct from $\Om'$. Moreover $\wt{\Om}'$ is in $\cHp$ (it is minimal, accessible from the outer vertices and the outer face of $H'$ is a clockwise directed cycle). This contradicts the uniqueness of $\Om'$ given by Theorem~\ref{theo:shifted-orientation-light}.
\end{proof}

Claim~\ref{claim:reduction-works} ensures that any dark-rooted charged hypermap $(H,\si)$ satisfying the conditions of Theorem~\ref{theo:shifted-orientation-dark} admits a unique $\si$-weighted hyperorientation in $\cHm$. This together with Lemma~\ref{lem:girth-necessary} completes the proof of Theorem~\ref{theo:shifted-orientation-dark}.\\

We now prove Theorem~\ref{theo:shifted-orientation-0} by a reduction to Theorem~\ref{theo:shifted-orientation-dark}. Let $H$ be a vertex-rooted hypermap, and let $\si$ be a charge function fitting $H$. 
We want to prove that there exists a unique $\si$-weighted hyperorientation of $H$ in $\cH_0$.
Let $v_0$ be the root-vertex of $H$, and let $f_0$ be a light face incident to $v_0$. Let $H'$ be the dark-rooted hypermap (with outer degree $1$) obtained from $H$ by adding a loop edge $e_0$ incident to $v_0$ inside $f_0$ as indicated in Figure~\ref{fig:reduction-outer-0}. The face of degree 1 incident to $e_0$ thus created (which is dark) is taken as the outer face of $H'$, and is denoted by $f_1$. The light face of $H'$ incident to $e_0$ is denoted by $f_2$. We define a charge function $\si'$ of $H'$ by setting $\si'(f_1)=-1$, $\si'(f_2)=\si(f_0)+1$ and $\si'(a)=\si(a)$ for any other face or vertex of $H'$.

\fig{width=.5\linewidth}{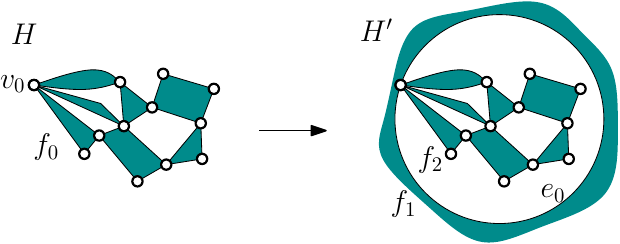}{The dark-rooted hypermap $H'$ obtained from $H$ by adding a loop edge $e_0$.}

\begin{claim}
The charge function $\si'$ fits the dark-rooted hypermap $H'$.
\end{claim}
 
\begin{proof}
Since $\si$ fits $H$, the charge $\si'(v)$ of every inner vertex $v$ of $H'$ is positive. Moreover the charge of the outer vertex $v_0$ is $\si'(v_0)=\si(v_0)=0$, and the charge of the dark outer face $f_1$ is $\si'(f_1)=-1=-\deg(f_1)$. Furthermore, 
$$\si_\tot'=\si_\tot-\si(f_0)+\si'(f_1)+\si'(f_2)=0.$$ 
It remains to prove that $H'$ satisfies the $\si'$-girth condition. Let $R'$ be a light region of $H'$. First suppose that $f_2\notin R'$. In this case $f_1\notin R'$, hence $R'$ is a light region of $H$ and $|\R'|\geq\si(R')=\si'(R')$ as wanted. Next suppose that both $f_1$ and $f_2$ are in $R'$. In this case, we consider the light region $R$ of $H$ obtained from $R'$ by replacing $f_1$ and $f_2$ by $f_0$. Since $\R=\R'$, we get 
$|\R'|=|\R|\geq \si(R)=\si'(R')$ with strict inequality if $v_0$ is strictly inside $R'$. 
Lastly, suppose that $f_2\in R'$ and $f_1\notin R'$. We consider the light region $R$ of $H$ obtained from $R'$ by replacing $f_2$ by $f_0$. Note that $e_0\in \R'$, and $\R=\R'\setminus \{e_0\}$. Thus $|\R'|=|\R|+1\geq \si(R)+1=\si'(R')$, as wanted.
\end{proof}

Since $\si'$ fits $H'$, Theorem~\ref{theo:shifted-orientation-dark} implies that $H'$ has a unique $\si'$-weighted hyperorientation $\Om'$ in $\cHm$. Let $\Om$ be the the restriction to $H$ of the hyperorientation $\Om'$.

\begin{claim}\label{claim:reduction-works-0}
The hyperorientation $\Om$ is the unique $\si$-weighted hyperorientation of $H$ in $\cH_0$. 
\end{claim}

\begin{proof}
By definition, the weight of $v_0$ and $e_0$ in $\Om'$ is 1. 
Hence the weight of $v_0$ in $\Om$ is 0 and the weight $w(f_0)$ of $f_0$ in $\Om$ is the same as the weight $w'(f_2)$ of $f_2$ in $\Om'$. 
Hence $w(f_0)=w'(f_2)=\si'(f_2)-\deg(f_2)=\si(f_0)-\deg(f_0)$. Hence the hyperorientation $\Om$ is $\si$-weighted. Moreover because the hyperorientation $\Om'$ is minimal and accessible from $v_0$, the hyperorientation $\Om$ is also minimal and accessible from $v_0$. Thus $\Om$ is in $\cH_0$.

Conversely, suppose that there is another $\si$-weighted hyperorientation $\wt{\Om}\neq \Om$ of $H$ in $\cH_0$. We then consider the hyperorientation $\wt{\Om}'$ of $H'$ defined as follows: the weight of $e_0$ is 1 and the weight of the other edges is as in $\wt{\Om}$. It is easily seen that $\wt{\Om}'$ is a $\si$-weighted hyperorientation of $H'$ distinct from $\Om'$. Moreover $\wt{\Om}'$ is in $\cHm$. This contradicts the uniqueness of $\Om'$ given by Theorem~\ref{theo:shifted-orientation-dark}.
\end{proof}

Claim~\ref{claim:reduction-works-0} shows that if a charge function $\si$ fits a vertex-rooted map $H$, then $H$ admits a unique $\si$-weighted hyperorientation in $\cH_0$. This together with Lemma~\ref{lem:girth-necessary} completes the proof of Theorem~\ref{theo:shifted-orientation-0}.

\vspace{.2cm}

\noindent{\bf Acknowledgments.} We thank the anonymous referees for their many useful comments and suggestions.

\bibliographystyle{plain}
\bibliography{biblio}

\end{document}

%% file: defs.tex
\usepackage{amssymb,amsmath}
\usepackage{stmaryrd,paralist}
\usepackage{bbm}
\usepackage{graphics,graphicx}

\usepackage[latin1]{inputenc}
\usepackage{subfigure}

\graphicspath{{./Figures/}}

\newtheorem{theo}{Theorem}
\newtheorem{thm}[theo]{Theorem}
\newtheorem{lem}[theo]{Lemma}
\newtheorem{lemma}[theo]{Lemma}
\newtheorem{prop}[theo]{Proposition}
\newtheorem{claim}[theo]{Claim}
\newtheorem{Def}[theo]{Definition}
\newtheorem{cor}[theo]{Corollary}
\theoremstyle{remark}
\newtheorem{remark}[theo]{Remark}

{\medbreak}

\newcommand{\ds}{\displaystyle}
\def\ni{\noindent}

\newcommand{\cacher}[1]{}

\def\cO{\mathcal{O}}

\def\cOp{\cO_+}
\def\cOm{\cO_-}
\def\cH{\mathcal{H}}
\def\cHz{\cH_0}
\def\cHp{\cH_+}
\def\cHm{\cH_-}
\def\cJ{\mathcal{J}}
\def\cJz{\cJ_0}
\def\cJp{\cJ_+}
\def\cJm{\cJ_-}
\def\cT{\mathcal{T}}
\def\cTz{\cT_0}
\def\cTp{\cT_+}
\def\cTm{\cT_-}

\def\bz{\mathbb{Z}}

\def\br{\mathbb{R}}
\def\rp{\mathbb{R}^+}

\def\cB{\mathcal{B}}

 \def\cRad{\vec{\cB}_{d,e}}
 \def\cRadb{\vec{\cB}^{\blacklozenge k}_{d,e}}
 \def\cRadw{\vec{\cB}^{\lozenge k}_{d,e}}
 
 \def\Radbk{B^{\blacklozenge k}_{d,e}}
 \def\Radwk{B^{\lozenge k}_{d,e}}
 
 \def\cC{\mathcal{C}}

\def\cSad{\vec{\cC}_{d,e}}
 \def\cSadb{\vec{\cC}^{\blacklozenge k}_{d,e}}
 \def\cSadw{\vec{\cC}^{\lozenge k}_{d,e}}
 
 \def\Sadbk{C^{\blacklozenge k}_{d,e}}
 \def\Sadwk{C^{\lozenge k}_{d,e}}
 
\def\cA{\mathcal{A}} 
\def\vcA{\vec{\cA}} 
\def\cAad{\vcA_{d,e}}
\def\cAadb{\vcA_{d,e}^\bullet}
\def\cAbw{\vcA^{\blacklozenge k,\lozenge\ell}_{d,e}}
\def\cAbb{\vcA^{\blacklozenge k,\blacklozenge\ell}_{d,e}}
\def\cAwb{\vcA^{\lozenge k,\blacklozenge\ell}_{d,e}}
\def\cAww{\vcA^{\lozenge k,\lozenge\ell}_{d,e}}
\def\Abw{A^{\blacklozenge k,\lozenge\ell}_{d,e}}
\def\Abb{A^{\blacklozenge k,\blacklozenge\ell}_{d,e}}
\def\Awb{A^{\lozenge k,\blacklozenge\ell}_{d,e}}
\def\Aww{A^{\lozenge k,\lozenge\ell}_{d,e}}

\def\cA{\mathcal{A}}

\def\cF{\mathcal{F}}
\def\cG{\mathcal{G}}
\def\cW{\mathcal{W}}
\def\cL{\mathcal{L}}

\newcommand{\Hk}{H^{(k)}}
\newcommand{\sik}{\sigma^{(k)}}
\newcommand{\Omk}{\Omega^{(k)}}

\newcommand{\bw}{\overline{w}}
\newcommand{\bOm}{\overline{\Om}}
\newcommand{\wOm}{\widetilde{\Om}}

\def\igirth{ingirth}

\newcommand{\tot}{\textrm{total}}

\newcommand{\al}{\alpha}
\newcommand{\be}{\beta}

\newcommand{\de}{\delta}

\newcommand{\si}{\sigma}
\newcommand{\eps}{\epsilon}
\renewcommand{\th}{\theta}

\newcommand{\Om}{\Omega}
\newcommand{\Ga}{\Gamma}

\newcommand{\ov}[1]{\overline{#1}}
\newcommand{\wt}[1]{\widetilde{#1}}

\newcommand{\hT}{\widehat{T}}
\newcommand{\bG}{G'}
\newcommand{\ba}{{\bar{a}}}

\newcommand{\Id}{\textrm{Id}}

\newcommand{\fig}[3]{\begin{figure}[h!]\begin{center}\includegraphics[#1]{#2.pdf}\end{center}\caption{#3}\label{fig:#2}\end{figure}}

\newcommand{\xx}{\textbf{x}}
\newcommand{\yy}{\textbf{y}}

\newcommand{\R}{\partial R}

\newcommand{\bR}{\overline{R}}